\documentclass[10pt, oneside]{article}

\usepackage{amsfonts}
\usepackage{amssymb} 
\usepackage{amsmath,amsfonts,amssymb,amsthm}
\usepackage{caption}
\usepackage{subfigure}
\usepackage{cite}
\usepackage{xcolor}
\usepackage{graphicx}
\usepackage{subfigure}
\usepackage{float}
\usepackage{paralist}
\usepackage{indentfirst}
\usepackage{authblk}
\usepackage[colorlinks,
linkcolor=red,
citecolor=blue
]{hyperref}

\usepackage[utf8]{inputenc}  
\usepackage[T1]{fontenc}
\usepackage{lmodern}         
\input{glyphtounicode}

\usepackage[T1]{fontenc}
\def\div{ \hbox{\rm div}\,  }

\newtheorem{theorem}{Theorem}[section]
\newtheorem{lemma}{Lemma}[section]
\newtheorem{proposition}{Proposition}[section]

\newtheorem{remark}{Remark}[section]

\newtheorem{defn}{Definition}[section]

\def\var{\varepsilon}

\def\bma#1\ema{{\allowdisplaybreaks\begin{aligned}#1\end{aligned}}}

\numberwithin{equation}{section}

\usepackage[title,titletoc]{appendix}

\usepackage{titlesec}

\begin{document}

\allowdisplaybreaks

\title{{\textbf{
High-capillarity limit and smoothing effect of large solutions for a multi-dimensional generic non-conservative compressible two-fluid model}}}

\author{Ling-Yun Shou, Jiayan Wu, Lei Yao, and Yinghui Zhang}







\date{}


\maketitle

\begin{abstract} 
We investigate the global existence and long-time behavior of large solutions, in the high-capillarity regime, for a general multidimensional non-conservative compressible two-fluid model with the capillary pressure relation \(f(\alpha^{-}\rho^{-})=P^{+}-P^{-}\).
Our main contributions are threefold.
First, for sufficiently large capillarity coefficients, we prove the existence and uniqueness of global solutions in critical Besov spaces for large initial perturbations, under the sharp stability condition
\(-\frac{s_{-}^{2}(1,1)}{\alpha^{-}(1,1)}<f^{\prime}(1)<0\),
thereby removing the additional negativity restriction assumed by Evje--Wang--Wen [Arch. Ration. Mech. Anal. 221:1285--1316, 2016].
Second, we  give a rigorous justification of the global-in-time convergence to the incompressible Navier-Stokes flows and obtain explicit convergence rates in critical spaces for ill-prepared data.  
Third, if additionally the initial perturbation lies in a lower-regularity Besov space, we derive optimal decay rates for the solution and its derivatives of any order, revealing a long-term smoothing effect.
To the best of our knowledge, this work provides the first global large-amplitude well-posedness theory for multidimensional compressible two-fluid flows. Our key ingredient is to capture the nonlinear coupling between capillarity and viscosity, through which the parabolic structure dominates the underlying hyperbolic dynamics and reveals the coexistence of dispersive (two-phase Gross–Pitaevskii) and parabolic regularization mechanisms.
\end{abstract}

\renewcommand{\thefootnote}{}

\footnotemark

\maketitle
\section{Introduction}\label{section1}
\subsection{Background}
Multiphase fluids are ubiquitous in nature and  also in various industrial applications, such as nuclear power, chemical processing, oil and gas manufacturing.  In this paper, we investigate the mathematical theory of the following generic non-conservative compressible two-fluid system with capillarity effects in $\mathbb{R}^{d}$ ($d\geq2$):
\begin{equation}\label{system1}
 \left\{
   \begin{array}{ll}
     \partial_t (\alpha^{\pm}\rho^{\pm}) +\text{div}(\alpha^{\pm}\rho^{\pm}u^{\pm})=0,\\
     \partial_t (\alpha^{\pm}\rho^{\pm}u^{\pm}) +\text{div}(\alpha^{\pm}\rho^{\pm}u^{\pm}\otimes u^{\pm})+\alpha^{\pm}\nabla P^{\pm}(\rho^{\pm})=\text{div}(\alpha^{\pm}\tau^{\pm}+ \mathcal{K}^\pm),\\
P^{+}(\rho^{+})-P^{-}(\rho^{-})=f(\alpha^-\rho^{-}).   \end{array}
 \right.
\end{equation}
Here $0\leq \alpha^{\pm}\leq 1 $ are the volume fractions of the $+$ and  $-$ phases with $\alpha^++\alpha^-=1$. The quantities $\rho^{\pm}(x,t)\geq 0$, $u^{\pm}(x,t)$ and $P^{\pm}(\rho^{\pm})$
denote the densities, velocities, and pressure functions of the two phases, respectively. 
The viscous stress tensors are
\begin{align} 
    \tau^{\pm}:=2\mu^{\pm}(\rho^\pm)\mathbb{D}u^\pm+\lambda^{\pm}(\rho^\pm)\text{div}u^{\pm} {\rm Id},\label{system2}
\end{align}
where $\mathbb{D}u^\pm=\frac{\nabla u^\pm+(\nabla u^\pm)^{\top}}{2}$ the deformation tensor, and the viscosity coefficients $\mu(\rho^\pm), \lambda(\rho^\pm)$ satisfy
\begin{align}\label{viscosity}
    \mu^\pm(s), \lambda^\pm(s)\in C^{\infty}(\mathbb{R}_+),\quad \mu(\rho^\pm)>0,\quad 2\mu(\rho^\pm)+\lambda(\rho^\pm)>0.
\end{align}
The Korteweg (capillarity) tensor is given by
\begin{equation}
\begin{aligned}
\mathcal{K}^\pm:=\kappa \alpha^{\pm}\rho^{\pm}\,\mathrm{div}\big(m(\alpha^{\pm}\rho^{\pm})\nabla (\alpha^{\pm}\rho^{\pm})\big) \mathrm{Id}&+\frac{1}{2}\kappa^\pm\big(m(\alpha^{\pm}\rho^{\pm})-\alpha^{\pm}\rho^{\pm} m'(\alpha^{\pm}\rho^{\pm})\big)\lvert\nabla (\alpha^{\pm}\rho^{\pm})\lvert^{2}\mathrm{Id}\\
&\quad-\kappa^\pm m(\alpha^{\pm}\rho^{\pm})\nabla (\alpha^{\pm}\rho^{\pm})\otimes \nabla(\alpha^{\pm}\rho^{\pm}).
\end{aligned}
\end{equation}
Noticing that for smooth enough $\alpha^\pm \rho^\pm$ and $m(s)$, one has
\begin{align}
\text{div} \mathcal{K}^\pm:=\kappa  \alpha^{\pm}\rho^{\pm}\nabla\left(m(\alpha^{\pm}\rho^{\pm})\Delta(\alpha^{\pm}\rho^{\pm})+\frac{1}{2}m'(\alpha^{\pm}\rho^{\pm}) |\nabla(\alpha^{\pm}\rho^{\pm})|^2\right),
\end{align}
where $\kappa>0$ is the capillarity coefficient. The pressure functions $P^{\pm}$ are assumed to satisfy the general conditions
\begin{align}
P^{\pm}(s)\in C^{\infty}(\mathbb{R}_{+}),\quad \partial_{s} P^{\pm}(s)>0\quad\text{for}\quad s>0.\label{system3}
\end{align}
The function
  $f$, called the capillary pressure, belongs to $C^3([0,
\infty))$ and is  strictly decreasing  near the
equilibrium. For background on this model, we refer to \cite{Bearbook, Brennenbook,BDGG2010, BHL2012, EvjeWangWang2016,Ishiibook,Prospererttibook}   and the references
therein; see in particular the overview in the introduction of  \cite{BDGG2010}.  We consider the Cauchy problem for the two-fluid system \eqref{system1} supplemented with the initial data
\begin{align}\label{d}
	&(\alpha^{\pm}, \rho^\pm, u^{\pm})(x,0)=(\alpha^{\pm}_{0}, \rho^{\pm}_{0}, u^{\pm}_{0})(x)\rightarrow (\bar{\alpha}^\pm,\bar{\rho}^\pm,0),\quad |x|\rightarrow \infty,
\end{align}
with the given constants $\bar{\alpha}^\pm, \bar{\rho}^\pm>0$.

From the analytical viewpoint, the two-fluid system presents several difficulties, including:
\begin{itemize}
\item \textbf{Partial dissipation.} The mass conservation equations lack dissipation, whereas the momentum equations contain viscous dissipation.

\item \textbf{Degeneracy to single-phase regions.} Transitions to regions where $\alpha^{+} \rho^{+}$ or $\alpha^{-} \rho^{-}$ vanishes may occur when either the volume fractions $\alpha^{\pm}$ or the densities
$\rho^{\pm}$ become zero.

\item \textbf{Non-conservative structure.} The non-conservative terms $\alpha^{\pm} \nabla P^{\pm}$ in the momentum equations prevent a straightforward adaptation of methods developed for single-phase models.
\end{itemize}

In the excellent work \cite{EvjeWangWang2016}, Evje-Wang-Wen proved the global existence and decay rates of small, strong solutions to \eqref{system1} without capillary effects ($\kappa^\pm=0$) under the assumption that
\begin{align} 
    -\frac{s_-^2(1,1)}{\alpha^-(1,1)}<f'(1)<\frac{\eta-s_-^2(1,1)}{\alpha^-(1,1)}<0,\label{stability0}
\end{align}
where $\eta>0$ is a small fixed constant, and $s_-$ denotes the sound speed of the fluid $-$.
It should be mentioned that the technical smallness of $\eta$ is essential for closing the higher–order energy estimates;  
see \cite{EvjeWangWang2016}, pp. 1302–1308 for details. Building on  \cite{EvjeWangWang2016}, Lai-Wen-Yao \cite{LWY2017} established the vanishing–capillarity limit for the non-conservative compressible two–fluid model \eqref{system1}.
In the opposite direction, Wu-Yao-Zhang \cite{WuYaoZhangin} proved instability when $f^{\prime}(1)>0$.
For the two–fluid model \eqref{system1} with equal pressure $P^+(\rho^+)=P'(\rho^-)$ (i.e., $f\equiv0$), 
Bresch-Desjardins-Ghidaglia-Grenier \cite{BDGG2010} proved the existence of global weak solutions
on the periodic domain for  $1<\gamma^\pm <6$. Later, Cui–Wang–Yao–Zhu \cite{CWYZ2016} derived time–decay rates of strong solutions for the model \eqref{system1} with the special density–dependent viscosities and equal coefficients
\begin{align}\label{1.6}
   \mu^\pm(\rho^{\pm})=\nu\rho^\pm\quad\text{with}\quad \nu>0, \quad \lambda^{\pm}(\rho^{\pm})=0,\quad \kappa^\pm=\kappa>0.
\end{align}
More recently, Li–Wang–Wu–Zhang \cite{LWWZ2023}  extended the result of \cite{CWYZ2016}  to the case of general constant viscosities, as in \eqref{system2}.
As for the three-dimensional 
two-fluid model \eqref{system1} with equal pressures ($f\equiv0$) and without capillary effects ($\kappa^\pm=0$), 
Bresch-Huang-Li  \cite{BHL2012} showed the existence
of global weak solutions in one dimension, when the capillary effects are absent ($\kappa=0$) and
 $\gamma^\pm >1$.
Subsequently, Wu-Yao-Zhang \cite{WYZMathAnn} obtained global stability for the 3D Cauchy problem with sufficiently small initial data in
$L^1\cap H^3$, and  
very recently,  \cite{TWYZ2025} established   the vanishing capillary limit based on that result. Shou-Wu-Yao-Zhang \cite{SWYZ} removed the $L^1$ condition in  \cite{WYZMathAnn} and established the global existence and large-time behavior of solutions for data close to equilibrium in the low-regularity space $\dot{B}^{-1/2}_{2,1}\cap\dot{B}^{5/2}_{2,1}$.  In conclusion, all existing results for \eqref{system1} heavily rely on the smallness of either the initial data or the parameter ${\eta}$ in the condition \eqref{stability0}.

On the other hand, a natural problem is to identify the minimal regularity assumption of initial data that ensures the global well-posedness of solutions for \eqref{system1}. Observe that the system \eqref{system1} remains  invariant for all $l>0$ by the transformation
\begin{align*}
\alpha^\pm (t,x)\rightsquigarrow \alpha^\pm(l^{2}t,lx),\quad  \rho^\pm(t,x)\rightsquigarrow \rho^\pm(l^{2}t,lx) \ \ \hbox{and} \ \ u^\pm(t,x)\rightsquigarrow l u^\pm(l^{2}t,lx)
\end{align*}
up to a change of the pressure term $P^\pm $ into $ l^{2}P^\pm$. Therefore, one can employ the so-called {\emph{critical}} functional spaces to study
\eqref{system1}, endowed with norms that enjoy scaling invariance. A natural example of the critical space for initial data reads
\begin{align*}
\alpha_0^\pm-\bar{\alpha}^\pm, \rho_0^\pm-\bar{\rho}^\pm\in \dot{B}^{\frac{d}{2}}_{2,1}\quad \text{and} \quad u_0^\pm \in \dot{B}^{\frac{d}{2}-1}_{2,1}.
\end{align*}
When studying global existence, it is also natural to additionally require $\alpha_0^\pm-\bar{\alpha}^\pm, \rho_0^\pm-\bar{\rho}^\pm\in \dot{B}^{\frac{d}{2}-1}_{2,1}$ due to the symmetric hyperbolic-parabolic structure of the system.

We are also interested in studying the effects of large capillarity coefficient on the regularity and decay properties of the solutions for the two-fluid model \eqref{system1}.  In Korteweg-type diffuse interface models, the capillarity coefficient sets the thickness of the interfacial layer: a larger coefficient corresponds to smoother and more diffuse interfaces (see, e.g., \cite{LT1998}). Moreover, in compressible single-fluid systems, capillarity is known to have a stabilizing and regularizing effect. In the large-capillarity regime, enhanced dissipation overwhelms acoustic propagation, so the long-time dynamics are effectively governed by a strongly diffusive mechanism; see \cite{AH2017,AHM2020,BDGS2008,BGL2019,CDX2021,CV-Song-2024} and references for the study of one-fluid systems. However, to the best of our knowledge, there are no results regarding how capillarity influences the dynamics of compressible two-fluid flows.

In this paper, we investigate the global stability, large-capillarity limit, and long-time behavior of the Cauchy problem for the multidimensional two-fluid system \eqref{system1} in the regime where $\kappa$ is suitably large. As far as we are aware, no previous work has addressed the mathematical theory of global well-posedness for this system without restrictive smallness requirements on both the initial data or without the structural assumption \eqref{stability0} with a suitably small $\eta>0$. This paper aims to improve the known results with large initial data and a sharp structural condition that is more general than  \eqref{stability0}. Our main contributions are threefold:
\begin{itemize}
\item[(1)] \textbf{Global large solutions at critical regularity.} For sufficiently large capillarity coefficients, we prove the global existence and uniqueness of large solutions with critical regularity under the stability condition 
$$
-\frac{s_{-}^{2}(1,1)}{\alpha^{-}(1,1)}<f^{\prime}(1)<0,
$$
which stands in sharp contrast to the case $f'(1)>0$, where the system becomes unstable \cite{WuYaoZhangin}, and to the borderline case $f'(1)=0$ where the densities exhibit only slower decay \cite{CWYZ2016}.

\item[(2)] \textbf{Large–capillarity limit for ill–prepared data.} We justify the global–in–time singular limit of the two-fluid system \eqref{system1} to the incompressible Navier–Stokes equations
\begin{equation}\label{INS}
\left\{
\begin{aligned}
&\partial_t  v^\pm +v^\pm\cdot\nabla v^{\pm}-   \Delta  v^\pm +\nabla \pi^\pm= 0 , \\
&\operatorname{div}v^\pm=0, \\
&v^\pm(0,x)=v_0^\pm(x),
\end{aligned}
\right.
\end{equation}
as $\kappa\rightarrow \infty$, and provide explicit rates of convergence  in the critical space for ill–prepared data.
\item[(3)] \textbf{Optimal decay of arbitrary–order derivatives.} Under the additional assumption that the initial perturbation lies in a lower–regularity Besov space, we derive optimal decay rates for the solution and for derivatives of any order, which reveals the smoothing effect of the two-fluid system terms even at large times for only critical data. 
\end{itemize}


\vspace{2mm}

This paper is organized as follows. In the rest of Section \ref{section1}, we give a reformation of the system \eqref{system1},  introduce our main results and  then provide  our main strategies. In Section \ref{notations}, we briefly explain  the notations used and recall the Littlewood-Paley decomposition, Besov spaces and related analysis tools. In Section \ref{section2}, we prove global existence and uniqueness of large solutions in critical Besov spaces with sufficiently large capillarity coefficients. Section \ref{section3} is devoted to proving optimal decay rates for the solution and for derivatives of any order. Finally, in Appendix \ref{appendixA} we collect several useful tools in Besov spaces. We also provide proofs of the dispersive and Strichartz estimates in Appendix \ref{appendixB}, and establish decay and smoothness for the incompressible Navier–Stokes equations in Appendix \ref{appendixC}.

\subsection{Reformulation}

Here we reformulate the system \eqref{system1} as follows. From the pressure relation \eqref{system1}$_3$, we have
\begin{equation}\label{1.9}
\mathrm{d} P^+-\mathrm{d} P^-= \mathrm{d} f(\alpha^-\rho^-)
\end{equation}
and 
\begin{equation} \label{1.11}
\mathrm{d} P^+ = s_{+}^{2}\mathrm{d}\rho^+,\  \mathrm{d} P^- = s_{-}^{2}\mathrm{d}\rho^-,
\end{equation}
 where $s_{\pm}^{2}:=\frac{\mathrm{d} P^{\pm}}{\mathrm{d}
\rho^{\pm}}\left(\rho^{\pm}\right)=\bar{\gamma}^{\pm}
\frac{P^{\pm}\left(\rho^{\pm}\right)}{\rho^{\pm}}$ represents the  
sound speed of each phase, respectively. Following \cite{BDGG2010}, introduce the fractional densities
\begin{equation}\label{1.10}
R^{\pm}:=\alpha^{\pm} \rho^{\pm},
\end{equation}
which, together with $\alpha^++\alpha^-=1$, implies  
\begin{equation}\label{1.13}
\mathrm{d} \rho^{+}=\frac{1}{\alpha^{+}}\left(\mathrm{d}
R^{+}-\rho^{+} \mathrm{d} \alpha^{+}\right), \quad \mathrm{d}
\rho^{-}=\frac{1}{\alpha^{-}}\left(\mathrm{d} R^{-}+\rho^{-}
\mathrm{d} \alpha^{+}\right).
\end{equation}
Using \eqref{1.9} and \eqref{1.10}, we obtain
\begin{equation}\label{1.14}
\mathrm{d} \alpha^{+}=\frac{\alpha^{-} s_{+}^{2}}{\alpha^{-}
\rho^{+} s_{+}^{2}+\alpha^{+} \rho^{-} s_{-}^{2}} \mathrm{d}
R^{+}-\frac{\alpha^{+}\alpha^- }{\alpha^{-} \rho^{+}
s_{+}^{2}+\alpha^{+} \rho^{-}
s_{-}^{2}}\left(\frac{s_{-}^{2}}{\alpha^-}+f^{\prime}\right)
\mathrm{d} R^{-}. \end{equation}
 Substituting \eqref{1.14} into \eqref{1.13} yields
$$
\mathrm{d} \rho^{+}=\frac{\rho^{+} \rho^{-}
s_{-}^{2}}{R^{-}\left(\rho^{+}\right)^{2}
s_{+}^{2}+R^{+}\left(\rho^{-}\right)^{2} s_{-}^{2}}\left(\rho^{-}
\mathrm{d} R^{+}+\left(\rho^++\rho^+\frac{\alpha^-f^{\prime}}{s_-^2}\right)\mathrm{d} R^{-}\right),
$$
and
$$
\mathrm{d} \rho^{-}=\frac{\rho^{+} \rho^{-}
s_{+}^{2}}{R^{-}\left(\rho^{+}\right)^{2}
s_{+}^{2}+R^{+}\left(\rho^{-}\right)^{2} s_{-}^{2}}\left(\rho^{-}
\mathrm{d} R^{+}+\left(\rho^+-\rho^-\frac{\alpha^+f^{\prime}}{s_+^2}\right)\mathrm{d} R^{-}\right),
$$
which, together with \eqref{1.9}, gives the pressure differentials
$$
\mathrm{d} P^+=\mathcal{C}^2\left(\rho^{-} \mathrm{d}
R^{+}+\left(\rho^++\rho^+\frac{\alpha^-f^{\prime}}{s_-^2}\right) \mathrm{d} R^{-}\right),\quad \mathrm{d} P^-=\mathcal{C}^2\left(\rho^{-} \mathrm{d}
R^{+}+\left(\rho^+-\rho^-\frac{\alpha^+f^{\prime}}{s_+^2}\right) \mathrm{d} R^{-}\right), 
$$
where
$$
\mathcal{C}^2:=\frac{s_{-}^{2} s_{+}^{2}}{\alpha^{-} \rho^{+}
s_{+}^{2}+\alpha^{+} \rho^{-} s_{-}^{2}}.
$$

Next, since  $\alpha^++\alpha^-=1$, we have  
\begin{align} \label{1.15}
\frac{R^+}{\rho^+} +\frac{R^-}{\rho^-} =1\quad\text{and hence}\quad \rho^- \;=\;\frac{R^-\,\rho^+}{\rho^+ - R^+}\,. 
\end{align}
From the pressure relation \eqref{system1}$_3$, define
$$
\varphi(\rho^+) \;=\; P^+(\rho^+)\;-\;P^-\!\Bigl(\frac{R^-\,\rho^+}{\rho^+ - R^+}\Bigr)\;-\;f(R^-),
$$
and determine $\rho^+$ from $\varphi(\rho^+)=0.$

Fixing $R^\pm$ and  differentiating with respect to $\rho^+$ gives
$$
\varphi'(\rho^+)
=\;(P^+)'(\rho^+)
\;+\;(P^-)' \!\Bigl(\tfrac{R^-\,\rho^+}{\rho^+ - R^+}\Bigr)\;\frac{R^-\,R^+}{(\rho^+ - R^+)^2}
\;=\;
s_+^2
\;+\;
s_-^2\,\frac{R^-\,R^+}{(\rho^+ - R^+)^2},
$$
where $s_\pm^2=(P^\pm)'(\cdot)$ denotes the sound speed of each phase separately.  By construction $R^+<\rho^+<\infty$, it is clear that $\varphi'(\rho^+)>0$ on $(R^+,\infty)$ while $\varphi$ maps $(R^+,\infty)$ onto $(-\infty,\infty)$.  By the implicit function theorem, there is a unique solution
$$
\rho^+=\rho^+(R^+,R^-)\in (R^+,\infty).
$$
Finally, combining this with $\alpha^++\alpha^-=1$, \eqref{1.11}  and \eqref{1.15}, we set
$$
\rho^\pm=\rho_\pm(R^+,R^-), 
\quad
\alpha^\pm=\alpha_\pm(R^+,R^-),
\quad
\mathcal{C}=\mathcal{C}(R^+,R^-),
$$
see p. 614 of \cite{BDGG2010} for more details.

Our central objective is to determine whether, in the regime $\kappa \to \infty$, 
the system---where the two fluids may obey distinct pressure laws---admits global-in-time solutions for a class of  large initial data. 
Throughout this paper, we assume that the viscosity is given by \eqref{1.6}; 
in fact, our results remain valid for more general viscosity coefficients \eqref{viscosity} under suitable assumptions at equilibrium.



Therefore, we can rewrite  \eqref{system1}-\eqref{d} into the following equivalent form 
\begin{equation}\label{reformulat1}
\left\{
\begin{aligned}
&\partial_t R^\pm + \operatorname{div}(R^\pm u^\pm) = 0, \\
&\partial_t (R^+ u^+) + \operatorname{div}(R^+ u^+ \otimes u^+) + \alpha^+ \mathcal{C}^2 \left( \rho^-\nabla R^++\left(\rho^++\rho^+\frac{\alpha^-f^{\prime}}{s_-^2}\right)\nabla R^-\right)\\
&\quad\quad= \operatorname{div} \left( 2\nu R^+ \mathbb{D}u^+   \right) 
+ \kappa R^+ \nabla \left(  m(R^+) \Delta R^+ + \frac{1}{2}m'(R^+) |\nabla R^+|^2\right) , \\
&\partial_t (R^- u^-) + \operatorname{div}(R^- u^- \otimes u^-) + \alpha^- \mathcal{C}^2 \left( \rho^-\nabla R^++\left(\rho^+-\rho^-\frac{\alpha^+f^{\prime}}{s_+^2}\right)\nabla R^-\right) \\
&\quad\quad= \operatorname{div} \left( 2\nu R^- \mathbb{D}u^-   \right)
+ \kappa R^- \nabla \left(   m(R^-)\Delta R^- + \frac{1}{2}m'(R^-) |\nabla R^-|^2 \right),
\end{aligned}
\right.
\end{equation}
with the initial data
\begin{align}\label{reformulat2}
	&(R^{+}, R^{-}, u^{+},u^{-})(x,0)=(R^{+}_{0,\kappa}, R^{-}_{0,\kappa}, u^{+}_{0,\kappa},u^{-}_{0,\kappa})(x)\rightarrow (\bar{R}^+,\bar{R}^-,0,0),\quad |x|\rightarrow \infty,
\end{align}
where $\bar{R}^\pm:=\bar{\alpha}^\pm \bar{\rho}^\pm$. Without loss of generality, in what follows, we shall assume
$$
\bar{\alpha}^\pm \bar{\rho}^\pm=\partial_s P^\pm(1) =\nu = m(1) = 1.
$$


Denote the density fluctuations by
 $$
 n^{\pm}:=\sqrt{\kappa} L(R^\pm) \quad \mbox{with}\quad L(R^\pm):=\int_1^{R^\pm} \sqrt{\frac{m(\tau)}{\tau}}{\rm d}\tau.
 $$
If $n^\pm$ is uniformly bounded and $\kappa$ is sufficiently large such that $R^\pm$ is close to $1$, then we have   $\partial_{R^\pm} \big( \kappa^{-\frac{1}{2}} n^\pm\big)=\sqrt{\frac{m(R^\pm)}{R^\pm}}>0$ and therefore the inverse function theorem implies that 
$$
R^\pm=1+\mathcal{O}(1) \kappa^{-\frac{1}{2}} n^\pm.
$$
Formally, as $\kappa\rightarrow\infty$, $R^\pm$ is expected to converge to $1$, and the dynamics of \eqref{system-perturb1} are governed by the  incompressible Navier-Stokes equations \eqref{INS}.

We consider the system in terms of $(n^\pm,u^\pm)$. Note that  the fraction density $R^{\pm}$ is also given by $R^\pm=L^{-1}(\kappa^{-\frac12 }n^\pm)$. 
 Define 
 \begin{equation}
    \begin{aligned}
        Q(R^\pm) :=\frac{1}{R^\pm}-1,\qquad \phi(R^\pm):=R^\pm-1, \qquad \psi(R^{\pm}):=\sqrt{R^\pm m(R^\pm)}-1.
    \end{aligned}
\end{equation}
A simple calculation on \eqref{reformulat1} leads to the
following perturbation problem:
\begin{equation}\label{system-perturb1}
\left\{
\begin{aligned}
&\partial_t n^+  +v^+\cdot\nabla n^++\sqrt{\kappa }\operatorname{div} u^+ = F_1, \\
&\partial_t u^+ + v^+\cdot\nabla u^+  +\frac{\beta_1}{\sqrt{\kappa }} \nabla n^+ + \frac{\beta_2}{\sqrt{\kappa }} \nabla n^- - \Delta u^+ -\nabla\operatorname{div}u^+-\sqrt{\kappa }\nabla \Delta n^+ = F_2, \\
&\partial_t n^-+v^-\cdot\nabla n^- + \sqrt{\kappa}\operatorname{div} u^- = F_3, \\
&\partial_t u^- + v^-\cdot\nabla u^- + \frac{\beta_3 }{\sqrt{\kappa}}\nabla n^+ +  \frac{\beta_4}{\sqrt{\kappa }} \nabla n^- - \Delta u^--\nabla\operatorname{div}u^-  - \sqrt{\kappa} \nabla \Delta n^- = F_4,\\
&(n^\pm,u^\pm)(x,0)=(n_{0,\kappa}^\pm,u_{0,\kappa}^\pm)
\end{aligned}
\right.
\end{equation}
with $n_{0,\kappa}^\pm:=\sqrt{\kappa}L(R_{0,\kappa}^\pm)$,
\begin{equation*}
\left\{
\begin{aligned}
   \beta_1:&=\frac{\mathcal{C}^{2}(1,1)\rho^-(1,1)}{\rho^+(1,1)}, \qquad  \beta_2:=\mathcal{C}^{2}(1,1)+\frac{\mathcal{C}^{2}(1,1)\alpha^-(1,1)f'(1)}{s_-^2(1,1)},\\
 \beta_3:&=\mathcal{C}^{2}(1,1),\qquad\qquad\quad~~\beta_4:=\frac{\mathcal{C}^{2}(1,1)\rho^+(1,1)}{\rho^-(1,1)}-\frac{\mathcal{C}^{2}(1,1)\alpha^+(1,1)f'(1)}{s_+^2(1,1)}, 
\end{aligned}
\right.
\end{equation*}
and
\begin{equation}\label{F1F2F3F4}
    \left\{
\begin{aligned}
    F_1:=&-\sqrt{\kappa}\widetilde{\psi}(\kappa^{-\frac{1}{2}}n^+)\text{div}u^+-(u^+-v^+)\cdot\nabla n^+,\\
    F_2:=& - (u^+-v^+) \cdot \nabla u^++\nu(1+\widetilde{Q}( \kappa^{-\frac{1}{2}}n^+))\operatorname{div}\left( 2\widetilde{\phi}(\kappa^{-\frac{1}{2}}n^+)\mathbb{D}u^+\right)+\nu \widetilde{Q}( \kappa^{-\frac{1}{2}}n^+) (\Delta u^++\nabla\operatorname{div}u^+)\\
    & -\widetilde{g}_1^+(\kappa^{-\frac{1}{2}}n^+,\kappa^{-\frac{1}{2}}n^{-})\kappa^{-\frac{1}{2}} \nabla n^+ -\widetilde{g}_2^+(\kappa^{-\frac{1}{2}}n^+,\kappa^{-\frac{1}{2}}n^{-})\kappa^{-\frac{1}{2}}\nabla n^{-}\\
    &+\nabla\Big(\frac12|\nabla n^+|^2\Big)+ \sqrt{\kappa}\nabla\Big(\widetilde{\psi}(\kappa^{-\frac{1}{2}}n^+ )\Delta n^+\Big):=\sum_{l=1}^7F_2^l,\\
    F_3:=&-\sqrt{\kappa}\widetilde{\psi}(\kappa^{-\frac{1}{2}}n^-)\text{div}u^--(u^--v^-)\cdot\nabla n^-,\\
    F_4:=&- (u^--v^-) \cdot \nabla u^-+\nu(1+\widetilde{Q}( \kappa^{-\frac{1}{2}}n^-))\operatorname{div}\left( 2\widetilde{\phi}(\kappa^{-\frac{1}{2}}n^-)\mathbb{D}u^-\right)+\nu \widetilde{Q}( \kappa^{-\frac{1}{2}}n^-) (\Delta u^-+\nabla\operatorname{div}u^-) \\
    &-\widetilde{g}_4^-(\kappa^{-\frac{1}{2}}n^+,\kappa^{-\frac{1}{2}}n^{-})\kappa^{-\frac{1}{2}} \nabla n^- -\widetilde{g}_3^-(\kappa^{-\frac{1}{2}}n^+,\kappa^{-\frac{1}{2}}n^{-})\kappa^{-\frac{1}{2}} \nabla n^+\\
    &+\nabla\Big(\frac12|\nabla n^-|^2\Big)+ \sqrt{\kappa}\nabla(\widetilde{\psi}\Big(\kappa^{-\frac{1}{2}}n^-)\Delta n^-\Big):=     \sum_{l=1}^7F_4^l.\\
\end{aligned}
    \right.
\end{equation}
Here we denote
$\widetilde{Q}(n^\pm)=Q\circ L^{-1}(n^\pm)$ (similarly, definitions for $\widetilde{\phi}$, $\widetilde{\psi}$, and $\widetilde{g}_j $ with $j=1,2,3,4$)  and  the nonlinear functions of $(R^+,R^-)$ are  defined  by
\begin{equation*}
    \left\{
    \begin{aligned}
        g_1(R^+,R^-)&:=\frac{( \mathcal{C}^{2}\rho^-)(R^+,R^-)}{\rho^+(R^+,R^-)}\sqrt{\frac{R^+}{m(R^+)}}-\beta_1,\\
        g_4(R^+,R^-)&:=\left( \frac{(\mathcal{C}^{2}\rho^+)(R^+,R^-) }{\rho^-(R^+,R^-)}   -\frac{(\mathcal{C}^{2}\alpha^+)(R^+,R^-)f'(R^-)}{s_+^2(R^+,R^-)} \right)\sqrt{\frac{R^+}{m(R^+)}}-\beta_4,\\
    \end{aligned}
    \right.
\end{equation*}
\begin{equation*}
    \left\{
    \begin{aligned}
        g_2(R^+,R^-)&:= \left(\mathcal{C}^{2}(R^+,R^-) 
         +\frac{(\mathcal{C}^{2}\alpha^-)(R^+,R^-)f'(R^-)}{s_-^2(R^+,R^-)}\right)\sqrt{\frac{R^-}{m(R^-)}}-\beta_2,\\
         g_3(R^+,R^-)&:= \mathcal{C}^{2}(R^+,R^-)\sqrt{\frac{R^+}{m(R^+)}}-\beta_3.
    \end{aligned}
    \right.
\end{equation*}

\subsection{Main results}
Our first result establishes the global well-posedness of the Cauchy problem for the two-fluid system \eqref{system1}  in two dimensions with arbitrarily large initial data, provided the capillarity coefficient is sufficiently large. We also prove the convergence of the system \eqref{system1}  to the incompressible Navier-Stokes equations \eqref{INS}  with an explicit rate.

\begin{theorem}\label{cor1}
    Let $d=2$. Assume that 
    \begin{align} 
    -\frac{s_-^2(1,1)}{\alpha^-(1,1)}<f'(1)<0 \label{stability} 
\end{align}
holds,  and the initial data $(n_{0,\kappa}^\pm,u_{\kappa,0}^\pm)$ satisfies $\mathcal{P}u^\pm_{0,\kappa}=v_0^\pm\in \dot{B}^{0}_{2,1}$ and $$
    (\kappa^{-\frac{1}{2}}n_{0,\kappa}^\pm,\nabla n_{0,\kappa}^\pm, u^\pm_{0,\kappa})\in \dot{B}^{0}_{2,1}\quad \text{uniformly in}\quad \kappa.
    $$
    There exists a constant $\kappa_1>0$ such that if $\kappa\geq \kappa_1$, then the Cauchy problem \eqref{system-perturb1} admits a unique global solution $(n_\kappa^\pm, u_\kappa^\pm)$ satisfying
    $$
    \bigl(\kappa^{-\frac{1}{2}}n_\kappa^\pm,\;\nabla n_{\kappa}^\pm,\;u_{\kappa}^\pm\bigr)
\;\in\;
C\bigl(\mathbb{R}_+;\dot{B}_{2,1}^{0}\bigr)
\;\cap\;
L^1\bigl(\mathbb{R}_+;\dot{B}_{2,1}^{2}\bigr).
$$
Moreover, it holds for any $2<p<\infty$ that
\begin{equation}\label{eq:incompressible-limit:01}
\begin{aligned}
&\bigl\|(\kappa^{-\frac{1}{2}}n_\kappa^\pm,\nabla n_{\kappa}^\pm,\mathcal{Q}u_{\kappa}^\pm)\bigr\|_{L^2(\mathbb{R}_+;\dot B^{\frac 2 p}_{p,1})}+\bigl\|\mathcal{P} u_{\kappa}^\pm - v^\pm\bigr\|_{L^\infty(\mathbb{R}_+;\dot B^{0}_{2,1})\cap L^1(\mathbb{R}_+;\dot B^{2}_{2,1})}
\;\le\;
C\,\kappa^{-\frac{1}{4}(1-\frac{2}{p})}.
\end{aligned}
\end{equation}
\end{theorem}

For the case $d\geq3$, we also have the global well-posedness of the Cauchy problem for the two-fluid system under a mild smallness condition only depending on the incompressible parts for initial velocities, i.e., the initial data for the limiting incompressible Navier-Stokes equations \eqref{INS}.

\begin{theorem}\label{cor2}
    Let $d\geq 3$ and suppose that \eqref{stability} holds. Assume that the initial data $(n_{0,\kappa}^\pm,u_{\kappa,0}^\pm)$ satisfies $\mathcal{P}u^\pm_{0,\kappa}=v_0^\pm\in \dot{B}^{\frac{d}{2}-1}_{2,1}$, $$
    (\kappa^{-\frac{1}{2}}n_{0,\kappa}^\pm,\nabla n_{0,\kappa}^\pm, u^\pm_{0,\kappa})\in \dot{B}^{\frac{d}{2}-1}_{2,1}\quad\text{uniformly in}\quad \kappa\quad\text{and} \quad \|v_0^\pm\|_{\dot{B}^{\frac{d}{2}-1}_{2,1}}\leq \alpha_0,
    $$
 where  the constant $\alpha_0$ is given in \eqref{smallnessv0}. There exists a constant $\kappa_2>0$ such that if $\kappa\geq \kappa_2$, then the Cauchy problem \eqref{system-perturb1} admits a unique global solution 
 $(n_\kappa^\pm, u_\kappa^\pm)$ satisfying 
 $$
 \bigl(\kappa^{-\frac{1}{2}}n_\kappa^\pm,\;\nabla n_{\kappa}^\pm,\;u_{\kappa}^\pm\bigr)
\;\in\;
C\bigl(\mathbb{R}_+;\dot{B}_{2,1}^{\frac{d}{2}-1}\bigr)
\;\cap\;
L^1\bigl(\mathbb{R}_+;\dot{B}_{2,1}^{\frac{d}{2}+1}\bigr).
$$
Moreover, it holds for any $2<p\leq \frac{2d}{d-2}$ that
\begin{equation}\label{eq:incompressible-limit:02}
\begin{aligned}
&\bigl\|(\kappa^{-\frac{1}{2}}n_\kappa^\pm,\nabla n_{\kappa}^\pm,\mathcal{Q}u_{\kappa}^\pm)\bigr\|_{L^2(\mathbb{R}_+;\dot B^{\frac d p}_{p,1})}+\bigl\|\mathcal{P} u_{\kappa}^\pm - v^\pm\bigr\|_{L^\infty(\mathbb{R}_+;\dot B^{\frac d2-1}_{2,1})\cap L^1(\mathbb{R}_+;\dot B^{\frac d2+1}_{2,1})}
\;\le\;
C\,\kappa^{-\frac{1}{4}(\frac{d}{2}-\frac{d}{p})}.
\end{aligned}
\end{equation}
\end{theorem}

Theorems \ref{cor1} and \ref{cor2}  follow directly from Theorem \ref{thm:global-existence} and Proposition \ref{prop:incompressible} below. We establish global stability for the two-fluid model \eqref{system1} under large, critically regular perturbations in the high-capillarity regime, assuming that the solutions of the corresponding incompressible Navier-Stokes equations \eqref{INS} $v^\pm$  exist globally in time within critical spaces.

\begin{theorem}\label{thm:global-existence}
Let $d\geq2 $ and suppose that \eqref{stability} holds.
 Assume that the initial data $(n_{0,\kappa}^\pm,u_{\kappa,0}^\pm)$ satisfies
\begin{align}\label{initialenergy}
\|(\kappa^{-\frac{1}{2}}n_{0,\kappa}^\pm,\nabla n_{0,\kappa}^\pm, u^\pm_{0,\kappa})\|_{\dot{B}_{2,1}^{\frac{d}{2}-1}}\leq M_0,
\end{align}
where $M_0>0$ is a given constant independent of $\kappa$. If 
$v^\pm$ is the unique global solution to the Cauchy problem \eqref{INS} for the incompressible Navier–Stokes equations  with the initial datum $v^\pm_0$ such that
$$
v^\pm\;\in\;
C\bigl(\mathbb{R}_+;\dot B^{\frac d2-1}_{2,1}\bigr)
\;\cap\;
L^1\bigl(\mathbb{R}_+;\dot B^{\frac d2+1}_{2,1}\bigr),
$$
then there exists \(\kappa_0=\kappa_0(M_0,\nu,f,P^\pm)>0\), depending only on the initial data \footnote{The explicit choice of $\kappa_0$ can be found in \eqref{kappa0}.}, such that for all \(\kappa\ge\kappa_0\), under
\begin{align}\label{avelocity}
    \|\mathcal {P} u^\pm_{0,\kappa}-v^\pm_0\|_{\dot{B}^{\frac{d}{2}-1}_{2,1}}\leq \kappa^{-\delta},
\end{align}
 the Cauchy problem  \eqref{system-perturb1} admits a unique global-in-time solution $ (n_\kappa^\pm, u_\kappa^\pm) $ satisfying
\begin{align}\label{12455}
\Bigl(\kappa^{-\frac{1}{2}}n_\kappa^\pm,\;\nabla n_{\kappa}^\pm,\;u_{\kappa}^\pm\Bigr)
\;\in\;
C\bigl(\mathbb{R}_+;\dot{B}_{2,1}^{\frac{d}{2}-1}\bigr)
\;\cap\;
L^1\bigl(\mathbb{R}_+;\dot{B}_{2,1}^{\frac{d}{2}+1}\bigr).
\end{align}
Furthermore, the following error estimate holds: 
\begin{equation}\label{eq:incompressible-limit}
\begin{aligned}
&\bigl\|(\kappa^{-\frac{1}{2}}n_\kappa^\pm,\nabla n_{\kappa}^\pm,\mathcal{Q}u_{\kappa}^\pm)\bigr\|_{L^2(\mathbb{R}_+;\dot B^{\frac d p}_{p,1})}+\bigl\|\mathcal{P} u_{\kappa}^\pm - v^\pm\bigr\|_{L^\infty(\mathbb{R}_+;\dot B^{\frac d2-1}_{2,1})\cap L^1(\mathbb{R}_+;\dot B^{\frac d2+1}_{2,1})}
\;\le\;
C\,\kappa^{-\delta},
\end{aligned}
\end{equation}
where $C>0$ is a constant independent of $\kappa$, and the pair \((\delta,p)\) is given by
\begin{equation}
\left\{
\begin{aligned}\label{condi-p}
&2<p\leq\tfrac{2d}{d-2},\quad \delta =\frac{1}{4}(\frac{d}{2}-\frac{d}{p}),\quad \text{if}\quad d\geq3,\\
&2<p<\infty,~~\quad \delta =
 \frac{1}{4}(1-\frac{2}{p}),\quad\,  \text{if}\quad d=2.
\end{aligned}
\right.
 \end{equation}
In particular, it holds that
\begin{equation}
\left\{
\begin{aligned}
(\nabla n_{\kappa}^\pm,\;\mathcal{Q} u_{\kappa}^\pm)\;& \longrightarrow\;0
~~\quad\text{in}\quad
 L^2\bigl(\mathbb{R}_+;\dot B^{\frac{d}{p} }_{p,1}\bigr),\\
\mathcal{P} u_{\kappa}^\pm\;&\longrightarrow\;v^\pm\quad \text{in}\quad
L^\infty\bigl(\mathbb{R}_+;\dot{B}_{2,1}^{\frac{d}{2}-1}\bigr)
\;\cap\;
L^1\bigl(\mathbb{R}_+;\dot{B}_{2,1}^{\frac{d}{2}+1}\bigr).
\end{aligned}
\right.
\end{equation}
\end{theorem}

Theorem \ref{thm:global-existence} relies on the global existence of the solution $v^\pm$ to the Cauchy problem \eqref{INS} for the incompressible Navier–Stokes equations.
The standard global existence theorem for the incompressible Navier-Stokes equations in scaling critical spaces is given as follows. The small-data results are well-known in $\mathbb{R}^d$ ($d\geq2$), cf. Fujita-Kato \cite{fujita1}, Cannone \cite{cannone1}, Chemin \cite{chemin2}, et al. For the case $d=2$, the smallness can be removed (see Danchin \cite{danchin17Adv}).

\begin{proposition}\label{prop:incompressible}
  For $d\geq2$, assume that the initial data $v_0^\pm$ satisfies 
   \begin{itemize}
       \item {\rm(}$d=2${\rm)}: $v_0^\pm\in \dot{B}^{0}_{2,1}$. 
       \item {\rm(}$d\geq3${\rm)}: $v_0^\pm\in \dot{B}^{\frac{d}{2}-1}_{2,1}$, and there exists a generic constant $\alpha_0$ such that 
       \begin{align}\label{smallnessv0}
       \|v_0^\pm\|_{\dot{B}^{\frac{d}{2}-1}_{2,1}}\leq \alpha_0.
       \end{align}
   \end{itemize}
     Then the Cauchy problem \eqref{INS} has a unique global solution $v^\pm$ satisfying $v^\pm\in C(\mathbb{R}_+;\dot{B}^{\frac{d}{2}-1}_{2,1})$ and
     \begin{align}\label{es:in}
     \|v^\pm\|_{\widetilde{L}^{\infty}(\mathbb{R}_+;\dot{B}^{\frac{d}{2}-1}_{2,1})}+\|v^\pm\|_{L^1(\mathbb{R}_+;\dot{B}^{\frac{d}{2}+1}_{2,1})}\leq C \|v_0^\pm\|_{\dot{B}^{\frac{d}{2}-1}_{2,1}},
     \end{align}
     for a generic constant $C>0$.
\end{proposition}

\begin{remark}\normalfont
For the two-fluid system \eqref{system1}, all existing results heavily rely on the smallness of  initial data and the stronger structural condition \eqref{stability0}.  Moreover, to the best of our knowledge, there has been no result addressing how the capillarity coefficient influences the required size of initial data or the long-time behavior of solutions.
In Theorem \ref{thm:global-existence}, we establish the global well-posedness and high-capillarity limit of solutions with critical regularity for arbitrarily large initial perturbations, provided that the corresponding incompressible Navier–Stokes system admits a global solution and the capillary pressure $f$ 
satisfies the condition \eqref{stability} (a sharp requirement for stability, which is necessary to ensure that the 
energy functional $\mathcal{E}_j(t)$ remains positive (see \eqref{2.20} and \eqref{2.21} for details) Indeed, for $f'(1)>0$, the instability has been investigated in \cite{WuYaoZhangin}.
\end{remark}

\begin{remark}\normalfont
The key idea here is that by exploiting the Schr\"odinger-type dispersive operators arising from the large-capillarity linearization together with fully diffusive dissipation mechanisms, we derive sharp error estimates between solutions of the two-fluid system and the incompressible Navier–Stokes equations. As a result, for sufficiently large capillarity coefficient $\kappa$, we establish the existence and uniqueness of a global solution to the two-fluid system without requiring any smallness condition on the initial data.
\end{remark}

\begin{remark}\normalfont
Our results apply to general {\emph{ill-prepared}} data satisfying $\div u^\kappa_0=\mathcal{O}(1)$, without requiring $\div u^\kappa_0\rightarrow 0$ as $\kappa\rightarrow \infty$. 
\end{remark}

Finally, if the initial perturbation is uniformly bounded in $\dot{B}^{\sigma_1}_{2,\infty}$ ($-\frac{d}{2}\leq \sigma_1<\frac{d}{2}-1)$, we establish the optimal time-decay rates of the solution.

\begin{theorem}\label{thm:decay-transfer}
Let  $d\geq2 $ and let $ -\frac d2 \leq \sigma_1< \frac d2 -1.$ 
Let $ (n_{\kappa}^\pm,u_{\kappa}^\pm)$ be the global-in-time solution to the Cauchy problem \eqref{system-perturb1} constructed in Theorems \ref{cor1}-\ref{cor2}.  Assume the initial data additionally satisfies
\begin{equation}\label{eq:extra-data}
\bigl(\kappa^{-\frac{1}{2}} n_{\kappa,0}^\pm,\;\nabla n_{\kappa,0}^\pm,\;u_{\kappa,0}^\pm, \;v_0^\pm\bigr)
 \in \dot B^{\sigma_1}_{2,\infty}\quad\text{uniformly in}\quad \kappa.
\end{equation}
Then for all $t\geq1$, the solution $ (n_{\kappa}^\pm,u_{\kappa}^\pm) $ satisfies
\begin{equation}\label{eq:cns-decay}
\big\|D^\alpha\bigl( \kappa^{-\frac{1}{2}}n_{\kappa}^\pm,\;\nabla n_{\kappa}^\pm,\;u_{\kappa}^\pm \bigr)(t)\big\|_{L^2}
\;\leq C\;(1+t)^{-\frac{\lvert\alpha\rvert}{2}+\frac{\sigma_1}{2}}
\quad\text{for all }~~\lvert\alpha\rvert>\sigma_1,
\end{equation}
where the constant $C>0$ is independent of $t$ and $\kappa$.
\end{theorem}

\begin{remark}\normalfont
Theorem \ref{thm:decay-transfer} is the first result concerning the large-time smoothness of solutions for the non-conservative compressible two-fluid system, revealing the regularizing effect of the viscosity and capillarity terms with large initial data. It is shown that the solution becomes {\emph{globally smooth}} and all its derivatives decay at the rates that coincide with the optimal rates for the heat equation. 
In particular, under the stronger assumption $\bigl(\kappa^{-\frac{1}{2}} n_{\kappa,0}^\pm,\;\nabla n_{\kappa,0}^\pm,\;u_{\kappa,0}^\pm\bigr)\in L^1(\mathbb{R}^d)$, one infers from the embedding $L^1(\mathbb{R}^d)\hookrightarrow \dot{B}^{-\frac{d}{2}}_{2,\infty}$ and Gagliardo–Nirenberg inequalities that
\begin{align*}
\big\|D^\alpha\bigl( \kappa^{-\frac{1}{2}}n_{\kappa}^\pm,\;\nabla n_{\kappa}^\pm,\;u_{\kappa}^\pm \bigr)(t)\big\|_{L^p}
\;\lesssim \;(1+t)^{-\frac{d}{4}-\frac{\lvert\alpha\rvert}{2}-\frac{d}{2}(\frac{1}{2}-\frac{1}{p})}
\end{align*}
for all $\lvert\alpha\rvert>\sigma_1+d(\frac{1}{2}-\frac{1}{p})$ and $p\geq 2$.
\end{remark}

\subsection{Strategies}

Now, let us give some comments on the proofs of Theorems \ref{thm:global-existence}-\ref{thm:decay-transfer}. To prove Theorem  \ref{thm:global-existence}, we need to make a formal spectral analysis of the linearized system  of \eqref{system-perturb1}.  
 Let $\mathcal{P}:={\rm Id}+(-\Delta)^{-1}\nabla\operatorname{div}$, and $\mathcal{Q}={\rm Id}-\mathcal{P}$.
Then the compressible parts $\mathcal{Q}u^\pm $ of the system \eqref{system-perturb1} satisfy:
\begin{equation}\label{system-perturb2}
\left\{
\begin{aligned}
&\partial_t n^+ +v^+\cdot\nabla n^++ \sqrt{\kappa}\text{div}\mathcal{Q}u^+ = F_1, \\
&\partial_t \mathcal{Q}u^++\mathcal{Q}(v^+\cdot\nabla u^+) -   2\Delta\mathcal{Q}u^+ + \frac{\beta_1}{\sqrt{\kappa}} \nabla  n^+ + \frac{\beta_2}{\sqrt{\kappa}}\nabla  n^-  - \sqrt{\kappa} \nabla\Delta n^+ = \mathcal{Q}F_2, \\
&\partial_t n^- + v^-\cdot\nabla n^-+\sqrt{\kappa}\text{div}\mathcal{Q}u^- = F_3, \\
&\partial_t \mathcal{Q}u^-+\mathcal{Q}(v^-\cdot\nabla u^-) -  2 \Delta\mathcal{Q}u^-+ \frac{\beta_3}{\sqrt{\kappa}} \nabla  n^- + \frac{\beta_4}{\sqrt{\kappa}}\nabla  n^-  - \sqrt{\kappa} \nabla\Delta n^- =\mathcal{Q}F_4.
\end{aligned}
\right.
\end{equation}
and the incompressible  part $\mathcal{P}u^\pm$ satisfies
\begin{equation}
\left\{
\begin{aligned}
&\partial_t \mathcal{P}u^+ -   \Delta \mathcal{P}u^+ =\mathcal{P}F_2 , \\
&\partial_t \mathcal{P}u^- -   \Delta \mathcal{P}u^- =\mathcal{P}F_4 . \\
\end{aligned}
\right.
\end{equation}
Conducting a direct spectral analysis of the linearized system   \eqref{system-perturb1} is rather involved. The main observation here is to introduce the following four new linear combinations:
\begin{equation}\label{N+N-M+M-}
\left\{
\begin{aligned}
    &N_1^+:=\beta_3 n^++(r_+-\beta_1)n^-, \quad\quad~~ N_2^-:=\beta_3 n^++(r_--\beta_1)n^-,\\
    &M_1^+:=\beta_3 \mathcal{Q}u^++(r_+-\beta_1)\mathcal{Q}u^-,\quad M_2^-:=\beta_3 \mathcal{Q}u^++(r_--\beta_1)\mathcal{Q}u^-,
\end{aligned}
\right.
\end{equation}
where  $r_+$ and $r_-$ are given by
\begin{align}\label{lambdapm}
r_\pm:=\frac{1}{2}(\beta_1+\beta_4\pm \sqrt{(\beta_1-\beta_4)^2+4\beta_2\beta_3}),
\end{align}
which are the roots of 
$
r^2-(\beta_1+\beta_4)r+\beta_1\beta_4-\beta_2\beta_3=0.
$
Note that both $r_+$ and $r_-$ are strictly positive because $\beta_1\beta_4-\beta_2\beta_3>0$ under the stability condition \eqref{stability}. Therefore, owing to \eqref{system-perturb2}, the system \eqref{system-perturb1} can be decoupled into the following two  Navier-Stokes-Korteweg models:
\begin{align}\label{NSK_1}
\begin{cases}
\partial_t N_1^+ + \sqrt{\kappa}\mathrm{div}\,M_1^+ = G_1^+,\\[0.5ex]
\partial_t M_1^+ + \frac{r_+}{\sqrt{\kappa}}\nabla N_1^+ - 2\Delta M_1^+  - \sqrt{\kappa}\nabla\Delta N_2^+ = G^+_2
\end{cases}
\end{align}
and 
\begin{align}\label{NSK_2}
\begin{cases}
\partial_t N_2^- + \sqrt{\kappa}\mathrm{div}\,M_2^- = G^-_1,\\[0.5ex]
\partial_t M_2^- + \frac{r_-}{\sqrt{\kappa}}\nabla N_2^- -2\Delta M_2^-  - \sqrt{\kappa}\nabla\Delta N_2^- = G_2^-,
\end{cases}
\end{align}
where \begin{align*}
    \begin{cases}
        G_1^\pm:=\beta_3(F_1-v\cdot\nabla n^+)+(r_\pm-\beta_1)(F_3-v\cdot\nabla n^-),\\[0.5ex]
        G_2^\pm:=\beta_3\mathcal{Q}(F_2-v^-\cdot\nabla u^+)+(r_\pm-\beta_1)\mathcal{Q}(F_4-v^-\cdot\nabla u^-). 
    \end{cases}
\end{align*}

Now, it is convenient to introduce $\mathcal{M}_1^+=\Lambda^{-1}\operatorname{div}M_1^+$ and $\mathcal{M}_2^-=\Lambda^{-1}\operatorname{div}M_2^-$. Then $(N_1^+,\mathcal{M}_1^+)$ and $(N_2^-,\mathcal{M}_2^-)$ satisfy the coupling $2\times2$ systems:
\begin{equation}\label{eq:compress1}
\left\{
 \begin{aligned}
    &\partial_t N_1^+ + \sqrt{\kappa}\mathrm{div}\,\mathcal{M}_1^+ =0 ,\\
&\partial_t \mathcal{M}_1^+ + \frac{r_+}{\sqrt{\kappa}}\nabla N_1^+ -  2\Delta \mathcal{M}_1^+  - \sqrt{\kappa}\nabla\Delta N_1^+ = 0
 \end{aligned}
 \right.
\end{equation}
and
\begin{equation}\label{eq:compress2}
\left\{
 \begin{aligned}
    &\partial_t N_2^- + \sqrt{\kappa}\mathrm{div}\,\mathcal{M}_2^- =0 ,\\
&\partial_t \mathcal{M}_2^- + \frac{r_-}{\sqrt{\kappa}}\nabla N_2^- -  2\Delta \mathcal{M}_2^-  - \sqrt{\kappa}\nabla\Delta N_2^- = 0.
 \end{aligned}
 \right.
\end{equation}
Then taking the Fourier transform of systems \eqref{eq:compress1} and \eqref{eq:compress2} with respect to $x$,  $\mathcal{U}^+=(N_1^+,\mathcal{M}_1^+)^T$ and $\mathcal{U}^-=(N_2^-,\mathcal{M}_2^-)^T$ satisfies the following  equation in matrix form:
\begin{equation*}
    \begin{aligned}
        \partial_t\mathcal{U}^\pm=A^\pm (\xi)\mathcal{U}^\pm
    \end{aligned}
\end{equation*}
with
$$
A^\pm(\xi)=\left(\begin{matrix}
    0&-\sqrt{\kappa}|\xi|\\
      \frac{r_\pm}{\sqrt{\kappa}}|\xi|+\sqrt{\kappa}|\xi|^3&- 2|\xi|^2
\end{matrix}\right), 
$$
where $\xi$ is the Fourier variable. It is easy to check that $$\lambda_\pm(\xi)=- |\xi|^2\pm\sqrt{(1-\kappa)|\xi|^4-r_\pm|\xi|^2}.$$
Consequently, we find that the coupling systems \eqref{eq:compress1} and \eqref{eq:compress2}  present {\emph{both dissipation and dispersion}}.
In particular, for $\kappa>1$, the corresponding dispersive structure is closely related to the so-called {\emph{Gross–Pitaevskii}} equation \cite{GNT206,GNT209}, which reads as
\begin{align*}
    i\partial_t\psi+\Delta\psi-2\text{Re}\,\psi=F(\psi).
\end{align*}

To prove Theorem  \ref{thm:global-existence}, our key ingredient is to combine dissipative with dispersive estimates, along with the error estimates between the incompressible part $\mathcal{P}u^\pm$ and the limiting solution $v^\pm$. 
Consequently, even without imposing any smallness condition on the initial data, we succeed in deriving the uniform energy estimates for suitably large $\kappa$. The proof of global stability can be constructed by the following four steps:

(1) First, we derive uniform energy-dissipation estimates for $(\kappa^{-\frac{1}{2}}n^\pm, \nabla n^\pm, u^\pm)$ in $L^{\infty}_t(\dot{B}^{\frac{d}{2}-1}_{2,1})\cap L^1_t(\dot{B}^{\frac{d}{2}+1}_{2,1})$. To achieve this goal, for every  $j\in\mathbb{Z}$, we construct a Lyapunov functional $\mathcal{E}_j$ such that
\begin{equation}\label{Ly}
    \begin{aligned}
     \mathcal{E}_j(t)\sim \left\| \big(\kappa^{-\frac{1}{2}}n_j^+,\kappa^{-\frac{1}{2}}n_j^-,\nabla n_j^+,\nabla n_j^-,  u_j^+, u_j^+\big)\right\|_{L^2}^2\quad\text{and}\quad   \frac{{\rm d}}{{\rm d}t} \mathcal{E}_j(t)+c_0 2^{2j}\mathcal{E}_j(t)&\lesssim \mathcal{R}_j\sqrt{\mathcal{E}_j(t)},
    \end{aligned}
\end{equation}
where we write $f_j=\dot{\Delta}_jf$, and $\mathcal{R}_j$ denotes the nonlinear term. It should be noted that \eqref{Ly} reveals the {\emph{fully diffusive in all frequencies}} for the linearized system, which is stronger than the dissipation structures for generic
non-conservative compressible two-fluid equations without capillarity (see  \cite{EvjeWangWang2016}), where the high frequencies only have a weaker damping effect. See Lemma \ref{lemma22} for more details.

(2) Next,  we establish the dispersive estimates for $(\kappa^{-\frac{1}{2}}n^\pm,\nabla n^\pm,\mathcal Q u^\pm)$ in the $L^p$ type space $L^2_t(\dot{B}^{\frac{d}{p}}_{p,1})$, with the rate $\kappa^{-\delta}$.  Inspired by \cite{CV-Song-2024} for the study of one-phase fluids, we introduce the two-phase  dispersive quantities
\begin{equation}\label{UH}
    \begin{aligned}
        U_\pm  =\sqrt{\frac{-\Delta}{r_\pm\kappa^{-1}-\Delta}}\quad\mbox{ and }\quad H_\pm= \sqrt{-\Delta(r_\pm\kappa^{-1}-\Delta)}
    \end{aligned}
\end{equation}
and define 
\begin{align}\label{Z1Z2}
z_1^+:=U_+^{-1}\nabla N_1^++i\mathcal{Q}M_1^+ \quad\text{and}\quad z_2^-:=U_-^{-1}\nabla N_2^-+i\mathcal{Q}M_2^-.
\end{align}
Then  \eqref{NSK_1} and \eqref{NSK_2} can be rewritten as 
\begin{equation}\label{GP}
\left\{
\begin{aligned}
&i\partial_t z_1^{+} -i\nu\Delta z_1^{+}-\sqrt{\kappa}H_+z_1^{+}  = -i\nu H_+\Lambda N_1^{+}+iU_+^{-1}\nabla G_1^+ -G_2^+ , \\
&i\partial_t z_2^- -i\nu\Delta z_2^--\sqrt{\kappa}H_-z_2^-    = -i\nu H_-\Lambda N_2^{-} +iU_-^{-1}\nabla G_1^- -G_2^- . \\
\end{aligned}
\right.
\end{equation}
With the above reformulation, we employ Strichartz estimates for the two-phase Gross–Pitaevskii structure in \eqref{GP} to obtain the desired bounds for $(z_1^+,z_2^-)$. Combining these with \eqref{UH} and 
\begin{equation}\label{mmmm}
\left\{
\begin{aligned}
    &\mathcal{Q}u^+=\frac{(\beta_1-r_-)M_1^++(r_+-\beta_1)M_2^-}{r_+-r_-},\quad \mathcal{Q}u^-=\frac{M_1^+-M_2^-}{r_+-r_-},\\
    &n^+=\frac{(\beta_1-r_-)N_1^++(r_+-\beta_1)N_2^-}{r_+-r_-},\quad\quad~  n^-=\frac{N_1^+-N_2^-}{r_+-r_-},
\end{aligned}
\right.
\end{equation}
one can recover the desired dispersive estimates for $(\nabla n^\pm,\mathcal{Q}u^\pm)$ in $L^2_t(\dot{B}^{\frac{d}{p}}_{p,1})$. For details, readers can refer to Lemma \ref{lemma23}.

(3) Subsequently, we consider the error $\widetilde{u}^\pm = \mathcal{P}u^\pm - v^\pm$, representing the deviation between the incompressible parts of the solutions to the compressible two-fluid model and the incompressible Navier–Stokes equations. These differences solve heat equations with inhomogeneous source terms as follows:
\begin{equation*} 
\left\{
\begin{aligned}
&\partial_t \widetilde{u}^++ \mathcal{P}(v^+\cdot\nabla \mathcal{P}\widetilde{u}^+)-  \Delta \widetilde{u}^+  = \mathcal{P}\bar{F}_2 , \\
&\partial_t \widetilde{u}^-+  \mathcal{P}(v^-\cdot\nabla \mathcal{P}\widetilde{u}^-)  - \Delta  \widetilde{u}^- =  \mathcal{P}\bar{F}_4 , 
\end{aligned}
\right.
\end{equation*}
where  
\begin{equation*}
\left\{
\begin{aligned}
    \bar{F}_2 &:=- v^+\cdot\nabla \mathcal{Q}u^++F_2^1+F_2^2+F_2^3+F_2^4+F_2^5,\\
    \bar{F}_4 &:= -v^-\cdot\nabla \mathcal{Q}u^- +  F_4^1+F_4^2+F_4^3+F_4^4+F_4^5.
\end{aligned}
\right.
\end{equation*}
By applying maximal regularity estimates associated with the Lamé system, we rigorously obtain the convergence rate $\kappa^{-\delta}$ of $\mathcal{P}u^\pm-v^\pm$ in $L^{\infty}_t(\dot{B}^{\frac{d}{2}-1}_{2,1})\cap L^1_t(\dot{B}^{\frac{d}{2}+1}_{2,1})$.

(4) Finally, synthesizing the above three steps, we derive uniform-in-time a priori energy inequalities (see Proposition  \ref{prop1}) that ensure the global existence and stability of solutions.

The main difficulties arise from handling the higher-order nonlinear terms without relying on any smallness assumption regarding the initial data.
To overcome these difficulties, we make full use of the regularities of the limiting solutions, Gr\"onwall's lemma and the smallness caused by the $\kappa^{-\delta}$-rate for the dispersive quantities $(\nabla n^\pm, \mathcal{Q}u^\pm)$. 
For example, in the global existence proof, we introduce the auxiliary fields $v^\pm$ in \eqref{system-perturb1} to derive the incompressible limit. Indeed, the convective term, e.g., $u^\pm \cdot \nabla u^\pm$, whose $L^1\bigl(\dot B^{\tfrac d2-1}_{2,1}\bigr)$–norm cannot be directly controlled. We localized $u^\pm \cdot \nabla u^\pm$ as
$$
\dot{\Delta}_j\big(u^\pm \cdot \nabla u^\pm\big)= v^\pm \cdot \dot{\Delta}_j  u^\pm+[\dot{\Delta}_j ,v^\pm ] \nabla u^\pm+  \dot{\Delta}_j\big(\mathcal{Q}u^\pm \cdot \nabla u^\pm \big)+\dot{\Delta}_j\big((\mathcal{P}u^\pm-v^\pm) \cdot \nabla u^\pm \big).
$$
Note that when performing $L^2$ energy estimates, one can handle the first two terms on the right-hand side of the above decomposition by using $\div v^\pm=0$, commutator estimates, and Gr\"onwall's inequality, while the last two terms can be analyzed by some non-standard product laws and the $\kappa^{-\delta}$-bounds for $\|\mathcal{Q}u^\pm\|_{L^2_t(\dot{B}^{\frac{d}{p}}_{p,1})}$ and $\|\mathcal{P}u^\pm-v^\pm\|_{L^\infty_t(\dot B^{\frac d2-1}_{2,1})\cap L^1_t(\dot B^{\frac d2+1}_{2,1})}$. Similar arguments can also be applied to address the nonlinear terms arising from the pressure and the Korteweg tensor by making use of the  $\kappa^{-\delta}$-bound for $\|\nabla n^\pm\|_{L^2_t(\dot{B}^{\frac{d}{p}}_{p,1})}$. With these observations, when $\kappa$ is chosen to be suitably large, we can enclose uniform {\emph a priori} estimates for general data, which ensures us to prove the global existence and prove Theorem \ref{thm:global-existence}.

 
After establishing the global existence of large solutions, we address the natural question of determining their optimal time–decay rates (Theorem \ref{thm:decay-transfer}). To this end, we develop several new ideas and techniques, with our strategy based on a time–weighted Lyapunov framework for optimal decay of any-order derivatives tailored to the critical regularity data. Specifically, for a large exponent $M\gg 1$, we multiply the Lyapunov inequality \eqref{Ly} by $t^{M}$ ($M>>1$) to obtain, for $j\in\mathbb{Z}$,
\begin{align}\label{Lymm}
\frac{{\rm d}}{{\rm d}t} \Big( t^{M}\mathcal{E}_j(t)\Big)+c_0 2^{2j} t^M \mathcal{E}_j(t)&\lesssim \Big(t^{M-1}\sqrt{\mathcal{E}_j(t)} +t^M \mathcal{R}_j \Big)\sqrt{\mathcal{E}_j(t)}.
\end{align}
A key point is that $t^{M}\mathcal{E}_j(0)=0$, so no additional higher-order regularity at $t=0$ is needed to run large-time weighted estimates. In the spirit of Theorem~\ref{thm:global-existence}, the nonlinear remainder $t^{M}\mathcal{R}_j$ is controlled by Moser-type product bounds, high-order commutator estimates, the dispersive and incompressible error bound \eqref{eq:incompressible-limit} and a Gr\"onwall argument. Consequently, for any $s>\frac{d}{2}-1$, \eqref{Lymm} yields
 \begin{equation}\label{sgghhbb}
   \begin{aligned}
   &(1-C\kappa^{-\delta})\left\|\left(\kappa^{-\frac{1}{2}}n^+,\kappa^{-\frac{1}{2}}n^-,\nabla n^+,\nabla n^-,   u^+, u^-\right)\right\|_{L_t^\infty(\dot{B}^{s}_{2,1})\cap L^1_t(\dot{B}^{s+2}_{2,1})}\\
   &\quad\lesssim \left \|t^{M-1}\left(\kappa^{-\frac{1}{2}}n^+,\kappa^{-\frac{1}{2}}n^-,\nabla n^+,\nabla n^-,   u^+, u^-\right)\right\|_{L_t^1(\dot{B}^{s}_{2,1})}+\|t^M(v^+, v^-)\|_{L_t^\infty(\dot{B}^{s}_{2,1})\cap L^1_t(\dot{B}^{s+2}_{2,1})}.
\end{aligned}  
 \end{equation}
After taking $\kappa$ sufficiently large, it suffices to analyze the right-hand side of \eqref{sgghhbb}. An observation is to use real interpolation between $\dot{B}^{\sigma_1}_{2,\infty}$ and $\dot{B}^{s+2}_{2,1}$ such that
\begin{align*}
&\left \|t^{M-1}\left(\kappa^{-\frac{1}{2}}n^+,\kappa^{-\frac{1}{2}}n^-,\nabla n^+,\nabla n^-,   u^+, u^-\right)\right\|_{L_t^1(\dot{B}^{s}_{2,1})}\\
\leq & \frac{1}{2}\left\|t^{M}\left(\kappa^{-\frac{1}{2}}n^+,\kappa^{-\frac{1}{2}}n^-,\nabla n^+,\nabla n^-,   u^+, u^-\right)\right\|_{L^1_t(\dot{B}^{s+2}_{2,1})}\\
&\quad+Ct^{M-\frac{1}{2}(s-\sigma_1)} \left\|\left(\kappa^{-\frac{1}{2}}n^+,\kappa^{-\frac{1}{2}}n^-,\nabla n^+,\nabla n^-,   u^+, u^-\right)\right\|_{L^{\infty}_t(\dot{B}^{\sigma_1}_{2,\infty})}.
\end{align*}
In parallel, it is also proved that the solutions $v^\pm$ to the incompressible Navier-Stokes equations \eqref{INS} satisfy
\begin{align*}
\|t^M(v^+, v^-)\|_{L_t^\infty(\dot{B}^{s}_{2,1})\cap L^1_t(\dot{B}^{s+2}_{2,1})}\lesssim t^{M-\frac{1}{2}(s-\sigma_1)} \|(v^+,v^-)\|_{\dot{B}^{\sigma_1}_{2,\infty}}.
\end{align*}
Thus, to complete the proof of Theorem \ref{thm:decay-transfer}, we prove the following propagation of $\dot{B}^{\sigma_1}_{2,\infty}$-regularity:
\begin{align*}
\sup_{t\in\mathbb{R}^+}\left\|\left(\kappa^{-\frac{1}{2}}n^+,\kappa^{-\frac{1}{2}}n^-,\nabla n^+,\nabla n^-,   u^+, u^-, v^+,v^-\right)\right\|_{\dot{B}^{\sigma_1}_{2,\infty}}<\infty.
\end{align*}
The framework we employ enables us to derive any-order uniform-in-$\kappa$ decay estimates of large solutions, essentially due to the pure parabolic (dissipative) effect of the perturbed system \eqref{system-perturb1}.

\subsection{Notations and outline}\label{notations}
 We list some notations that are used frequently throughout the paper.
For simplicity, $C$ denotes a generic positive constant. $A\lesssim B$ ($A\gtrsim B$) means that $A\leq C B$ ($A\geq C B$),  while $A\sim B$ means that both $A\lesssim  B$
and $A\gtrsim B$. 
For $A=(A_{i,j})_{1\leq i,j\leq d}$ and $B=(B_{i,j})_{1\leq i,j\leq d}$ two $d\times d$ matrices, we denote $A:B={\rm{Tr}}\,AB=\sum_{i,j}A_{i,j}B_{j,i}$. 
 For any Banach space $X$ and the functions $g,h\in X$, let $\|(g,h)\|_{X}:=\|g\|_{X}+\|h\|_{X}$. For $p\in[1, \infty]$ and
$T>0$, the notation $L^p(0, T; X)$ or $L^p_T(X)$ designates the set of measurable functions $f: [0, T]\to X$ with $t\mapsto\|f(t)\|_X$ in $L^p(0, T)$, endowed with the norm $\|\cdot\|_{L^p_{T}(X)} :=\|\|\cdot\|_X\|_{L^p(0, T)}$, and $C([0,T];X)$ denotes the set of continuous functions $f: [0, T]\to X$. Let $\mathcal{F}(f)=\widehat{f}$ and $\mathcal{F}^{-1}(f)=\breve{f}$ be the Fourier transform of $f$ and its inverse. We denote $\mathcal{P}:={\rm Id}+(-\Delta)^{-1}\nabla\operatorname{div}$ as the incompressible projector and $\mathcal{Q}=I-\mathcal{P}$ as the compressible projector, respectively.

\vspace{2mm}

For the convenience of readers,
 we shall collect some basic facts on  Littlewood-Paley theory in this section to make the paper as self-contained as possible. The reader may refer to Chapters 2 and 3 in \cite{bahouri1} for more details.

Let $\chi(\xi)$ be 
 a smooth, radial, non-increasing function that is  compactly supported in $B(0,\frac{4}{3})$ and satisfies $\chi(\xi)=1$ in $B(0,\frac{3}{4})$. Then $\varphi(\xi):=  \chi(\frac{\xi}{2})-\chi(\xi)$ verifies
$$
\sum_{j\in \mathbb{Z}}\varphi(2^{-j}\xi)=1\quad\mbox{for}\ \ \xi\neq 0, \text{and}\quad  \operatorname{Supp} \varphi\subset \Big\{\xi\in\mathbb{R}^{d}~|~\frac{3}{4}\leq |\xi|\leq \frac{8}{3}\Big\}.
$$
For any $j\in \mathbb{Z}$, we define the homogeneous dyadic blocks $\dot{\Delta}_{j}$ by
$$
\dot{\Delta}_{j}u:=  \mathcal{F}^{-1}\big{(} \varphi(2^{-j}\cdot )\mathcal{F}(u) \big{)}=2^{jd}h(2^{j}\cdot)\star u\quad\text{with}\quad h:=  \mathcal{F}^{-1}\varphi,
$$
which satisfy $\dot{\Delta}_{j}\dot{\Delta}_{\ell}u=0$ if $|j-\ell|\geq2$.  We also define the low-frequency cut-off
$$
\dot{S}_j u:= \mathcal{F}^{-1}\big{(} \varphi(2^{-j}\cdot )\mathcal{F}(u) \big{)}=\sum_{j'\leq j-1}\dot{\Delta}_j u.
$$
We have the Littlewood-Paley decomposition $u=\sum_{j\in \mathbb{Z}}\dot{\Delta}_{j}u$
that holds in $\mathcal{S}_h'$, i.e.,  tempered distributions $\mathcal{S}'$ such that $\lim\limits_{j\rightarrow-\infty}\|\dot{S}_ju\|_{L^{\infty}}=0$ holds for any $u\in \mathcal{S}_h'$.

With the above  dyadic operators, we recall the definition of homogeneous Besov spaces.

\begin{defn}
For $s\in \mathbb{R}$ and $1\leq p,r\leq \infty$, the  homogeneous Besov norm $\|\cdot\|_{\dot{B}^{s}_{p,r}}$ is defined by
$$
\|u\|_{\dot{B}^{s}_{p,r}}:=  \|\{2^{js}\|\dot{\Delta}_{j}u\|_{L^{p}}\}_{j\in\mathbb{Z}}\|_{\ell^{r}}.
$$
\begin{itemize}
\item If $s<\frac{d}{p}$ {\rm(}or $s=\frac{d}{p}$ and $r=1${\rm)}, we define the Besov space $\dot{B}^{s}_{p,r}:=  \{u\in\mathcal{S}_{h}'~:~\|u\|_{\dot{B}^{s}_{p,r}}<\infty\}$.
\item If $k\in \mathbb{N}$ and $\frac{d}{p}+k\leq s<\frac{d}{p}+k+1$ {\rm(}or $s=\frac{d}{p}+k+1$ and $r=1${\rm)}, then we define $\dot{B}^{s}_{p,r}$ as the subset of $u$ in $\mathcal{S}_{h}'$ such that $\nabla^{\beta} u$ belongs to $\dot{B}^{s-k}_{p,r}$ whenever $|\beta|=k$.
\end{itemize}
\end{defn}
Here and below, all functional spaces
will be considered in
$\mathbb{R}^{d}$, we shall omit the space dependence for simplicity. Due to Bernstein's inequalities, there are  well-known  properties in Besov spaces.
\begin{lemma}\label{lemma62}
The following properties hold{\rm:}
\begin{itemize}
\item{} For any $s\in\mathbb{R}$, $1\leq p_{1}\leq p_{2}\leq \infty$ and $1\leq r_{1}\leq r_{2}\leq \infty$, it holds
\begin{equation}\nonumber
\begin{aligned}
\dot{B}^{s}_{p_{1},r_{1}}\hookrightarrow \dot{B}^{s-d\left(\frac{1}{p_{1}}-\frac{1}{p_{2}}\right)}_{p_{2},r_{2}}.
\end{aligned}
\end{equation}
\item{} For any $1\leq p\leq q\leq\infty$, we have the following chain of continuous embedding{\rm:}
\begin{equation}\nonumber
\begin{aligned}
\dot{B}^{0}_{p,1}\hookrightarrow L^{p}\hookrightarrow \dot{B}^{0}_{p,\infty}\hookrightarrow \dot{B}^{\sigma}_{q,\infty} \quad \text{for} \quad\sigma=-d\left(\frac{1}{p}-\frac{1}{q}\right)<0.
\end{aligned}
\end{equation}

\item{}
For any $\sigma\in \mathbb{R}$, the Fourier multiplier $\Lambda^{\sigma}$ with symbol $|\xi|^{\sigma}$ is an isomorphism from $\dot{B}^{s}_{p,r}$ to $\dot{B}^{s-\sigma}_{p,r}$.
\item{} Let $1\leq p_{1},p_{2},r_{1},r_{2}\leq \infty$, $s_{1}\in\mathbb{R}$ and $s_{2}\in\mathbb{R}$ satisfy
    \begin{align}
    s_{2}<\frac{d}{p_{2}}\quad\text{\text{or}}\quad s_{2}=\frac{d}{p_{2}}~\text{and}~r_{2}=1.\nonumber
    \end{align}
    The space $\dot{B}^{s_{1}}_{p_{1},r_{1}}\cap \dot{B}^{s_{2}}_{p_{2},r_{2}}$ endowed with the norm $\|\cdot \|_{\dot{B}^{s_{1}}_{p_{1},r_{1}}}+\|\cdot\|_{\dot{B}^{s_{2}}_{p_{2},r_{2}}}$ is a Banach space. 
\end{itemize}
\end{lemma}

\section{Proof of Theorem \ref{thm:global-existence}}\label{section2}

\subsection{Uniform a-priori estimates}
The proof of Theorem \ref{thm:global-existence} relies on uniform a priori estimates for the solution. Throughout this section, we simplify the notation by suppressing the superscript $\kappa$ on the unknowns for brevity.

We define the energy-dissipation functional
\begin{align}\label{definition-E}
    \mathcal{E}_t:=\left\|\left(\kappa^{-\frac{1}{2}}n^+,\kappa^{-\frac{1}{2}}n^+,\nabla n^+,\nabla n^-,u^+,u^-\right)\right\|_{L^\infty_t(\dot{B}_{2,1}^{\frac{d}{2}-1})\cap L^1_t(\dot{B}_{2,1}^{\frac{d}{2}+1})}
\end{align}
and its initial value
\begin{align}\label{definition-E0}
    \mathcal{E}_0:=\left\|\left(\kappa^{-\frac{1}{2}}n_0^+,\kappa^{-\frac{1}{2}}n_0^+,\nabla n_0^+,\nabla n_0^-,u_0^+,u_0^-\right)\right\|_{\dot{B}_{2,1}^{\frac{d}{2}-1}}.
\end{align}
The dispersive estimates of $(n^\pm,\mathcal{Q}u^\pm)$, together with error bounds for $\mathcal{P}u^{\pm}-v^\pm$, play a key role in controlling  \eqref{definition-E} without assuming small initial perturbations. Introduce the dispersion functional
\begin{align}\label{definition-D}
\mathcal{D}_t:=\left\|\left(\kappa^{-\frac{1}{2}}n^+,\kappa^{-\frac{1}{2}}n^+,\nabla n^+,\nabla n^-,\mathcal{Q}u^+,\mathcal{Q}u^-\right)\right\|_{L^2_t(\dot{B}_{p,1}^{\frac{d}{p}})}
\end{align}
and the error functional
\begin{align}\label{definition-W}
    \mathcal{W}_t:=\|(\mathcal{P}u^{+}-v^+,\mathcal{P}u^{-}-v^-)\|_{L^\infty_t(\dot{B}_{2,1}^{\frac{d}{2}-1})\cap L^1_t(\dot{B}_{2,1}^{\frac{d}{2}+1})}.
\end{align}
We also denote the functional of the limiting solution $v^\pm$ by 
\begin{align}\label{definition-V}
    \mathcal{V}_t:=\| (v^+,v^-)\|_{L^\infty_t(\dot{B}_{2,1}^{\frac{d}{2}-1})}+\| (v^+,v^-)\|_{L^1_t(\dot{B}_{2,1}^{\frac{d}{2}+1})}.
\end{align}

We now establish uniform a-priori estimates, carefully tracking their dependence on the parameter $\kappa$.

\begin{proposition}\label{prop1}
    Let $d\geq2 $ and $\kappa>1$. Suppose that $(n^\pm,u^\pm)$ is the solution of the Cauchy problem \eqref{system-perturb1} defined on $[0,T)\times \mathbb{R}^d$. Then, for all $t\in [0,T)$, 
 the following estimates 
    \begin{align}
           &\mathcal{E}_t\leq C_0e^{C_0\mathcal{V}_t}\bigg(  \mathcal{E}_0+\mathcal{W}_t\mathcal{E}_t+ (1+\mathcal{E}_t)^{\frac{d}{2}+1}\big(\mathcal{E}_t \mathcal{D}_t+\kappa^{-\frac{1}{2}}\mathcal{E}_t^2\big) \Bigg),\label{pri-1}\\
& \mathcal{D}_t\leq C_0  \kappa^{-\delta}\Big(\mathcal{E}_0 + \mathcal{E}_t+ (1+\mathcal{E}_t)^{\frac{d}{2}+1} \mathcal{E}_t^2\Big) \label{pri-21}
\end{align}
    and 
\begin{equation}\label{pri-2}
    \begin{aligned}
        \mathcal{W}_t \leq C_0e^{C_0(\mathcal{V}_t+\mathcal{E}_t^2)}\left(\|(\mathcal {P} u^+_{0,\kappa}-v^+_0,\mathcal {P} u^-_{0,\kappa}-v^-_0)\|_{\dot{B}^{\frac{d}{2}-1}_{2,1}} + (1+\mathcal{E}_t)^{\frac{d}{2}+1}\big(\mathcal{E}_t \mathcal{D}_t+\kappa^{-\frac{1}{2}}\mathcal{E}_t^2\big)\right)
    \end{aligned}
\end{equation}
hold, where $C_0>0$ is a constant uniform in $t$ and $\kappa$.
\end{proposition}

We split the proof of Proposition \ref{prop1} into Lemmas \ref{lemma22}, \ref{lemma23} and \ref{lemma24}.

\subsection{Uniform dissipative estimates for \texorpdfstring{$(\kappa^{-\frac{1}{2}}n^\pm, n^\pm,  u^\pm)$}{(n±, Q u±)}}

Before establishing uniform estimates of $\mathcal{E}_t$, we first employ hypocoercivity in Fourier space and construct localized Lyapunov inequalities that exhibit fully diffusive dissipation of the solution.

\begin{lemma}\label{lemma21}
For any $j\in \mathbb{Z}$, there exists a Lyapunov functional $\mathcal{E}_j(t)$ and a generic constant $c_0>0$ such that
\begin{align}\label{ineqforbasicenergy0}
\mathcal{E}_j(t)\sim \left\| \big(\kappa^{-\frac{1}{2}}n_j^+,\kappa^{-\frac{1}{2}}n_j^-,\nabla n_j^+,\nabla n_j^-,  u_j^+, u_j^+\big)\right\|_{L^2}^2
\end{align}
and
\begin{equation}\label{ineqforbasicenergy1}
    \begin{aligned}
        \frac{{\rm d}}{{\rm d}t} \mathcal{E}_j(t)+c_0 2^{2j}\mathcal{E}_j(t)&\lesssim \mathcal{R}_j\sqrt{\mathcal{E}_j(t)},
    \end{aligned}
\end{equation}
where $(n_j^\pm, u_j^\pm):=(\dot{\Delta}_jn^\pm,\dot{\Delta}_ju^\pm)$ and
\begin{equation*}
\begin{aligned}
\mathcal{R}_j:&=\Big\|\dot{\Delta}_j\big(\kappa^{-\frac{1}{2}}F_1,\kappa^{-\frac{1}{2}}F_3,\sum_{l=1}^6 F_2^l ,\sum_{l=1}^6 F_4^l\big)\Big\|_{L^2}\\
        & \quad+ \|(v^+,v^-)\|_{\dot{B}_{2,1}^{\frac{d}{2}+1}} \|(\kappa^{-\frac{1}{2}}n_j^+,\kappa^{-\frac{1}{2}}n_j^-,\nabla n_j^+,\nabla n_j^-,   u_j^+,  u_j^+)\|_{L^2}\\
        &\quad+2^j\|\dot{\Delta}_j ((u^+-v^+)\cdot\nabla n^+)\|_{L^2}+2^j\|\dot{\Delta}_j ((u^--v^-)\cdot\nabla n^-)\|_{L^2}\\
        &\quad+2^j\sqrt{\kappa}\|[\dot{\Delta}_j, \widetilde{\psi}(\kappa^{-\frac12}n^+)](\operatorname{div}u^++\Delta n^+)\|_{L^2}+2^j\sqrt{\kappa}\|[\dot{\Delta}_j, \widetilde{\psi}(\kappa^{-\frac12}n^-)](\operatorname{div}u^-+\Delta n^-)\|_{L^2}\\
        &\quad+ \|[\dot{\Delta}_j,v^+\cdot\nabla](\kappa^{-\frac{1}{2}}n^+,\nabla n^+,u^+)\|_{L^2}+\|[\dot{\Delta}_j,v^-\cdot\nabla](\kappa^{-\frac{1}{2}}n^-,\nabla n^-,u^-)\|_{L^2}.
\end{aligned}
\end{equation*}
\end{lemma}

\begin{proof}By applying the Fourier localization operator $\dot{\Delta}_j$ to \eqref{system-perturb1}, we rewrite the coupling system of $(n^\pm,  u^\pm)$ as 
\begin{equation}\label{sys1}
\left\{
\begin{aligned}
&\partial_t n_j^++v^+\cdot\nabla n_j^++ \sqrt{\kappa}\text{div} u_j^+ = \dot{\Delta}_jF_1-[\dot{\Delta}_j,v^+\cdot\nabla]n^+, \\
&\partial_t   u_j^++   v^+\cdot\nabla   u_j^+  -  \Delta u_j^+-\nabla\operatorname{div} u_j^+ + \frac{\beta_1}{\sqrt{\kappa}} \nabla  n_j^+ +\frac{\beta_2}{\sqrt{\kappa}}\nabla  n_j^-  - \sqrt{\kappa} \nabla\Delta  n_j^+ = \dot{\Delta}_j F_2- [ \dot{\Delta}_j,v^+\cdot\nabla] u^+, \\
&\partial_t  n_j^-+v^-\cdot\nabla n_j^- + \sqrt{\kappa}\text{div}u_j^- = \dot{\Delta}_jF_3-[\dot{\Delta}_j,v^-\cdot\nabla]n^-, \\
&\partial_t   u_j^-  +v^-\cdot\nabla   u_j^-- \Delta u_j^--\nabla\operatorname{div} u_j^- + \frac{\beta_3}{\sqrt{\kappa}} \nabla  n_j^+ +\frac{\beta_4}{\sqrt{\kappa}}\nabla   n_j^-  - \sqrt{\kappa} \nabla\Delta n_j^- =\dot{\Delta}_j F_4-[  \dot{\Delta}_j,v^-\cdot\nabla]u^-.
\end{aligned}
\right.
\end{equation}
Taking the $L^2$ inner product of $\eqref{sys1}_1$ and $\eqref{sys1}_3$ with $n_j^+$ and  $n_j^-$, respectively, we have
\begin{equation}\label{energy6}
    \begin{aligned}
       & \frac{{\rm d}}{{\rm d}t}\Big(  \frac{\beta_1}{2\beta_2\kappa}\|n_j^+\|_{L^2}^2 + \frac{\beta_4}{2\beta_3\kappa}\|n_j^-\|_{L^2}^2 \Big)\\
       &\qquad+ \frac{\beta_1}{\beta_2\sqrt{\kappa}} \int_{\mathbb{R}^d} \operatorname{div}   u_j^+ \, n_j^+ \, {\rm d}x+ \frac{\beta_4}{\beta_3\sqrt{\kappa}} \int_{\mathbb{R}^d} \operatorname{div}   u_j^- \, n_j^- \, {\rm d}x\\
        &=-\frac{\beta_1}{\beta_2\kappa} \int_{\mathbb{R}^d} \dot{\Delta}_j F_1 n_j^+   {\rm d}x-\frac{\beta_4}{\beta_3\kappa} \int_{\mathbb{R}^d} \dot{\Delta}_j  F_3 n_j^-   {\rm d}x -\frac{\beta_1}{\beta_2\kappa} \int_{\mathbb{R}^d}([\dot{\Delta}_j,v\cdot\nabla]n^++v\cdot\nabla n_j^+)n_j^+{\rm d}x\\
        &\quad-\frac{\beta_4}{\beta_3\kappa} \int_{\mathbb{R}^d}([\dot{\Delta}_j,v^+\cdot\nabla]n^-+v^+\cdot\nabla n_j^-)n_j^- {\rm d}x.
    \end{aligned}
\end{equation}
Since $\operatorname{div}v^\pm=0$, we have 
\begin{equation}\label{energy7}
    \begin{aligned}
        &-\kappa^{-\frac{1}{2}} \int_{\mathbb{R}^d}( \nabla n_j^- \cdot   u_j^+ + \nabla n_j^+ \cdot   u_j^- )\, {\rm d}x\\
       = &\kappa^{-\frac{1}{2}} \int_{\mathbb{R}^d}(  n_j^- \operatorname{div}  u_j^+ +  n_j^+ \operatorname{div}   u_j^- )\, {\rm d}x\\
        =&-\frac{1}{\kappa}\frac{{\rm d}}{{\rm d}t}\int_{\mathbb{R}^d} n_j^-n_j^+{\rm d}x+ \frac{1}{\kappa}\int_{\mathbb{R}^d}\Big(\dot{\Delta}_jF_3 n_j^++\dot{\Delta}_jF_1 n_j^- \Big){\rm d}x\\
        &\quad + \int_{\mathbb{R}^d}([\dot{\Delta}_j,v^+\cdot\nabla]n^++v^+\cdot\nabla n_j^+)n_j^-{\rm d}x+ \int_{\mathbb{R}^d}([\dot{\Delta}_j,v^-\cdot\nabla]n^-+v^-\cdot\nabla n_j^-)n_j^+{\rm d}x.
    \end{aligned}
\end{equation}
Imposing the gradient on the continuity equations of \eqref{sys1}, we have 
\begin{equation}\label{sys2}
\left\{
\begin{aligned}
&\partial_t \nabla n_j^+ + v^+\cdot \nabla^2 n_j^++\sqrt{\kappa}\nabla\operatorname{div}  u_j^+\\
&\quad=\nabla\dot{\Delta}_j(-\sqrt{\kappa}\widetilde{\psi}(\kappa^{-\frac{1}{2}}n^+)\text{div}u^+-(u^+-v^+)\cdot\nabla n^+)-[\Delta_j,v^+]\nabla^2 n^+,\\
&\partial_t \nabla n_j^- + v^-\cdot  \nabla^2  n_j^-+ \sqrt{\kappa}\nabla\operatorname{div}   u_j^- \\
&\quad= \nabla\dot{\Delta}_j(-\sqrt{\kappa}\widetilde{\psi}(\kappa^{-\frac{1}{2}}n^-)\text{div}u^--(u^--v^-)\cdot\nabla n^-)-[\Delta_j,v^+]\nabla^2 n^+.
\end{aligned}
\right.
\end{equation}
Taking the $L^2$ inner product of $(\ref{sys2})_1$ and  $(\ref{sys2})_2$ with $\nabla n_j^+$ and $\nabla n_j^-$, respectively, we have 
\begin{equation}\label{energy1}
    \begin{aligned}
        &\frac{1}{2} \frac{{\rm d}}{{\rm d}t}\Big( \frac{1}{\beta_2}\|\nabla n_j^+\|_{L^2}^2 +\frac{1}{\beta_3}\|\nabla n_j^-\|_{L^2}^2\Big)+  \frac{\sqrt{\kappa}}{\beta_2}\int_{\mathbb{R}^d} \Delta   u_j^+ \cdot \nabla n_j^+ \, {\rm d}x+ \frac{\sqrt{\kappa}}{\beta_3} \int_{\mathbb{R}^d} \Delta   u_j^- \cdot \nabla n_j^- \, {\rm d}x \\
        =&  \frac{1}{\beta_2}\int_{\mathbb{R}^d} \nabla \dot{\Delta}_j(-\sqrt{\kappa}\widetilde{\psi}(\kappa^{-\frac{1}{2}}n^+)\text{div}u^+-(u^+-v^+)\cdot\nabla n^+)\cdot \nabla n_j^+ \, {\rm d}x\\
         +& \frac{1}{\beta_3} \int_{\mathbb{R}^d} \nabla \dot{\Delta}_j(-\sqrt{\kappa}\widetilde{\psi}(\kappa^{-\frac{1}{2}}n^-)\text{div}u^--(u^--v^-)\cdot\nabla n^-) \cdot \nabla n_j^+ \, {\rm d}x\\
         - &\frac{1}{\beta_2}\int_{\mathbb{R}^d} (v^+\cdot \nabla^2 n_j^++[\Delta_j,v^+]\nabla^2 n^+)\cdot\nabla n_j^+{\rm d}x   - \frac{1}{\beta_3}\int_{\mathbb{R}^d} (v^-\cdot \nabla^2 n_j^-+[\Delta_j,v^-]\nabla^2 n^-)\cdot\nabla n_j^-{\rm d}x. 
    \end{aligned}
\end{equation}
Similarly, it follows from $(\ref{sys1})_2$ and  $(\ref{sys1})_4$ that
\begin{equation}\label{energy2}
    \begin{aligned}
        &\frac{1}{2} \frac{{\rm d}}{{\rm d}t}\Big( \frac{1}{\beta_2}\|  u_j^+\|_{L^2}^2+ \frac{1}{\beta_3}\| u_j^-\|_{L^2}^2 \Big) + \frac{1}{\beta_2}\|\nabla   u_j^+\|_{L^2}^2+\frac{1}{\beta_3}\|\nabla   u_j^-\|_{L^2}^2 + \frac{1}{\beta_2}\|\operatorname{div}   u_j^+\|_{L^2}^2+\frac{1}{\beta_3}\|\operatorname{div}   u_j^-\|_{L^2}^2\\
         -& \frac{\sqrt{\kappa}}{\beta_2}\int_{\mathbb{R}^d}\nabla\Delta n_j^+ \cdot    u_j^+\, {\rm d}x-  \frac{\sqrt{\kappa}}{\beta_3}\int_{\mathbb{R}^d}\nabla\Delta n_j^- \cdot    u_j^-\, {\rm d}x 
        + \frac{\beta_1}{\beta_2\sqrt{\kappa}} \int_{\mathbb{R}^d} \nabla n_j^+ \cdot   u_j^+ \, {\rm d}x+\frac{\beta_4}{\beta_3\sqrt{\kappa}} \int_{\mathbb{R}^d} \nabla n_j^- \cdot   u_j^- \, {\rm d}x\\
        =&-\frac{1}{\beta_2}\int_{\mathbb{R}^d} (v^+\cdot\nabla u_j^++[ \dot{\Delta}_j,v^+\cdot\nabla]u_j^+)\cdot  u_j^+{\rm d}x -\frac{1}{\beta_3}\int_{\mathbb{R}^d} (v^-\cdot\nabla u_j^-+[ \dot{\Delta}_j,v^-\cdot\nabla]u_j^-)\cdot   u_j^-{\rm d}x\\
        & +\frac{1}{\beta_2} \int_{\mathbb{R}^d} \dot{\Delta}_j F_2 \cdot   u_j^+ \, {\rm d}x+ \frac{1}{\beta_3}\int_{\mathbb{R}^d} \dot{\Delta}_j F_4 \cdot   u_j^- \, {\rm d}x-\kappa^{-\frac{1}{2}} \int_{\mathbb{R}^d} \nabla n_j^- \cdot   u_j^+ \, {\rm d}x-\kappa^{-\frac{1}{2}} \int_{\mathbb{R}^d} \nabla n_j^+ \cdot   u_j^- \, {\rm d}x.
    \end{aligned}
\end{equation}

Next, to capture the dissipation of  $\nabla n^\pm$ and $\Delta n^\pm$,  one takes the $L^2$ inner product of  $  (\ref{sys2})_1$ and  $ (\ref{sys1})_2$ with $\frac{1}{\beta_2\sqrt{\kappa}}   u_j^+$ and $ \frac{1}{\beta_2\sqrt{\kappa}}  \nabla  n_j^+$, respectively, and makes use of $(\ref{sys2})$:
\begin{equation}\label{energy3}
    \begin{aligned}
& \kappa^{-\frac{1}{2}}\frac{{\rm d}}{{\rm d}t} \int_{\mathbb{R}^d} \Big( \frac{1}{\beta_2}  u_j^+ \cdot \nabla n_j^++ \frac{1}{\beta_3} u_j^- \cdot \nabla n_j^-\Big) \, {\rm d}x 
-\frac{1}{\beta_2} \|\nabla   u_j^+\|_{L^2}^2-\frac{1}{\beta_3} \|\nabla   u_j^-\|_{L^2}^2 \\
&\qquad +\frac{2}{\beta_2\sqrt{\kappa}}\int_{\mathbb{R}^d}\nabla  u_j^+ \cdot \nabla^2 n_j^+ {\rm d}x +\frac{2}{\beta_2\sqrt{\kappa}}\int_{\mathbb{R}^d} \nabla  u_j^- \cdot \nabla^2 n_j^- {\rm d}x\\
&\qquad+\frac{\beta_1}{\beta_2\kappa}\|\nabla n_j^+\|_{L^2}^2 + \frac{\beta_4}{\beta_3\kappa}\|\nabla n_j^-\|_{L^2}^2 
+\frac{2}{\kappa}\int_{\mathbb{R}^d}  \nabla \dot{\Delta}_j n^- \cdot  \nabla \dot{\Delta}_j n^+\, {\rm d}x\\
&\qquad+ \frac{1}{\beta_2}\|\Delta n_j^+\|_{L^2}^2+\frac{1}{\beta_3}\|\Delta n_j^-\|_{L^2}^2 \\
&\quad= \kappa^{-\frac{1}{2}} \int_{\mathbb{R}^d} \nabla \dot{\Delta}_j(-\sqrt{\kappa}\widetilde{\psi}(\kappa^{-\frac{1}{2}}n^+)\text{div}u^+-(u^+-v^+)\cdot\nabla n^+)\cdot   u_j^+ \, {\rm d}x 
 \\
&\qquad+\kappa^{-\frac{1}{2}} \int_{\mathbb{R}^d} \nabla \dot{\Delta}_j(-\sqrt{\kappa}\widetilde{\psi}(\kappa^{-\frac{1}{2}}n^-)\text{div}u^--(u^--v^-)\cdot\nabla n^-)\cdot   u_j^- \, {\rm d}x \\
&\qquad+ \frac{1}{\beta_2\sqrt{\kappa}}\int_{\mathbb{R}^d}\dot{\Delta}_j  F_2 \cdot \nabla n_j^+ \, {\rm d}x+ \frac{1}{\beta_3\sqrt{\kappa}}\int_{\mathbb{R}^d}\dot{\Delta}_j  F_4 \cdot \nabla n_j^- \, {\rm d}x.
\end{aligned}
\end{equation}
Notice that the quasi-linear terms $-\sqrt{\kappa}\widetilde{\psi}(\kappa^{-\frac12}n^\pm)\operatorname{div}u^\pm$  and $\sqrt{\kappa}\nabla(\widetilde{\psi}(\kappa^{-\frac12}n^\pm) \Delta n^\pm)$ satisfy
\begin{equation}\label{energy4}
    \begin{aligned}
        &-\sqrt{\kappa}\int_{\mathbb{R}^d} \dot{\Delta}_j\nabla (\widetilde{\psi}(\kappa^{-\frac12}n^\pm)\operatorname{
        div}u^\pm)\cdot \nabla n_j^\pm {\rm d}x\\ &\quad=\sqrt{\kappa}\int_{\mathbb{R}^d}[\dot{\Delta}_j,\widetilde{\psi}(\kappa^{-\frac12}n^\pm)]\operatorname{div}u^\pm \Delta n^\pm_j{\rm d}x+\sqrt{\kappa}\int_{\mathbb{R}^d} \widetilde{\psi}(\kappa^{-\frac12}n^\pm)\operatorname{div} u_j^\pm \Delta n_j^\pm {\rm d}x
    \end{aligned}
\end{equation}
and 
\begin{equation}\label{energy5}
    \begin{aligned}
        &\sqrt{\kappa}\int_{\mathbb{R}^d} \dot{\Delta}_j\nabla (\widetilde{\psi}(\kappa^{-\frac12}n^\pm) \Delta n^\pm)\cdot\nabla  u_j^\pm {\rm d}x\\
        &\quad=-\sqrt{\kappa}\int_{\mathbb{R}^d}[\dot{\Delta}_j,\widetilde{\psi}(\kappa^{-\frac12}n^\pm)]\Delta n^\pm\operatorname{div} u_j^\pm {\rm d}x-\sqrt{\kappa}\int_{\mathbb{R}^d}\widetilde{\psi}(\kappa^{-\frac12}n^\pm)\operatorname{div} u_j^\pm \Delta n^\pm {\rm d}x,
    \end{aligned}
\end{equation}
which exhibit  a cancellation between \eqref{energy1} and \eqref{energy3}.

For any $j\in\mathbb{Z}$, we now define the Lyapunov functional
\begin{align*}
    \mathcal{E}_j(t)&=\frac{\beta_1}{2\beta_2}\|\kappa^{-\frac{1}{2}}n_j^+\|_{L^2}^2 + \frac{\beta_4}{2\beta_3}\|\kappa^{-\frac{1}{2}}n_j^-\|_{L^2}^2 +\frac{1}{\kappa}\int_{\mathbb{R}^d}n_j^+ n^-_j {\rm d}x+ \frac{1}{2\beta_2}\|\nabla n_j^+\|_{L^2}^2 +  \frac{1}{2\beta_3}\|\nabla n_j^-\|_{L^2}^2\nonumber\\
    &+\frac{1}{2\beta_2}\|  u_j^+\|_{L^2}^2+ \frac{1}{2\beta_3}\|  u_j^-\|_{L^2}^2+ \frac{\delta_1}{\sqrt{\kappa}} \int_{\mathbb{R}^d} \left( \frac{1}{\beta_2}  u_j^+ \cdot \nabla n_j^++ \frac{1}{\beta_3}  u_j^-\cdot \nabla n_j^- \, \right){\rm d}x\\.
\end{align*}
and observe the facts
\begin{equation}
    \begin{aligned}
       &\frac{\beta_1}{\beta_2}
       \|\kappa^{-\frac{1}{2}}n_j^+\|_{L^2}^2 + \frac{\beta_4}{\beta_3}\|\kappa^{-\frac{1}{2}}n_j^-\|_{L^2}^2 +\frac{2}{\kappa}\int_{\mathbb{R}^d}n_j^+ n^-_j {\rm d}x\\
       =&\left\|\sqrt{\frac{\beta_1}{\beta_2}}\kappa^{-\frac{1}{2}}n_j^++\sqrt{\frac{\beta_2}{\beta_1}}\kappa^{-\frac{1}{2}}n_j^-\right\|^2_{L^2}+\frac{\beta_1\beta_4-\beta_2\beta_3}{\beta_1\beta_3}\|\kappa^{-\frac{1}{2}} n^-_j\|^2_{L^2}\label{2.20}
    \end{aligned}
\end{equation}
and
\begin{equation}
    \begin{aligned}
       &\frac{\beta_1}{\beta_2}\|\nabla n_j^+\|_{L^2}^2 + \frac{\beta_4}{\beta_3}\|\nabla n_j^-\|_{L^2}^2 +\frac{2}{\kappa}\int_{\mathbb{R}^d}  \nabla n^-_j \cdot  \nabla n^+_j\, {\rm d}x\\
       =&\left\|\sqrt{\frac{\beta_1}{\beta_2}}\nabla  n_j^++\sqrt{\frac{\beta_2}{\beta_1}}\nabla n_j^-\right\|^2_{L^2}+\frac{\beta_1\beta_4-\beta_2\beta_3}{\beta_1\beta_3}\|\nabla n^-_j \|^2_{L^2}.\label{2.21}
    \end{aligned}
\end{equation}
In order to promise the positivity of the functional $\mathcal{E}_j$,  we shall use the stability condition \eqref{stability} such that
$\beta_1,\beta_2,\beta_3,\beta_4>0$ and $\beta_1\beta_4>\beta_2\beta_3$. 
Then, choosing a uniform, sufficiently small constant $\delta_1>0$, we can check that \eqref{ineqforbasicenergy0} and \eqref{ineqforbasicenergy1} hold true.
\end{proof}

Building on the  Lyapunov functional estimates obtained in Lemma \ref{lemma21}, we establish the following uniform energy-dissipation estimates. 

\begin{lemma}\label{lemma22}
Under the conditions of Proposition \ref{prop1}, it holds that
\begin{equation}\label{ineqforbasicenergy_0}
\begin{aligned}
            \mathcal{E}_t
            &\leq C e^{C\mathcal{V}_t}\bigg( \mathcal{E}_0 +\mathcal{W}_t\mathcal{E}_t+ (1+\mathcal{E}_t)^{\frac{d}{2}+1}\big(\mathcal{E}_t \mathcal{D}_t+\kappa^{-\frac{1}{2}}\mathcal{E}_t^2\big) \bigg),
        \end{aligned}
    \end{equation}
    where $C>0$ is a constant independent of $t$ and $\kappa$.
\end{lemma}

\begin{proof}
Dividing  \eqref{ineqforbasicenergy1} by $\sqrt{\mathcal{E}_j(t)+\varepsilon_*}$ with $\varepsilon_*>0$, integrating the resulting inequality over  $[0,t]$,and then letting  $\varepsilon_*\rightarrow0$, we obtain
\begin{equation}\label{ineqforbasicenergy2}
    \begin{aligned}  \sqrt{\mathcal{E}_j(t)}+2^{2j}\int_0^t\sqrt{\mathcal{E}_j(\tau)}\, {\rm d}\tau&\lesssim\sqrt{\mathcal{E}_j(0)} +\int_0^t \mathcal{R}_j \, {\rm d}\tau.
    \end{aligned}
\end{equation}
In view of \eqref{ineqforbasicenergy0}-\eqref{ineqforbasicenergy1}, this yields
    \begin{align}\label{ineqforbasicenergy3}
        &\left\|\left(\kappa^{-\frac{1}{2}}n^+,\kappa^{-\frac{1}{2}}n^-,\nabla n^+,\nabla n^-,   u^+, u^+\right)\right\|_{L^\infty_t(\dot{B}_{2,1}^{\frac{d}{2}-1})\cap L^1_t(\dot{B}_{2,1}^{\frac{d}{2}+1})}\nonumber \\
        &\quad \lesssim \left\|\left(\kappa^{-\frac{1}{2}}n_0^+,\kappa^{-\frac{1}{2}}n_0^-,\nabla n_0^+,\nabla n_0^-,   u_0^+,  u_0^+\right)\right\|_{\dot{B}_{2,1}^{\frac{d}{2}-1} }+ \| (u^+-v^+)\cdot\nabla n^+\|_{L^1_t(\dot{B}_{2,1}^{\frac{d}{2}})}\nonumber\\
          &\quad \quad + \| (u^--v^-)\cdot\nabla n^-\|_{L^1_t(\dot{B}_{2,1}^{\frac{d}{2}})}
      +\bigg\| \Big(\kappa^{-\frac{1}{2}}F_1,\kappa^{-\frac{1}{2}} F_3, \sum_{l=1}^6 F_2^l ,\sum_{l=1}^6 F_4^l\Big)\bigg\|_{ L^1_t(\dot{B}_{2,1}^{\frac{d}{2}-1})}\\
      &\quad \quad  +\int_0^t \|(v^+,v^-)\|_{\dot{B}_{2,1}^{\frac{d}{2}+1}} \|(\kappa^{-\frac{1}{2}}n^+,\kappa^{-\frac{1}{2}}n^-,\nabla n^+,\nabla n^-,   u^+,  u^+ )\|_{\dot{B}_{2,1}^{\frac{d}{2}-1}}{\rm d}\tau \nonumber\\
       &\quad\quad +\sqrt{\kappa}\int_0^t\sum_{j\in \mathbb{Z}}2^{j\frac{d}{2}}\|[\dot{\Delta}_j, \widetilde{\psi}(\kappa^{-\frac12}n^+)](\operatorname{div}u^++\Delta n^+)\|_{L^2}{\rm d}\tau\nonumber\\
       &\qquad+\sqrt{\kappa}\int_0^t\sum_{j\in \mathbb{Z}}2^{j\frac{d}{2}}\|[\dot{\Delta}_j, \widetilde{\psi}(\kappa^{-\frac12}n^-)](\operatorname{div}u^-+\Delta n^-)\|_{L^2}{\rm d}\tau\nonumber.
    \end{align}
Applying Gr\"onwall inequality then gives 
    \begin{align}\label{estimatefor-DS}
        \mathcal{E}_t &\leq C e^{\widetilde{C}\mathcal{V}_t}\bigg( \left\|\left(\kappa^{-\frac{1}{2}}n_0^+,\kappa^{-\frac{1}{2}}n_0^-,\nabla n_0^+,\nabla n_0^-,   u_0^+, u_0^+\right)\right\|_{\dot{B}_{2,1}^{\frac{d}{2}-1} }+\| (u^+-v^+)\cdot\nabla n^+\|_{L^1_t(\dot{B}_{2,1}^{\frac{d}{2}})}\nonumber\\
          &\quad\quad\quad\quad\quad\quad\quad+\| (u^--v^-)\cdot\nabla n^-\|_{L^1_t(\dot{B}_{2,1}^{\frac{d}{2}})}+\Big\| \big(\kappa^{-\frac{1}{2}}F_1,\kappa^{-\frac{1}{2}}F_3, \sum_{l=1}^6 F_2^l ,\sum_{l=1}^6 F_4^l\big)\Big\|_{ L^1_t(\dot{B}_{2,1}^{\frac{d}{2}-1})}\nonumber\\
      &\quad \quad\quad\quad\quad\quad\quad
        +\sqrt{\kappa}\int_0^t\sum_{j\in \mathbb{Z}}2^{j\frac{d}{2}}\|[\dot{\Delta}_j, \widetilde{\psi}(\kappa^{-\frac12}n^+)](\operatorname{div}u^++\Delta n^+)\|_{L^2}{\rm d}\tau \\
        &\quad \quad\quad\quad\quad\quad\quad+\sqrt{\kappa}\int_0^t\sum_{j\in \mathbb{Z}}2^{j\frac{d}{2}}\|[\dot{\Delta}_j, \widetilde{\psi}(\kappa^{-\frac12}n^-)](\operatorname{div}u^-+\Delta n^-)\|_{L^2}{\rm d}\tau\bigg)\nonumber. 
    \end{align}

We now turn to the nonlinear terms in \eqref{estimatefor-DS}. We first bound the contributions involving $F_1$ and $F_3$. 
Since $p$  in \eqref{condi-p} satisfies $2\leq p\leq \frac{2d}{d-2}$ ($d\geq2$), and the product maps $L^2_t(\dot{B}^{\frac{d}{p}}_{p,1})\times L^2_t(\dot{B}^{\frac{d}{2}-1}_{2,1})$ into $L^1_t(\dot{B}^{\frac{d}{2}-1}_{2,1})$ (see \eqref{uv4}), we arrive at
\begin{equation*}
\begin{aligned}
    &\quad\kappa^{-\frac{1}{2}} \|(u^\pm-v^\pm)\cdot \nabla n^\pm\|_{L^1_t(\dot{B}_{2,1}^{\frac{d}{2}-1})}\\
    &\lesssim  \kappa^{-\frac{1}{2}} \|(\mathcal{P}u^\pm-v^\pm)\cdot \nabla n^\pm\|_{L^1_t(\dot{B}_{2,1}^{\frac{d}{2}-1})}+ \kappa^{-\frac{1}{2}} \|\mathcal{Q}u^\pm\cdot \nabla n^\pm\|_{L^1_t(\dot{B}_{2,1}^{\frac{d}{2}-1})}\\
    &\lesssim \|\mathcal{P}u^\pm-v^\pm\|_{L^2_t(\dot{B}_{p,1}^{\frac{d}{p}})}\| \kappa^{-\frac{1}{2}}\nabla n^\pm \|_{L^2_t(\dot{B}_{2,1}^{\frac{d}{2}-1})} +\|\mathcal{Q}u^\pm\|_{L^2_t(\dot{B}_{p,1}^{\frac{d}{p}})}\| \kappa^{-\frac{1}{2}}\nabla n^\pm \|_{L^2_t(\dot{B}_{2,1}^{\frac{d}{2}-1})}.
\end{aligned}
\end{equation*}
Since $\widetilde{\psi}(s)$ is   smooth  with $\widetilde{\psi}(0)=0$, by \eqref{uv4} and \eqref{F0}, we also get
\begin{equation*}
\begin{aligned}
       \|\widetilde{\psi}(\kappa^{-\frac12}n^\pm)\operatorname{div} u^\pm\|_{L^1_t(\dot{B}_{2,1}^{\frac{d}{2}-1})}&\lesssim \|\widetilde{\psi}(\kappa^{-\frac12}n^\pm)\|_{L^2(\dot{B}_{p,1}^{\frac{d}{p}}) }\|  u^\pm\|_{L^2_t(\dot{B}_{2,1}^{\frac{d}{2}})}\\
      &\lesssim \Big(1+\|n^\pm\|_{L^{\infty}_t(L^{\infty})}\Big)^{\frac{d}{2}+1}\kappa^{-\frac{1}{2}} \|n^\pm\|_{L^2_t(\dot{B}_{p,1}^{\frac{d}{p}}) }\ \|  u^\pm\|_{L^2_t(\dot{B}_{2,1}^{\frac{d}{2}})}.
\end{aligned}
\end{equation*}
Recalling that $L^{\infty}_t(\dot{B}^{\frac{d}{q}-1}_{q,1})\cap L^{1}_t(\dot{B}^{\frac{d}{q}+1}_{q,1})\hookrightarrow L^{2}_t(\dot{B}^{\frac{d}{q}}_{q,1})$ ($1\leq q\leq \infty)$, we conclude that
\begin{align}\label{N:F1F3}
\big\| \kappa^{-\frac{1}{2}}(F_1,F_3)\big\|_{ L^1_t(\dot{B}_{2,1}^{\frac{d}{2}-1})}\lesssim   (1+\mathcal{E}_t)^{\frac{d}{2}+1}\mathcal{E}_t \mathcal{D}_t +\mathcal{W}_t\mathcal{E}_t.
\end{align}
Similarly, by  \eqref{uv1},  
\begin{equation}\label{N:2}
\begin{aligned}
 &\quad\|(u^\pm-v^\pm)\cdot\nabla n^\pm\|_{ L^1_t(\dot{B}_{2,1}^{\frac{d}{2}})}\\
 &\lesssim \|u^\pm-v^\pm\|_{L_t^2(L^{\infty})}\|\nabla n^\pm\|_{L^2_t(\dot{B}^{\frac{d}{2}}_{2,1})}+\|\nabla n^\pm\|_{L_t^2(L^{\infty})} \|u^\pm-v^\pm\|_{L^2_t(\dot{B}^{\frac{d}{2}}_{2,1})}\\
 &\lesssim \Big(\|\mathcal{Q}u^\pm\|_{L^2_t(\dot{B}^{\frac{d}{p}}_{p,1})}+\|\mathcal{P}u^\pm-v^\pm\|_{L^2_t(\dot{B}^{\frac{d}{p}}_{p,1})}+\|n^\pm\|_{L^2_t(\dot{B}^{\frac{d}{p}+1}_{p,1})}\Big)\|(\nabla n^\pm, u^\pm)\|_{L^2_t(\dot{B}^{\frac{d}{2}}_{2,1})}\\
 &\lesssim   \mathcal{E}_t \mathcal{D}_t +\mathcal{W}_t\mathcal{E}_t.
\end{aligned}
\end{equation}

We now handle the terms involving $F_2=\sum\limits_{i=1}^6 F_2^i$ and $F_4=\sum\limits_{i=1}^6 F_2^i$. For  $F_2^1$ and $F_4^1$, by the product law \eqref{uv4}, 
\begin{equation*}
\begin{aligned}
    \|(u^\pm-v^\pm)\cdot\nabla u^\pm \|_{ L^1_t(\dot{B}_{2,1}^{\frac{d}{2}-1})}&\lesssim \|u^\pm-v^\pm\|_{L^2_t(\dot{B}_{p,1}^{\frac{d}{p}})} \|\nabla u^\pm\|_{L^2_t(\dot{B}_{2,1}^{\frac{d}{2}-1})}\\
    &\lesssim \Big(\|\mathcal{Q}u^\pm\|_{L^2_t(\dot{B}_{p,1}^{\frac{d}{p}})}+\|\mathcal{P}u^\pm-v^\pm\|_{L^2_t(\dot{B}_{p,1}^{\frac{d}{p}})}\Big) \|u^\pm\|_{L^2_t(\dot{B}_{2,1}^{\frac{d}{2}})}.
\end{aligned}
\end{equation*}
For   $F_2^2$ and $F_4^2$, using \eqref{uv2} and \eqref{F0},
\begin{equation*}
 \begin{aligned}
\Big\|\operatorname{div}\big(2\widetilde{\phi}(\kappa^{-\frac{1}{2}}n^\pm)\mathbb{D}u^\pm\big)\Big\|_{L^1_t(\dot{B}_{2,1}^{\frac{d}{2}-1})}&\lesssim  \|\widetilde{\phi}(\kappa^{-\frac{1}{2}}n^\pm)\|_{L^\infty_t(\dot{B}_{2,1}^{\frac{d}{2}})} \|\mathbb{D}u^\pm\|_{L^1_t(\dot{B}_{2,1}^{\frac{d}{2}})}\\
&\lesssim \Big(1+\|n^\pm\|_{L^{\infty}_t(L^{\infty})}\Big)^{\frac{d}{2}+1} \kappa^{-\frac{1}{2}}\|n^\pm\|_{L^\infty_t(\dot{B}_{2,1}^{\frac{d}{2}})}\|u^\pm\|_{L^1_t(\dot{B}_{2,1}^{\frac{d}{2}+1})} .
\end{aligned}
\end{equation*}
Similarly, for  $F_2^3$ and $F_4^3$, 
 \begin{align*}
    \|\nu \widetilde{Q}( \kappa^{-\frac{1}{2}}n^\pm)\Delta u^\pm\|_{L^1_t(\dot{B}_{2,1}^{\frac{d}{2}-1})}\lesssim 
    \Big(1+\|n^\pm\|_{L^{\infty}_t(L^{\infty})}\Big)^{\frac{d}{2}+1} \kappa^{-\frac{1}{2}}\|n^\pm\|_{L^\infty_t(\dot{B}_{2,1}^{\frac{d}{2}})}\| u^\pm\|_{L^1_t(\dot{B}_{2,1}^{\frac{d}{2}+1})} .
\end{align*}
Using \eqref{uv4} and Lemma \eqref{lemma64} on the estimates for the multi-component composite functions $\widetilde{g}_1^+(\kappa^{-\frac{1}{2}}n^+,\kappa^{-\frac{1}{2}}n^-)$, we also have
\begin{align*}
   &\quad \|\widetilde{g}_1^+(\kappa^{-\frac{1}{2}}n^+,\kappa^{-\frac{1}{2}}n^-)\kappa^{-\frac{1}{2}} \nabla n^+\|_{L^1_t(\dot{B}_{2,1}^{\frac{d}{2}-1})}\\
  & \lesssim \Big(1+\|(n^+,n^-)\|_{L^{\infty}_t(L^{\infty})}\Big)^{\frac{d}{2}+1} \kappa^{-\frac{1}{2}}\|(n^+,n^-)\|_{L^2 (\dot{B}_{p,1}^{\frac{d}{p}})} \|\kappa^{-\frac{1}{2}}n^+\|_{L^2_t(\dot{B}_{2,1}^{\frac{d}{2}})}
\end{align*}
and
\begin{equation*}
    \begin{aligned}
       &\quad\| \widetilde{g}_2(\kappa^{-\frac{1}{2}}n^+,\kappa^{-\frac{1}{2}}n^{-})\nabla \kappa^{-\frac{1}{2}}n^{-}\|_{L^1_t(\dot{B}_{2,1}^{\frac{d}{2}-1})}\\
       &\lesssim \Big(1+\|(n^+,n^-)\|_{L^{\infty}_t(L^{\infty})}\Big)^{\frac{d}{2}+1}  \kappa^{-\frac{1}{2}}\|(n^+,n^-)\|_{L^2_t (\dot{B}_{p,1}^{\frac{d}{p}})} \|\kappa^{-\frac{1}{2}}n^{-}\|_{L^2_t(\dot{B}_{2,1}^{\frac{d}{2}})}.
    \end{aligned}
\end{equation*}
For the nonlinear term $F_2^6$ and $F_4^6$, we similarly obtain
\begin{align*}
    \|\nabla(|\nabla n^\pm|^2\|_{L^1_t(\dot{B}_{2,1}^{\frac{d}{2}-1})}\lesssim \|\nabla n^\pm\|_{L_t^2(\dot{B}_{p,1}^{\frac{d}{p}})}\|\nabla n^\pm\|_{L^2_t(\dot{B}_{2,1}^{\frac{d}{2}})}.
\end{align*}
Collecting these bounds, we infer
\begin{align}\label{N:F2F4}
    \Big\| \big(\sum_{l=1}^6 F_2^l ,\sum_{l=1}^6 F_4^l\big)\Big\|_{L^1_t(\dot{B}_{2,1}^{\frac{d}{2}-1})}&\lesssim \mathcal{W}_t\mathcal{E}_t+  (1+\mathcal{E}_t)^{\frac{d}{2}+1}(\mathcal{E}_t \mathcal{D}_t+\kappa^{-\frac{1}{2}}\mathcal{E}_t^2).
\end{align}
For the commutator terms,     Lemmas \ref{Lemma5-7} and \ref{lemcommutator} give
\begin{equation}\label{N:com}
    \begin{aligned}
        &\sqrt{\kappa}\int_0^t\sum_{j\in \mathbb{Z}}2^{j\frac{d}{2}}\|[\dot{\Delta}_j, \widetilde{\psi}(\kappa^{-\frac12}n^\pm)](\operatorname{div} u^\pm+\Delta n^\pm)\|_{L^2}{\rm d}\tau\\
        \lesssim & \Big(1+\|n^\pm\|_{L^{\infty}_t(L^{\infty})}\Big)^{\frac{d}{2}+1} \|\nabla n^\pm\|_{L^2_t(\dot{B}^{\frac{d}{p}}_{p,1})} \|(\nabla n^\pm, u^\pm )\|_{L^2_t(\dot{B}_{2,1}^{\frac{d}{2}})}\\
      \lesssim &  (1+\mathcal{E}_t)^{\frac{d}{2}+1}\mathcal{D}_t\mathcal{E}_t.
    \end{aligned}
\end{equation}
Combining  \eqref{estimatefor-DS}-\eqref{N:com}, we finally obtain  \eqref{ineqforbasicenergy_0}, which completes the proof of Lemma \ref{lemma22}.
\end{proof}

\subsection{Dispersive estimates of \texorpdfstring{$(n^\pm,\mathcal Q u^\pm)$}{(n±,Qu±)} }

We are in a position to capture the dispersive effect and establish the uniform estimates for the functional   $\mathcal{D}_t$.

\begin{lemma}\label{lemma23}
Under the conditions of Proposition \ref{prop1}, it holds that
\begin{equation}\label{es:dissersive}
\begin{aligned}
            \mathcal{D}_t
            &\lesssim \kappa^{-\delta}\Big(\mathcal{E}_0 + \mathcal{E}_t+ (1+\mathcal{E}_t)^{\frac{d}{2}+1} \mathcal{E}_t^2\Big) .
        \end{aligned}
    \end{equation}
\end{lemma}

\begin{proof}
Recall that $(z_1^+,z_2^-)$, defined in \eqref{Z1Z2}, satisfies  the following matrix–valued Schr\"odinger system: 
\begin{equation}\label{eq:Z1Z2}
    \begin{aligned}
   i\partial_t  \left(\begin{array}{c}z_1^+\\z_2^-\end{array} \right)+\mathcal{H}   \left(\begin{array}{c}z_1^+\\z_2^-\end{array} \right)=\left(\begin{array}{c}  -i\nu H_+\Lambda N_1^{+}+iU_+^{-1}\nabla G_1^+ -G_2^+  \\-i\nu H_-\Lambda N_2^{-} +iU_-^{-1}\nabla G_1^- -G_2^-\end{array} \right),
    \end{aligned}
\end{equation}
where  $$\mathcal{H}=\left(\begin{matrix}
   -2i\nu\Delta-\sqrt{\kappa}H_+&0\\
    0&-2i\nu\Delta-\sqrt{\kappa}H_-
\end{matrix}\right). $$
By Duhamel’s formula applied to \eqref{eq:Z1Z2}, we obtain  
\begin{equation}\label{DSC}
    \begin{aligned}
        \left(\begin{array}{c}z_1^+\\z_2^-\end{array} \right)=e^{i\mathcal{H}t}\left(\begin{array}{c}z_1^+\\z_2^-\end{array} \right)(0)+i\int_0^t e^{i\mathcal{H}(t-\tau)}\left(\begin{array}{c}  -i\nu H_+\Lambda N_1^{+}+iU_+^{-1}\nabla G_1^+ -G_2^+  \\-i\nu H_-\Lambda N_2^{-} +iU_-^{-1}\nabla G_1^- -G_2^-\end{array} \right)(\tau){\rm d}\tau.
    \end{aligned}
\end{equation}
Applying the dispersive estimates of Lemma \ref{LemA.2}, with $(\delta, \,p)$ satisfying \eqref{condi-p}, to \eqref{DSC} yields
\begin{equation}\label{estimatefor_z}
    \begin{aligned}
        \|(z_1^+,z_2^-)\|_{\widetilde{L}^2_t(\dot{B}_{p,1}^{\frac{d}{p}})}\lesssim\kappa^{-\delta}\left( \|(z_1^+,z_2^-)(0)\|_{\dot{B}_{2,1}^{\frac d2-1}}+\|(H_+\Lambda N_1^{+},H_-\Lambda N_2^{-} )\|_{L^1_t(\dot{B}_{2,1}^{\frac d2-1})}\right.\\
        \left.+\|(U_1^{-1}\nabla G_1^+ ,G_2^+)\|_{L^1_t(\dot{B}_{2,1}^{\frac d2-1})}+\|(U_1^{-1}\nabla G_1^- ,G_2^-) \|_{L^1_t(\dot{B}_{2,1}^{\frac d2-1})} \right).
    \end{aligned}
\end{equation}
Recalling the definitions of $N_1^+$, $N_2^-$ and $H_\pm$ (see \eqref{N+N-M+M-} and  \eqref{UH}) and  noting that $H_\pm\Lambda$ is a Fourier multiplier of degree $2$, we obtain
\begin{align}
    \|(H_+\Lambda N_1^{+},H_-\Lambda N_2^{-} )\|_{L^1_t(\dot{B}_{2,1}^{\frac d2-1})}\lesssim\left\|\left(\kappa^{-\frac{1}{2}}n^+,\kappa^{-\frac{1}{2}}n^-,\nabla n^+,\nabla n^-, \mathcal{Q} u^+,\mathcal{Q} u^-\right)\right\|_{L^1_t(\dot{B}_{2,1}^{\frac{d}{2}+1})}.
\end{align}
Since the Fourier symbol of $U^{-1}_\pm\nabla$ is bounded by $\mathcal{O}(\kappa^{-\frac{1}{2}}+|\xi|)$, we also deduce
\begin{equation}
\begin{aligned}
   &\quad \|(U_+^{-1}\nabla G_1^+ ,G_2^+)\|_{L^1_t(\dot{B}_{2,1}^{\frac d2-1})}+\|(U_-^{-1}\nabla G_1^- ,G_2^-) \|_{L^1_t(\dot{B}_{2,1}^{\frac d2-1})}\\
   &\lesssim \kappa^{-\frac{1}{2}} \| (F_1-v^+\cdot\nabla n^+,F_3-v^-\cdot\nabla n^-)\|_{L^1_t(\dot{B}^{\frac{d}{2}-1}_{2,1})}+\| (F_1-v^+\cdot\nabla n^+,F_3-v^-\cdot\nabla n^-)\|_{L^1_t(\dot{B}_{2,1}^{\frac d2})}\\
   &\quad+\| (F_2-v^+\cdot\nabla u^+,F_4-v^-\cdot\nabla u^-)\|_{L^1_t(\dot{B}^{\frac{d}{2}-1}_{2,1})}.
\end{aligned}
\end{equation}

Using the definition of  $F_1$ and $F_3$, and invoking \eqref{uv2} ($p=2, s_1=\frac{d}{2}-1, s_2=\frac{d}{2}$) together with \eqref{F0}, we obtain
\begin{equation}
\begin{aligned}
&\quad\kappa^{-\frac{1}{2}} \| (F_1-v^+\cdot\nabla n^+,F_3-v^-\cdot\nabla n^-)\|_{L^1_t(\dot{B}^{\frac{d}{2}-1}_{2,1})}\\
&\lesssim \kappa^{-\frac{1}{2}} \|(u^+\cdot\nabla n^+,u^-\cdot\nabla n^-)\|_{L^1_t(\dot{B}^{\frac{d}{2}-1}_{2,1})}+ \|(\widetilde{\psi}(\kappa^{-\frac12}n^+)\operatorname{div} u^+,\widetilde{\psi}(\kappa^{-\frac12}n^-)\operatorname{div} u^-)\|_{L^1_t(\dot{B}_{2,1}^{\frac{d}{2}-1})}\\
&\lesssim \|(u^+,u^-)\|_{L^{\infty}_t(\dot{B}^{\frac{d}{2}-1}_{2,1})}\kappa^{-\frac{1}{2}}\|(\nabla n^+,\nabla n^-)\|_{L^1_t(\dot{B}^{\frac{d}{2}}_{2,1})}\\
&\quad+\Big(1+\|n^\pm\|_{L^{\infty}_t(L^{\infty})}\Big)^{\frac{d}{2}+1} \kappa^{-\frac{1}{2}} \|(n^+,n^-)\|_{L^{\infty}_t(\dot{B}^{\frac{d}{2}-1}_{2,1})} \| (\operatorname{div} u^+,\operatorname{div} u^-)\|_{L^1_t(\dot{B}^{\frac{d}{2}}_{2,1})}\\
&\lesssim  (1+\mathcal{E}_t)^{\frac{d}{2}+1}\mathcal{E}_t^2.
\end{aligned}
\end{equation}
Similarly,
\begin{equation}
\begin{aligned}
 &\quad \| (F_1-v^+\cdot\nabla n^+,F_3-v^-\cdot\nabla n^-)\|_{L^1_t(\dot{B}^{\frac{d}{2}}_{2,1})}\\
 &\lesssim \|(u^+,u^-)\|_{L^{\infty}_t(\dot{B}^{\frac{d}{2}-1}_{2,1}}\|(\nabla n^+,\nabla n^-)\|_{L^1_t(\dot{B}^{\frac{d}{2}+1}_{2,1})}\\
&\quad+ \Big(1+\|(n^+,n^-)\|_{L^{\infty}_t(L^{\infty})}\Big)^{\frac{d}{2}+1} \|(n^+,n^-)\|_{L^{\infty}_t(\dot{B}^{\frac{d}{2}}_{2,1})} \| (\operatorname{div} u^+,\operatorname{div} u^-)\|_{L^1_t(\dot{B}^{\frac{d}{2}}_{2,1})}\lesssim (1+\mathcal{E}_t)^{\frac{d}{2}+1}\mathcal{E}_t^2.
\end{aligned}
\end{equation}
In addition, since $\kappa\geq1$, it follows from \eqref{uv2}, \eqref{F0} and the composite estimates (\eqref{F1:m} for $d\geq3$ and \eqref{F1:m00} for $d=2$) that
    \begin{align}\label{238}
    &\quad\| (F_2-v^+\cdot\nabla u^+,F_4-v^-\cdot\nabla u^-)\|_{L^1_t(\dot{B}^{\frac{d}{2}-1}_{2,1})}\nonumber\\
    &\lesssim \|(u^+,u^-)\|_{L^{\infty}_t(\dot{B}^{\frac{d}{2}-1}_{2,1}}\|( u^+, u^-)\|_{L^1_t(\dot{B}^{\frac{d}{2}+1}_{2,1})}\nonumber\\
    &\quad+(1+\|(n^+,n^-)\|_{L^{\infty}_t(L^{\infty})})^{\frac{d}{2}+1} \|(n^+,n^-)\|_{L^{\infty}_t(\dot{B}^{\frac{d}{2}}_{2,1})}\|( u^+, u^-)\|_{L^1_t(\dot{B}^{\frac{d}{2}+1}_{2,1})}\nonumber\\
    &\quad+\Big(1+\|(n^+,n^-)\|_{L^{\infty}_t(L^{\infty})}\Big)^{\frac{d}{2}+1} \kappa^{-\frac{1}{2}}\|(n^+,n^-)\|_{L^2_t (\dot{B}_{2,1}^{\frac{d}{2}})} \|\kappa^{-\frac{1}{2}}n^+\|_{L^2_t(\dot{B}_{2,1}^{\frac{d}{2}})}\\
    &\quad+\Big(1+\|n^\pm\|_{L^{\infty}_t(L^{\infty})}\Big)^{\frac{d}{2}+1} \|\nabla n^\pm\|_{L^2_t(\dot{B}^{\frac{d}{2}}_{2,1})} \|(\nabla n^\pm, u^\pm )\|_{L^2_t(\dot{B}_{2,1}^{\frac{d}{2}})}\nonumber\\
    &\lesssim  (1+\mathcal{E}_t)^{\frac{d}{2}+1} \mathcal{E}_t^2\nonumber.
    \end{align}

Finally,  collecting    \eqref{estimatefor_z}-\eqref{238} together     using the definitions of $z_1^+$ and $z_2^-$, we obtain 
\begin{equation*}
    \begin{aligned}
&\quad\left\|\left(\sqrt{\frac{r_+\kappa^{-1}-\Delta}{-\Delta}}\nabla n^+,\mathcal{Q}u^+\right)\right\|_{\widetilde{L}^2_t(\dot{B}_{p,1}^{\frac{d}{p}})}+\left\|\left(\sqrt{\frac{r_-\kappa^{-1}-\Delta}{-\Delta}}\nabla n^-,\mathcal{Q}u^-\right)\right\|_{\widetilde{L}^2_t(\dot{B}_{p,1}^{\frac{d}{p}})}\\
&\lesssim  \|( {\rm{Re}}\,z_1^+,{\rm{Re}}\, z_2^-,{\rm{Im}}\,z_1^+,{\rm{Im}}\, z_2^-)\|_{\widetilde{L}^2_t(\dot{B}_{p,1}^{\frac{d}{p}})}\\
&\lesssim  \|( z_1^+,z_2^-)\|_{\widetilde{L}^2_t(\dot{B}_{p,1}^{\frac{d}{p}})}\lesssim \kappa^{-\delta}\Big(\mathcal{E}_0 + \mathcal{E}_t+ (1+\mathcal{E}_t)^{\frac{d}{2}+1} \mathcal{E}_t^2\Big).
    \end{aligned}
\end{equation*} 
Noting that $r_\pm$ given by \eqref{lambdapm} are strictly positive and $(-\Delta)^{-\frac{1}{2}}$ is a Fourier multiplier of degree $-1$, we therefore have
$$
\kappa^{-\frac{1}{2}} \|n^\pm\|_{L^2_t(\dot{B}_{p,1}^{\frac{d}{p}})} +\|n^\pm\|_{L^2_t(\dot{B}_{p,1}^{\frac{d}{p}+1})}\lesssim 
\kappa^{-\frac{1}{2}} \|n^\pm\|_{\widetilde{L}^2_t(\dot{B}_{p,1}^{\frac{d}{p}})} +\|n^\pm\|_{\widetilde{L}^2_t(\dot{B}_{p,1}^{\frac{d}{p}+1})}\lesssim \left\|\sqrt{\frac{r_\pm\kappa^{-1}-\Delta}{-\Delta}}\nabla n^\pm\right\|_{\widetilde{L}^2_t(\dot{B}_{p,1}^{\frac{d}{p}})},$$
which in turn implies \eqref{es:dissersive}.\end{proof}

\subsection{Estimates for the incompressible error  \texorpdfstring{$\mathcal{P}u^{\pm}-v^\pm$}{Pu±-v±}}

The remainder is to establish the estimates for the error $\mathcal{P}u^{\pm}-v^\pm$.

\begin{lemma}\label{lemma24}
Under the conditions of Proposition \ref{prop1}, it holds that
\begin{equation}\label{2.39}
\begin{aligned}
            \mathcal{W}_t
            &\lesssim e^{C(\mathcal{V}_t+\mathcal{E}_t^2)}\left( \|(\mathcal {P} u^+_{0}-v^+_0,\mathcal {P} u^-_{0}-v^-_0)\|_{\dot{B}_{2,1}^{\frac{d}{2}-1}}  + (1+\mathcal{E}_t)^{\frac{d}{2}+1}\big(\mathcal{D}_t(\mathcal{E}_t +\mathcal{V}_t) +\kappa^{-\frac{1}{2}}\mathcal{E}_t^2\big)\right).
        \end{aligned}
    \end{equation}
\end{lemma}

\begin{proof}
Note that the incompressible part of $\mathcal{P}u$ satisfies
\begin{equation}\label{Pupm}
\left\{
\begin{aligned}
&\partial_t \mathcal{P}u^++ \mathcal{P}(v^+\cdot\nabla \mathcal{P}u^+)  -  \Delta \mathcal{P}u^+ = -\mathcal{P}(v^+\cdot\nabla \mathcal{Q}u^+)+\mathcal{P}(F_2^1+F_2^2+F_2^3+F_2^4+F_2^5 ), \\
&\partial_t \mathcal{P}u^- +  \mathcal{P}(v^-\cdot\nabla \mathcal{P}u^-)- \Delta \mathcal{P}u^- = - \mathcal{P}(v^-\cdot\nabla \mathcal{Q}u^-)+\mathcal{P}(F_4^1+F_4^2 +F_4^3+F_4^4+F_4^5). \\
\end{aligned}
\right.
\end{equation}
where we used $\mathcal{P}F^l_2=\mathcal{P}F^l_4=0$ ($l=6,7$). Combining \eqref{INS} with \eqref{Pupm}, the error  $\widetilde{u}^\pm:=\mathcal{P}u^\pm-v^{\pm}$ of the incompressible part of two fluids solves 
\begin{equation}\label{errorV}
\left\{
\begin{aligned}
&\partial_t \widetilde{u}^++ \mathcal{P}(v^+\cdot\nabla \mathcal{P}\widetilde{u}^+)-  \Delta \widetilde{u}^+  = \mathcal{P}\bar{F}_2 , \\
&\partial_t \widetilde{u}^-+  \mathcal{P}(v^-\cdot\nabla \mathcal{P}\widetilde{u}^-)  - \Delta  \widetilde{u}^- =  \mathcal{P}\bar{F}_4  
\end{aligned}
\right.
\end{equation}
with
\begin{equation*}
\left\{
\begin{aligned}
    \bar{F}_2 &:=- v^+\cdot\nabla \mathcal{Q}u^++F_2^1+F_2^2+F_2^3+F_2^4+F_2^5,\\
    \bar{F}_4 &:= -v^-\cdot\nabla \mathcal{Q}u^- +  F_4^1+F_4^2+F_4^3+F_4^4+F_4^5.
\end{aligned}
\right.
\end{equation*}
Using  $\mathcal{P}={\rm Id}-\mathcal{Q}$ and the  orthogonality
of $\mathcal{Q}$ and $\mathcal{P}$, and performing $L^2$ energy estimates on \eqref{errorV}, we get
\begin{equation}
    \begin{aligned}
      & \quad \frac{1}{2}\frac{{\rm d}}{{\rm d}t}\|(\widetilde{u}_j^+,\widetilde{u}_j^-)\|_{L^2}^2+c2^{2j}\|(\widetilde{u}_j^+,\widetilde{u}_j^-)\|_{L^2}^2 \\
      &\lesssim \int_{\mathbb{R}^d}([\dot{\Delta}_j,v^+\cdot\nabla]\widetilde{u}^++v^\pm\cdot\nabla \widetilde{u}_j^+)\cdot \widetilde{u}_j^+ {\rm d}x+\int_{\mathbb{R}^d}([\dot{\Delta}_j,v^-\cdot\nabla]\widetilde{u}^-+v\cdot\nabla \widetilde{u}_j^-)\cdot \widetilde{u}_j^- {\rm d}x   \\
        &\quad +\|\widetilde{u}_j^+ \|_{L^2} \|\dot{\Delta}_j  \mathcal{P}\bar{F}_2\|_{L^2} +\|\widetilde{u}_j^- \|_{L^2} \|\dot{\Delta}_j \mathcal{P}\bar{F}_4\|_{L^2}.
    \end{aligned}
\end{equation}
Invoking $\div v^\pm=0$ and the commutator estimates in Lemma \ref{lemcommutator}, we infer
\begin{equation*}
    \begin{aligned}
&\quad\|(\widetilde{u}^+,\widetilde{u}^-)\|_{L^\infty_t(\dot{B}_{2,1}^{\frac{d}{2}-1})\cap L^1_t(\dot{B}_{2,1}^{\frac{d}{2}+1})}  \lesssim \|(\mathcal {P} u^+_{0}-v^+_0,\mathcal {P} u^-_{0}-v^-_0)\|_{\dot{B}_{2,1}^{\frac{d}{2}-1}} \\
& \qquad+ \int_0^t \|(v^+,v^-)\|_{\dot{B}_{2,1}^{\frac{d}{2}+1}}\|(\widetilde{u}^+,\widetilde{u}^-)\|_{\dot{B}_{2,1}^{\frac{d}{2}-1}}{\rm d}\tau +\|  \mathcal{P}\bar{F}_2\|_{L^1_t(\dot{B}_{2,1}^{\frac{d}{2}-1})} +\|\mathcal{P}\bar{F}_4\|_{L^1_t(\dot{B}_{2,1}^{\frac{d}{2}-1})}.
\end{aligned}
\end{equation*}
Hence, by Gr\"onwall's inequality,
\begin{equation}\label{estimatefor-W}
    \begin{aligned}
    \mathcal{W}_t\lesssim   e^{C\mathcal{V}_t}\Big(\|(\mathcal {P} u^+_{0}-v^+_0,\mathcal {P} u^-_{0}-v^-_0)\|_{\dot{B}_{2,1}^{\frac{d}{2}-1}} +\|  \mathcal{P}\bar{F}_2\|_{L^1_t(\dot{B}_{2,1}^{\frac{d}{2}-1})} +\|\mathcal{P}\bar{F}_4\|_{L^1_t(\dot{B}_{2,1}^{\frac{d}{2}-1})}\Big).
\end{aligned}
\end{equation}
For the nonlinear terms in \eqref{estimatefor-W}, using \eqref{uv4}, we have
\begin{align*}
\|\mathcal{P}(v^\pm\cdot\nabla \mathcal{Q}u^\pm)\|_{L^1_t(\dot{B}_{2,1}^{\frac{d}{2}-1})}\lesssim \|v^\pm\|_{L^2_t(\dot{B}_{2,1}^{\frac{d}{2}})}\|\nabla \mathcal{Q}u^\pm\|_{L^2_t(\dot{B}_{p,1}^{\frac{d}{p}-1})}\lesssim  \mathcal{D}_t \mathcal{V}_t.
\end{align*}
Moreover, we need to adjust the estimate of $(u^\pm-v^\pm)\cdot \nabla u^\pm$ in $F_2$ and $F_4$. One can deduce from previous product laws and Young's inequality that
\begin{equation*}
    \begin{aligned}
     \|(u^\pm-v^\pm)\cdot\nabla u^\pm \|_{ L^1_t(\dot{B}_{2,1}^{\frac{d}{2}-1})}&\lesssim  \|\mathcal{Q}u^\pm\|_{L^2_t(\dot{
B}_{p,1}^{\frac{d}{p}})}\|\nabla u^\pm\|_{L^2_t(\dot{B}_{2,1}^{\frac{d}{2}})}+\int_{0}^t \|\widetilde{u}^\pm\|_{\dot{B}_{2,1}^{\frac{d}{2}}} \|u^\pm\|_{\dot{B}_{2,1}^{\frac{d}{2}}}d\tau\\
     &\lesssim  \mathcal{D}_t\mathcal{E}_t+\var_0\|\widetilde{u}^\pm\|_{L^1_t(\dot{B}_{2,1}^{\frac{d}{2}+1})}+\var_0^{-1}  \int_0^t  \|u^\pm\|_{\dot{B}_{2,1}^{\frac{d}{2}}}^2\|\widetilde{u}^\pm\|_{\dot{B}_{2,1}^{\frac{d}{2}-1}}d\tau.
\end{aligned}
\end{equation*}
for some constant $\var_0>0$ to be chosen later.  Other terms in $F_2$ and $F_4$ can be addressed by repeating the computations as in Lemma \ref{lemma22}. So we have
\begin{align*}
    \|\mathcal{P}\bar{F}_2\|_{L^1_t(\dot{B}_{2,1}^{\frac{d}{2}-1})}&\lesssim 
    \var_0\|\widetilde{u}^+\|_{L^1_t(\dot{B}_{2,1}^{\frac{d}{2}+1})} +\var_0^{-1}  \int_0^t  \|u^+\|_{\dot{B}_{2,1}^{\frac{d}{2}}}^2\|\widetilde{u}^+\|_{\dot{B}_{2,1}^{\frac{d}{2}-1}}d\tau   +(1+\mathcal{E}_t)^{\frac{d}{2}+1}\big(\mathcal{E}_t \mathcal{D}_t+\kappa^{-\frac{1}{2}}\mathcal{E}_t^2\big),
\end{align*}
and
\begin{align*}
    \|\mathcal{P}\bar{F}_4\|_{L^1_t(\dot{B}_{2,1}^{\frac{d}{2}-1})}\lesssim \var_0 \|\widetilde{u}^-\|_{L^1_t(\dot{B}_{2,1}^{\frac{d}{2}+1})} +\var_0^{-1}  \int_0^t  \|u^-\|_{\dot{B}_{2,1}^{\frac{d}{2}}}^2\|\widetilde{u}^-\|_{\dot{B}_{2,1}^{\frac{d}{2}-1}}d\tau   +(1+\mathcal{E}_t)^{\frac{d}{2}+1}\big(\mathcal{E}_t \mathcal{D}_t+\kappa^{-\frac{1}{2}}\mathcal{E}_t^2\big).
\end{align*}
Consequently, collecting the above estimates, letting $\var_0$ be sufficiently small and taking advantage of Gr\"onwall's inequality, we conclude \eqref{2.39} and complete the proof of Lemma \ref{lemma24}.\end{proof}

With the above lemmas in hand, we can immediately complete the proof of Proposition \ref{prop1}.

\subsection{Proof of Theorem \ref{thm:global-existence}}

\subsubsection*{Existence}

To prove the global existence, one can construct a local Friedrichs approximation (cf. \cite{bahouri1}),  extend the local approximate sequence to a global one by the uniform-in-time a priori estimates  
established in Proposition \ref{prop1}, and then show the convergence of the approximate sequence to the expected global solution to the Cauchy problem for \eqref{system-perturb1}. More precisely, consider the approximate problems ($q\geq1$):
\begin{equation} \label{sys-Friedrich}
\left\{
\begin{aligned}
&\partial_t n^{+}+ \dot{\mathbb{E}}_q(u^+\cdot \nabla n^{+})+\dot{\mathbb{E}}_q ((1+n^{+})\operatorname{div} u^{+})=0, \\
&\partial_t u^++ \dot{\mathbb{E}}_q(u^+\cdot \nabla u^{+} )+ \frac{\beta_1}{\sqrt{\kappa }} \nabla \dot{\mathbb{E}}_qn^+ + \frac{\beta_2}{\sqrt{\kappa }} \dot{\mathbb{E}}_q \nabla n^- - \Delta \dot{\mathbb{E}}_qu^+ -\nabla\operatorname{div}\dot{\mathbb{E}}_qu^+-\sqrt{\kappa }\nabla \Delta \dot{\mathbb{E}}_qn^+ =\dot{\mathbb{E}}_q F_2, \\
&\partial_t n^{-}+ \dot{\mathbb{E}}_q(u^-\cdot \nabla n^{-})+\dot{\mathbb{E}}_q ((1+n^{-})\operatorname{div}  u^{-})=0, \\
&\partial_t u^- +\dot{\mathbb{E}}_q(u^-\cdot \nabla u^{-})+ \frac{\beta_3 }{\sqrt{\kappa}}\nabla \dot{\mathbb{E}}_q n^+ + \frac{\beta_4}{\sqrt{\kappa }} \nabla\dot{\mathbb{E}}_q n^- -\Delta \dot{\mathbb{E}}_q u^- -\nabla\operatorname{div}\dot{\mathbb{E}}_q u^- - \sqrt{\kappa} \nabla \Delta\dot{\mathbb{E}}_q n^- = \dot{\mathbb{E}}_qF_4,\\
&(\dot{\mathbb{E}}_q n^+,\dot{\mathbb{E}}_qu^+,\dot{\mathbb{E}}_q n^-,\dot{\mathbb{E}}_qu^-)(x,0)=(\dot{\mathbb{E}}_q n_0^+,\dot{\mathbb{E}}_qu_0^+,\dot{\mathbb{E}}_q n_0^-,\dot{\mathbb{E}}_qu^-_0).
\end{aligned}
\right.
\end{equation}
Since \eqref{sys-Friedrich} is a system of ordinary differential equations  in 
$L^2_q\times L^2_q\times L^2_q\times L^2_q $ and is locally Lipschitz with respect to the variable $( n^\pm, u^\pm)$ for every $q\geq 1$.  By virtue of the Cauchy-Lipschitz theorem 
(see \cite[Theorem 3.2]{bahouri1}),  there exists a maximal time $T_q>0$ such that the problem \eqref{sys-Friedrich} admits a unique solution $(n^{\pm,q}, u^{\pm,q})\in C([0,T_q);L^2_q)$. Then, $(n^{\pm,q}, u^{\pm,q})$   solves \eqref{sys-Friedrich} on $[0,T_q)\times \mathbb{R}^d$.

We claim $T_q=\infty$ and that there exists a large enough $\kappa_0$, such that for all $\kappa\geq \kappa_0$ and $q\geq1$, it holds
    \begin{align}\label{253}
    \sup_{t\in\mathbb{R}^+}\mathcal{E}_t\leq \mathcal{N}_1,\quad  \sup_{t\in\mathbb{R}^+}\mathcal{D}_t \leq \kappa^{-\delta}\mathcal{N}_2\quad\text{and}\quad  \sup_{t\in\mathbb{R}^+}\mathcal{W}_{t} \leq \kappa^{-\delta } \mathcal{N}_3
\end{align}
    with 
    \begin{align*}
    \mathcal{N}_1:&= C_0 e^{C_0\sup\limits_{t\in\mathbb{R}_+}\mathcal{V}_{t}}(\mathcal{E}_0
    +2),\\
   \mathcal{N}_2:&= C_0  \Big(\mathcal{E}_0 + \mathcal{N}_1+(1+\mathcal{N}_1)^{\frac{d}{2}+1}\mathcal{N}_1^2\Big),\\
\mathcal{N}_3:&=C_0 e^{C_0\sup\limits_{t\in\mathbb{R}^+}\mathcal{V}_t+C_0\mathcal{N}_1^2}    \bigg( 1+ (1+\mathcal{N}_1)^{\frac{d}{2}+1}(\mathcal{N}_1\mathcal{N}_2+\mathcal{N}_1^2) \bigg).
    \end{align*}
where $\mathcal{E}_t$, $\mathcal{D}_t$ and $\mathcal{W}_t$ are defined in \eqref{definition-E}, \eqref{definition-D} and \eqref{definition-W}, respectively, and $C_0$ is given in Proposition \ref{prop1}.

We now justify the uniform regularity estimates \eqref{253}. Define the maximal time
\begin{align*}
    T^*_q:&=\sup\left\{t\leq T_q: \mathcal{E}_{t}\leq \mathcal{N}_1,~ \mathcal{D}_{t}\leq \kappa^{-\delta}\mathcal{N}_2~\text{and}~\mathcal{W}_t\leq \kappa^{-\delta}\mathcal{N}_3\right\}.
\end{align*}
Employing $\delta<\frac{1}{2}$, the first estimate \eqref{pri-1} in Proposition \ref{prop1}, for all $t\in (0,T^*_q)$, we have
  \begin{equation} \label{mao1}
        \begin{aligned}
            &\mathcal{E}_{t}\leq C_0e^{C_0\mathcal{V}_t}\bigg(  \mathcal{E}_0+\mathcal{W}_t\mathcal{E}_t+ (1+\mathcal{E}_t)^{\frac{d}{2}+1}\big(\mathcal{E}_t \mathcal{D}_t+\kappa^{-\frac{1}{2}}\mathcal{E}_t^2\big) \Bigg)\\
               &\leq            
C_0e^{C_0\sup\limits_{t\in\mathbb{R}^+}\mathcal{V}_t}    \bigg(     \mathcal{E}_0+ \kappa^{-\delta} \mathcal{N}_1 \mathcal{N}_3+\kappa^{-\delta}(1+\mathcal{N}_1)^{\frac{d}{2}+1}\mathcal{N}_1 \mathcal{N}_2+\kappa^{-\frac{1}{2}}(1+\mathcal{N}_1)^{\frac{d}{2}+1}\mathcal{N}_1^2\bigg)\\
&\leq C_0e^{C_0\sup\limits_{t\in\mathbb{R}^+}\mathcal{V}_t}    \bigg(     \mathcal{E}_0+1\bigg)<\mathcal{N}_1,
        \end{aligned}
    \end{equation}
provided 
    \begin{align}\label{kappa0}
    \kappa\geq \kappa_0:=\min\left\{ 1,\frac{1}{(3\mathcal{N}_1\mathcal{N}_3)^{\frac{1}{\delta}}} ,  \frac{1}{(1+\mathcal{N}_1)^{\frac{1}{\delta}(\frac{d}{2}+1)}(3\mathcal{N}_1\mathcal{N}_2)^{\frac{1}{\delta}}},\frac{1}{(1+\mathcal{N}_1)^{2(\frac{d}{2}+1)}(3\mathcal{N}_1^2)^{2}}  \right\}.
    \end{align}
Consequently, using $\mathcal{W}_t\leq \kappa^{-\delta}\mathcal{N}_3, \mathcal{D}_t\leq \kappa^{-\delta}\mathcal{N}_2$ in the definition of $T_q^*$ and $\mathcal{D}_t<\mathcal{N}_1$ which have been obtained, we deduce from the assumption \eqref{avelocity} and the estimates \eqref{pri-21}-\eqref{pri-2} that
\begin{equation}\label{mao2}
    \begin{aligned}
  \mathcal{D}_t\leq \kappa^{-\delta} C_0  \Big(\mathcal{E}_0 + \mathcal{E}_t+ (1+\mathcal{E}_t)^{\frac{d}{2}+1} \mathcal{E}_t^2\Big)< C_0  \kappa^{-\delta}\Big(\mathcal{E}_0 + \mathcal{N}_1+(1+\mathcal{N}_1)^{\frac{d}{2}+1}\mathcal{N}_1^2\Big)=\kappa^{-\delta}\mathcal{N}_2
    \end{aligned}
\end{equation}
and
\begin{equation} \label{mao3}
    \begin{aligned}
        \mathcal{W}_{t} &\leq C_0e^{C_0\mathcal{V}_t+C_0\mathcal{E}^2_t}\left(\|(\mathcal {P} u^+_{0}-v^+_0,\mathcal {P} u^-_{0}-v^-_0)\|_{\dot{B}^{\frac{d}{2}-1}_{2,1}} + (1+\mathcal{E}_t)^{\frac{d}{2}+1}\big(\mathcal{E}_t \mathcal{D}_t+\kappa^{-\frac{1}{2}}\mathcal{E}_t^2\big) \right)\\
        &< \kappa^{-\delta}C_0 e^{C_0\sup\limits_{t\in\mathbb{R}^+}\mathcal{V}_t+C_0\mathcal{N}_1^2}    \bigg( 1+ (1+\mathcal{N}_1)^{\frac{d}{2}+1}(\mathcal{N}_1\mathcal{N}_2+\mathcal{N}_1^2) \bigg)=\kappa^{-\delta}\mathcal{N}_3. 
    \end{aligned}
\end{equation}
By \eqref{mao1}, \eqref{mao2} and a standard continuity argument, we have $T_q^*=T_q$.

Furthermore, by using the uniform estimates \eqref{253} in the definition of $T_q^*=T_q$ and the Cauchy-Lipschitz theorem, we can extend the existence time beyond $T_q$, which leads to a contradiction regarding the maximality of $T_{q}$. Therefore, we deduce $T_{q}=\infty$, so $(n^{\pm,q}, u^{\pm,q})$ to the problem \eqref{sys-Friedrich} exists globally and fulfills the uniform estimates \eqref{253}.

 By a standard compactness process (see \cite{bahouri1}, pp. 442–444), one can obtain a limit $(n^\pm,u^\pm)$ of a subsequence $(n^{\pm,q_k},u^{\pm,q_k})$ with $q_{k}\rightarrow\infty$ such that the approximate problems \eqref{sys-Friedrich} for $(n^{\pm,q_k},u^{\pm,q_k})$ converge to the Cauchy problem \eqref{system-perturb1} in the sense of distributions. By virtue of the uniform estimates \eqref{253} and the Fatou property, $(n^\pm,u^\pm)$ is indeed a global strong solution to the problem \eqref{system-perturb2} and satisfies the properties \eqref{eq:incompressible-limit}. This completes the global existence part.

\subsubsection*{Uniqueness}

Finally, we prove uniqueness. It suffices to work with the reformulated system \eqref{reformulat1}. Without loss of generality, as the parameters do not affect our argument for proving uniqueness, we set $\nu=\kappa=1$.
Fix $T>0$ and let $(n_1^\pm,u_1^\pm)$ and $(n_2^\pm,u_2^\pm)$ be two solutions satisfying \eqref{12455}-\eqref{eq:incompressible-limit} stated in Theorem~\ref{thm:global-existence}, on $[0,T]\times\mathbb{R}^d$
with the same initial data $(n_0^\pm,u_0^\pm)$.
Set the differences
\[
(\widetilde{n}^+,\widetilde{n}^-,\widetilde{u}^+,\widetilde{u}^-)
:=(n_1^+-n_2^+,\,n_1^- - n_2^-,\,u_1^+-u_2^+,\,u_1^- - u_2^-).
\]
Then $(\widetilde{n}^\pm,\widetilde{u}^\pm)$ solve
\begin{equation}\label{3.84e}
\left\{
\begin{aligned}
&\partial_t \widetilde{n}^++ \sqrt{\kappa} \operatorname{div}\widetilde{u}^+ = \widetilde{F}_1,\\
&\partial_t \widetilde{n}^- + \sqrt{\kappa} \operatorname{div}\widetilde{u}^- = \widetilde{F}_3,\\
&\partial_t \widetilde{u}^+ + \frac{\beta_1}{\sqrt{\kappa} }\nabla \widetilde{n}^+ + \frac{\beta_2}{\sqrt{\kappa} }\nabla \widetilde{n}^-
   - \Delta \widetilde{u}^+ - \nabla\operatorname{div}\widetilde{u}^+ - \sqrt{\kappa} \nabla\Delta \widetilde{n}^+ = \widetilde{F}_2,\\
&\partial_t \widetilde{u}^- + \frac{\beta_3}{\sqrt{\kappa} }\nabla \widetilde{n}^+ + \frac{\beta_4}{\sqrt{\kappa} }\nabla \widetilde{n}^-
   - \Delta \widetilde{u}^- - \nabla\operatorname{div}\widetilde{u}^- - \sqrt{\kappa} \nabla\Delta \widetilde{n}^- = \widetilde{F}_4,\\
&(\widetilde{n}^+,\widetilde{n}^-,\widetilde{u}^+,\widetilde{u}^-)(x,0)=(0,0,0,0),
\end{aligned}
\right.
\end{equation}
where 
\begin{equation*}
    \left\{
    \begin{aligned}
        \widetilde{F}_1&=-u_2^+\cdot\nabla\widetilde{n}^+ -\widetilde{u}^+\cdot\nabla n_1^+
 -\sqrt{\kappa} \widetilde{\psi}(\kappa^{-\frac{1}{2}}n_2^+)\,\operatorname{div}\widetilde{u}^+
 -\sqrt{\kappa} \big(\widetilde{\psi}(\kappa^{-\frac{1}{2}}n_2^+)-\widetilde{\psi}(\kappa^{-\frac{1}{2}}n_1^+)\big)\operatorname{div}u_1^+,\\[2pt]
 \widetilde{F}_3&=-u_2^-\cdot\nabla\widetilde{n}^- -\widetilde{u}^-\cdot\nabla n_1^-
 -\sqrt{\kappa} \widetilde{\psi}(\kappa^{-\frac{1}{2}}n_2^-)\,\operatorname{div}\widetilde{u}^-
 -\sqrt{\kappa} \big(\widetilde{\psi}(\kappa^{-\frac{1}{2}}n_2^-)-\widetilde{\psi}(\kappa^{-\frac{1}{2}}n_1^-)\big)\operatorname{div}u_1^-,\\[2pt]
    \end{aligned}
    \right.
\end{equation*}
and 
\begin{equation*}
\left\{
\begin{aligned}
\widetilde{F}_2&=-u_2^+\cdot\nabla\widetilde{u}^+ -\widetilde{u}^+\cdot\nabla u_1^+ 
 +(1+\widetilde{Q}(\kappa^{-\frac{1}{2}}n_2^+))\operatorname{div}\!\big(2\widetilde{\phi}(\kappa^{-\frac{1}{2}}n_2^+)\mathbb{D}u_2^+\big)
 -(1+\widetilde{Q}(\kappa^{-\frac{1}{2}}n_1^+))\operatorname{div}\!\big(2\widetilde{\phi}(\kappa^{-\frac{1}{2}}n_1^+)\mathbb{D}u_1^+\big)\\
&\quad +\widetilde{Q}(\kappa^{-\frac{1}{2}}n_2^+)(\Delta +\nabla\operatorname{div})u_2^+ -\widetilde{Q}(\kappa^{-\frac{1}{2}}n_1^+)(\Delta +\nabla\operatorname{div})u_1^+\\
&\quad -\widetilde{g}_1^+(\kappa^{-\frac{1}{2}}n_2^+,\kappa^{-\frac{1}{2}}n_2^-)\nabla\kappa^{-\frac{1}{2}} n_2^+
 +\widetilde{g}_1^+(\kappa^{-\frac{1}{2}}n_1^+,\kappa^{-\frac{1}{2}}n_1^-)\nabla \kappa^{-\frac{1}{2}}n_1^+\\
& -\widetilde{g}_1^+(\kappa^{-\frac{1}{2}}n_2^+,\kappa^{-\frac{1}{2}}n_2^-)\nabla \kappa^{-\frac{1}{2}}n_2^-
 +\widetilde{g}_1^+(\kappa^{-\frac{1}{2}}n_1^+,\kappa^{-\frac{1}{2}}n_1^-)\nabla \kappa^{-\frac{1}{2}}n_1^-\\
&\quad +\nabla\!\Big(\tfrac{1}{2}(\nabla n_1^+ +\nabla n_2^+)\cdot \nabla \widetilde{n}^+\Big)
 +\sqrt{\kappa}\nabla\!\big(\widetilde{\psi}(\kappa^{-\frac{1}{2}}n_2^+)\,\Delta \widetilde{n}^+\big)
 +\sqrt{\kappa}\nabla\!\big(\big[\widetilde{\psi}(\kappa^{-\frac{1}{2}}n_2^+) -\widetilde{\psi}(\kappa^{-\frac{1}{2}}n_1^+)\big]\Delta n_1^+\big),\\[2pt]
\widetilde{F}_4&=-u_2^-\cdot\nabla\widetilde{u}^- -\widetilde{u}^-\cdot\nabla u_1^- 
 +(1+\widetilde{Q}(\kappa^{-\frac{1}{2}}n_2^-))\operatorname{div}\!\big(2\widetilde{\phi}(\kappa^{-\frac{1}{2}}n_2^-)\mathbb{D}u_2^-\big)
 -(1+\widetilde{Q}(\kappa^{-\frac{1}{2}}n_1^-))\operatorname{div}\!\big(2\widetilde{\phi}(\kappa^{-\frac{1}{2}}n_1^-)\mathbb{D}u_1^-\big)\\
&\quad +\widetilde{Q}(\kappa^{-\frac{1}{2}}n_2^-)(\Delta +\nabla\operatorname{div})u_2^- -\widetilde{Q}(\kappa^{-\frac{1}{2}}n_1^-)(\Delta +\nabla\operatorname{div})u_1^-\\
&\quad -\widetilde{g}_3^+(\kappa^{-\frac{1}{2}}n_2^+,\kappa^{-\frac{1}{2}}n_2^-)\nabla\kappa^{-\frac{1}{2}} n_2^+
 +\widetilde{g}_3^+(\kappa^{-\frac{1}{2}}n_1^+,\kappa^{-\frac{1}{2}}n_1^-)\nabla\kappa^{-\frac{1}{2}} n_1^+\\
 &-\widetilde{g}_4^+(\kappa^{-\frac{1}{2}}n_2^+,\kappa^{-\frac{1}{2}}n_2^-)\nabla \kappa^{-\frac{1}{2}}n_2^-
 +\widetilde{g}_4^+(\kappa^{-\frac{1}{2}}n_1^+,\kappa^{-\frac{1}{2}}n_1^-)\nabla n_1^-\\
&\quad +\nabla\!\Big(\tfrac{1}{2}(\nabla n_1^- +\nabla n_2^-)\cdot \nabla \widetilde{n}^-\Big)
 +\sqrt{\kappa}\nabla\!\big(\widetilde{\psi}(\kappa^{-\frac{1}{2}}n_2^-)\,\Delta \widetilde{n}^-\big)
 +\sqrt{\kappa}\nabla\!\big(\big[\widetilde{\psi}(\kappa^{-\frac{1}{2}}n_2^-) -\widetilde{\psi}(\kappa^{-\frac{1}{2}}n_1^-)\big]\Delta n_1^-\big).
\end{aligned}
\right.
\end{equation*}

Performing the localized energy estimate as in Lemma~\ref{lemma21}, we obtain a functional
\[
\widetilde{\mathcal{E}}_j(t):=\big\|(\kappa^{-\frac{1}{2}}\widetilde{n}_j^+,\kappa^{-\frac{1}{2}}\widetilde{n}_j^-,\nabla \widetilde{n}_j^+,\nabla\widetilde{n}_j^-,\widetilde{u}_j^+,\widetilde{u}_j^-)\big\|_{L^2}^2
\]
and a constant $\widetilde{c}_0>0$ such that
\begin{equation}\label{tildeE}
\frac{{\rm d}}{{\rm d}t}\,\widetilde{\mathcal{E}}_j(t)+\widetilde{c}_0\,2^{2j}\widetilde{\mathcal{E}}_j(t)
\;\lesssim\;
\widetilde{R}_j^*(t)\,\sqrt{\widetilde{\mathcal{E}}_j(t)},
\end{equation}
where
\begin{equation*}
\begin{aligned}
\widetilde{R}_j^*(t):=&\big\|\dot{\Delta}_j(\kappa^{-\frac{1}{2}}\widetilde{F}_1,\kappa^{-\frac{1}{2}}\widetilde{F}_3, \widetilde{F}_2-\sqrt{\kappa}\nabla\!\big(\widetilde{\psi}(\kappa^{-\frac{1}{2}}n_2^+)\,\Delta \widetilde{n}^+\big),\widetilde{F}_4-\sqrt{\kappa}\nabla\!\big(\widetilde{\psi}(\kappa^{-\frac{1}{2}}n_2^-)\,\Delta \widetilde{n}^-\big))\big\|_{L^2}\\
&\,+ 2^{j\frac{d}{2}}\| \sqrt{\kappa} \big(\widetilde{\psi}(\kappa^{-\frac{1}{2}}n_2^\pm)-\widetilde{\psi}(\kappa^{-\frac{1}{2}}n_1^\pm)\big)\operatorname{div}u_1^\pm\|_{L^2}\\
&+2^{j\frac{d}{2}}\|u_2^\pm\cdot\nabla\widetilde{n}^\pm+\widetilde{u}^\pm\cdot\nabla n_1^\pm\|_{L^2}+\,\sqrt{\kappa}2^{j\frac{d}{2}}\big\|[\dot{\Delta}_j, \widetilde{\psi}(\kappa^{-\frac{1}{2}}n_2^\pm)]\big(\operatorname{div} \widetilde{u}^\pm+\Delta \widetilde{n}^\pm\big)\big\|_{L^2}.
\end{aligned}
\end{equation*}

Summing \eqref{tildeE} over $j\in\mathbb{Z}$ and using the product/commutator laws in Besov spaces, we deduce (for $t\in[0,T_0]$, with $T_0$ small enough depending on the a~priori bounds of the two solutions)
\begin{align*}
&\big\|(\kappa^{-\frac{1}{2}}\widetilde{n}^+,\kappa^{-\frac{1}{2}}\widetilde{n}^-,\nabla \widetilde{n}^+,\nabla\widetilde{n}^-,\widetilde{u}^+,\widetilde{u}^-)\big\|_{\widetilde{L}^\infty_t(\dot{B}_{2,1}^{\frac{d}{2}-1})}
+\big\|(\kappa^{-\frac{1}{2}}\widetilde{n}^+,\kappa^{-\frac{1}{2}}\widetilde{n}^-,\nabla \widetilde{n}^+,\nabla\widetilde{n}^-,\widetilde{u}^+,\widetilde{u}^-)\big\|_{L^1_t(\dot{B}_{2,1}^{\frac{d}{2}+1})}\\
&\lesssim \int_0^{T_0} \big\| (\kappa^{-\frac{1}{2}}\widetilde{F}_1,\kappa^{-\frac{1}{2}}\widetilde{F}_3,\nabla \widetilde{F}_1,\nabla\widetilde{F}_3, \widetilde{F}_2,\widetilde{F}_4)\big\|_{\dot{B}_{2,1}^{\frac{d}{2}-1}}\,{\rm d}\tau 
 +\sqrt{\kappa}\int_0^{T_0} \sum_{j\in \mathbb{Z}}2^{j\frac{d}{2}}\big\|[\dot{\Delta}_j, \widetilde{\psi}(\kappa^{-\frac{1}{2}}n_2^\pm)]\big(\operatorname{div} \widetilde{u}^\pm+\Delta \widetilde{n}^\pm\big)\big\|_{L^2}\, {\rm d}\tau.
\end{align*}
Each right-hand term is bounded by the left-hand norm times an integrable coefficient built from the a-priori norms of $(n_i^\pm,u_i^\pm)$; for instance,
\begin{align*}
&\int_0^{T_0}\kappa^{-\frac{1}{2}}\|u_2^\pm\cdot\nabla\widetilde{n}^\pm+\widetilde{u}^\pm\cdot\nabla n_1^\pm\|_{\dot{B}_{2,1}^{\frac{d}{2}-1}}\,{\rm d}\tau\\
\leq & \epsilon \|\kappa^{-\frac{1}{2}}\widetilde{n}^\pm\|_{L^1_t(\dot{B}_{2,1}^{\frac{d}{2}+1})}+C_\epsilon\int_0^{T_0}\|u_2^\pm\|^2_{ \dot{B}_{2,1}^{\frac{d}{2}}}\|\kappa^{-\frac{1}{2}}\widetilde{n}^\pm\|_{ \dot{B}_{2,1}^{\frac{d}{2}-1}}{\rm d}\tau
 +C\int_0^{T_0}\|\kappa^{-\frac{1}{2}}n_1^\pm\|_{\dot{B}_{2,1}^{\frac{d}{2}+1}} \|\widetilde{u}^\pm\|_{\dot{B}_{2,1}^{\frac{d}{2}-1}}\,{\rm d}\tau,\\
&\int_0^{T_0} \|(\widetilde{\psi}(\kappa^{-\frac{1}{2}}n_2^\pm)-\widetilde{\psi}(\kappa^{-\frac{1}{2}}n_1^\pm))\operatorname{div}u_1^\pm\|_{\dot{B}_{2,1}^{\frac{d}{2}-1}}\,{\rm d}\tau\lesssim \int_0^{T_0} \|u_1^\pm\|_{\dot{B}_{2,1}^{\frac{d}{2}+1}} \|\kappa^{-\frac{1}{2}}\widetilde{n}^\pm\|_{\dot{B}_{2,1}^{\frac{d}{2}-1}}\,{\rm d}\tau,\\
&\int_0^{T_0}\|\nabla(\tfrac12(\nabla n_1^\pm+\nabla n_2^\pm)\cdot\nabla \widetilde{n}^\pm)\|_{\dot{B}_{2,1}^{\frac{d}{2}-1}}\,{\rm d}\tau\lesssim  \Big(\|\nabla n_1^\pm\|_{L^2_t(\dot{B}_{p,1}^{\frac{d}{p}})}+\|\nabla n_2^\pm\|_{L^2_t(\dot{B}_{p,1}^{\frac{d}{p}})} \Big)\,
\|\nabla \widetilde{n}^\pm\|_{ L^2_t(\dot{B}_{2,1}^{\frac{d}{2}})},\\
&\int_0^{T_0}\| \widetilde{Q}(\kappa^{-\frac{1}{2}}n_2^\pm)(\Delta+\nabla\operatorname{div})u_2^\pm-\widetilde{Q}(\kappa^{-\frac{1}{2}}n_1^\pm)(\Delta+\nabla\operatorname{div})u_1^\pm \|_{\dot{B}_{2,1}^{\frac{d}{2}-1}}\,{\rm d}\tau\\ \lesssim & \kappa^{-\frac{1}{2}} \|n_2^\pm\|_{\widetilde{L}^{\infty}_t(\dot{B}_{2,1}^{\frac{d}{2}})}\|\widetilde{u}^\pm\|_{L^1_t(\dot{B}_{2,1}^{\frac{d}{2}+1})}
 +\kappa^{-\frac{1}{2}}\| u_1^\pm\|_{L^1_t(\dot{B}_{2,1}^{\frac{d}{2}+1})}\|\widetilde{n}^\pm\|_{\widetilde{L}^{\infty}_t(\dot{B}_{2,1}^{\frac{d}{2}})},\\
&\sum_{j}2^{j\frac{d}{2}}\int_0^{T_0}\!\!\sqrt{\kappa}\big\|[\dot{\Delta}_j, \widetilde{\psi}(\kappa^{-\frac{1}{2}}n_2^\pm)]\big(\operatorname{div} \widetilde{u}^\pm+\Delta \widetilde{n}^\pm\big)\big\|_{L^2}\, {\rm d}\tau\\
\lesssim & \Big(1+\|n_2^\pm\|_{L^\infty_t(L^\infty)}\Big)^{\frac{d}{2}+1}\|\nabla n_2^\pm\|_{L^2_t(\dot{B}_{p,1}^{\frac{d}{p}})}
\,\|(\nabla \widetilde{n}^\pm, \widetilde{u}^\pm )\|_{L^2_t(\dot{B}_{2,1}^{\frac{d}{2}})}.
\end{align*}
Hence, by Gr\"onwall’s inequality on $[0,T_0]$ and the zero initial differences, and then taking $\kappa$ sufficiently large and $\epsilon$ suitably small, we get
\[
\big\|(\widetilde{n}^+,\widetilde{n}^-,\nabla \widetilde{n}^+,\nabla\widetilde{n}^-,\widetilde{u}^+,\widetilde{u}^-)\big\|_{\widetilde{L}^\infty_{T_0}(\dot{B}_{2,1}^{\frac{d}{2}-1})}
+\big\|(\widetilde{n}^+,\widetilde{n}^-,\nabla \widetilde{n}^+,\nabla\widetilde{n}^-,\widetilde{u}^+,\widetilde{u}^-)\big\|_{L^1_{T_0}(\dot{B}_{2,1}^{\frac{d}{2}+1})}
=0.
\]
Repeating the argument on $[T_0,2T_0],[2T_0,3T_0],\dots$ covers $[0,T]$ and yields uniqueness.

\section{Proof of Theorem \ref{thm:decay-transfer}}\label{section3}

\subsection{Uniform evolution of Besov norms under lower-regularity exponents}
In this subsection, we derive uniform-in-time bounds for the solution in negative Besov norms. More precisely,  for any $t>0$
\begin{equation}
\biggl\|\Bigl(\kappa^{-\frac{1}{2}}n^+,\kappa^{-\frac{1}{2}}n^-, \nabla n^+,\nabla n^-, u^+,u^-\Bigr)(t)\biggr\|_{\dot B^{\sigma_1}_{2,\infty}}\label{3.1}
 \leq C_{0},
\end{equation}
where $C_{0}>0$ depends on the initial norm $\bigl\|(\kappa^{-\frac{1}{2}}n^\pm_0,\nabla n^\pm_{0},u^\pm_{0})\bigr\|_{\dot B^{ \sigma_1}_{2,\infty}}$ and on the regularity estimates for $(n^\pm,u^\pm)$ provided by Theorem \ref{thm:global-existence}. The key step is the following lemma, which describes the evolution of this norm.  

\begin{lemma}\label{Lem3.1}
Let $(n^\pm,u^\pm)$ be the solution provided by Theorem \ref{thm:global-existence}. 
Then, if $(\kappa^{-\frac{1}{2}}n^\pm_0,\nabla n^\pm_{0},u^\pm_{0})\in \dot B^{\sigma_1}_{2,\infty}$ for $-\frac d2 \leq \sigma_1< \frac d2 -1,$, it holds that
\begin{equation}\label{ineq-evo-1}
\begin{aligned}
&\biggl\|\Bigl(\kappa^{-\frac{1}{2}}n^+,\kappa^{-\frac{1}{2}}n^-,\nabla n^+,\nabla n^-,  u^+, u^-\Bigl)(t)\biggr\|_{\dot B^{\sigma_1}_{2,\infty}}\\
\lesssim&
{\rm exp}\bigg\{ \int_{0}^{t}
\Bigl(1 + \|(n^+,n^-)\|_{\dot B^{\frac d2}_{2,1}}\Bigr)^{\frac{d}{2}+1}
\biggl\|\Bigl(\kappa^{-\frac{1}{2}}n^+,\kappa^{-\frac{1}{2}}n^-,\nabla n^+,\nabla n^-,  u^+, u^+\Bigr)\biggl\|_{\dot{B}^{\frac{d}{2}+1}_{2,1}}\,{\rm d}\tau\bigg\}\\
&\qquad\qquad\qquad\qquad\qquad\qquad\qquad\times\bigl\|\bigl(\kappa^{-\frac{1}{2}}n_0^+,\kappa^{-\frac{1}{2}}n_0^-,\nabla n_0^+,\nabla n_0^-,   u_0^+, u_0^+\bigr)\bigr\|_{\dot B^{\sigma_1}_{2,\infty}}.
\end{aligned}
\end{equation}
\end{lemma}

\begin{proof}
Here, we do not incorporate the incompressible–limit component $v^\pm$ into the nonlinear terms to close the energy estimates. Instead, we consider the system
\begin{equation} \label{3.3}
\left\{
\begin{aligned}
&\partial_t n^+   +\sqrt{\kappa }\operatorname{div} u^+ = \widetilde{F}_1, \\
&\partial_t u^+  +\frac{\beta_1}{\sqrt{\kappa }} \nabla n^+ + \frac{\beta_2}{\sqrt{\kappa }} \nabla n^- - \Delta u^+ -\nabla\operatorname{div}u^+-\sqrt{\kappa }\nabla \Delta n^+ = \widetilde{F}_2, \\
&\partial_t n^-  + \sqrt{\kappa}\operatorname{div} u^- = \widetilde{F}_3, \\
&\partial_t u^-  + \frac{\beta_3 }{\sqrt{\kappa}}\nabla n^+ +  \frac{\beta_4}{\sqrt{\kappa }} \nabla n^- - \Delta u^--\nabla\operatorname{div}u^-  - \sqrt{\kappa} \nabla \Delta n^- = \widetilde{F}_4,
\end{aligned}
\right.
\end{equation}
where $\widetilde{F}_1=F_1-v^+\cdot\nabla n^+$, $\widetilde{F}_3=F_3-v^-\cdot\nabla n^-$, $\widetilde{F}_2=F_2-v^+\cdot\nabla u^+$ and $\widetilde{F}_4=F_4-v^-\cdot\nabla u^-$.

Recall that the functional $\mathcal{E}_j(t)$ is defined in Lemma \ref{lemma21} and obeys \eqref{ineqforbasicenergy0}. Proceeding from \eqref{3.3} as in Lemma \ref{lemma21}, we obtain 
\begin{equation} \label{ineqforbasicenergy3.3}
    \begin{aligned}
        \frac{{\rm d}}{{\rm d}t} \mathcal{E}_j(t)+c_0 2^{2j}\mathcal{E}_j(t)&\lesssim \widetilde{\mathcal{R}}_j\sqrt{\mathcal{E}_j(t)},
    \end{aligned}
\end{equation}
where  $\mathcal{E}_j(t)$ is denoted by \eqref{ineqforbasicenergy0}
and 
\begin{equation*}
\begin{aligned}
\widetilde{\mathcal{R}}_j:&=\Big\|\dot{\Delta}_j\big(\kappa^{-\frac{1}{2}}\widetilde{F}_1,\kappa^{-\frac{1}{2}}\widetilde{F}_3,\sum_{l=1}^6 \widetilde{F}_2^l ,\sum_{l=1}^6 \widetilde{F}_4^l\big)\Big\|_{L^2} +2^j\|\dot{\Delta}_j (u^+ \cdot\nabla n^+)\|_{L^2}+2^j\|\dot{\Delta}_j (u^-\cdot\nabla n^-)\|_{L^2}\\
        &\quad+2^j\sqrt{\kappa}\|[\dot{\Delta}_j, \widetilde{\psi}(\kappa^{-\frac12}n^+)](\operatorname{div}u^++\Delta n^+)\|_{L^2}+2^j\sqrt{\kappa}\|[\dot{\Delta}_j, \widetilde{\psi}(\kappa^{-\frac12}n^-)](\operatorname{div}u^-+\Delta n^-)\|_{L^2},
\end{aligned}
\end{equation*}

Dividing  \eqref{ineqforbasicenergy3.3} by $\sqrt{ \mathcal{E}_j(t)+\varepsilon_*}$ with $\varepsilon_*>0$, integrating over  $[0,t]$, and then letting $\varepsilon_*\rightarrow0$, we arrive at
\begin{equation}\label{ineq-evo-3}
    \begin{aligned}  \sqrt{\mathcal{E}_j(t)}+2^{2j}\int_0^t\sqrt{\mathcal{E}_j(\tau)}{\rm d}\tau&\lesssim\sqrt{\mathcal{E}_j(0)} +\int_0^t \widetilde{\mathcal{R}}_j {\rm d}\tau. 
    \end{aligned}
\end{equation}

Then, multiplying \eqref{ineq-evo-3} by $2^{j\sigma_1}$, and taking the supremum on $j\in \mathbb{Z} $ yields 
    \begin{align}\label{ineq-evo-4}
        &\left\|(\kappa^{-\frac{1}{2}}n^+,\kappa^{-\frac{1}{2}}n^-,\nabla n^+,\nabla n^-,   u^+,  u^-)(t)\right\|_{\dot B^{\sigma_1}_{2,\infty}}\nonumber\\ 
        \lesssim & \left\| (\kappa^{-\frac{1}{2}}n_0^+,\kappa^{-\frac{1}{2}}n_0^-,\nabla n_0^+,\nabla n_0^-,   u_0^+,  u_0^-)\right\|_{\dot{B}^{\sigma_1}_{2,\infty}}+\int_0^t  \bigg\| \Big(\kappa^{-\frac{1}{2}}\widetilde{F}_1,\kappa^{-\frac{1}{2}}\widetilde{F}_3,\sum_{l=1}^6 \widetilde{F}_2^l ,\sum_{l=1}^6 \widetilde{F}_4^l\Big)\bigg\|_{\dot{B}^{\sigma_1}_{2,\infty}}{\rm d}\tau\nonumber\\
        &+\int_0^t  \| u^+\cdot\nabla n^+\|_{\dot{B}^{\sigma_1+1}_{2,\infty}}{\rm d}\tau +\int_0^t  \| u^-\cdot\nabla n^-\|_{\dot{B}^{\sigma_1+1}_{2,\infty}}{\rm d}\tau\\
        &+\sqrt{\kappa} \int_0^t  \sup_{j\in\mathbb{Z}} 2^{(\sigma_1+1)j} \|\bigl[\dot\Delta_j,  \widetilde{\psi}(\kappa^{-\frac12}n^+)\bigr](\operatorname{div}u^++\Delta n^+)\|_{L^2} {\rm d}\tau\nonumber\\
        &+\sqrt{\kappa} \int_0^t  \sup_{j\in\mathbb{Z}} 2^{(\sigma_1+1)j} \|\bigl[\dot\Delta_j,  \widetilde{\psi}(\kappa^{-\frac12}n^-)\bigr](\operatorname{div}u^-+\Delta n^-)\|_{L^2} {\rm d}\tau\nonumber,
    \end{align}
 where we used Minkowski's inequality to exchange the $L^1_t$ integration and the supremum on $j\in\mathbb{Z}$.

Next, we focus on the nonlinear terms on the right-hand side of \eqref{ineq-evo-4}.
For the sake of a concise proof, we first present here some interpolation inequalities that will be employed, i.e.
\begin{align}\label{interpo1}
    \|(\kappa^{-\frac{1}{2}}n^\pm,\nabla n^\pm,u^\pm)\|_{\dot B^{\frac d2}_{2,1}}
 \leq 
\|(\kappa^{-\frac{1}{2}}n^\pm,\nabla n^\pm,u^\pm)\|_{\dot B^{\sigma_1}_{2,\infty}}^{ \eta}
\|(\kappa^{-\frac{1}{2}}n^\pm,\nabla n^\pm,u^\pm)\|_{\dot B^{\frac {d}{2}+1}_{2,1}}^{ 1-\eta}
\end{align}
and
\begin{align}\label{interpo2}
    \|(\kappa^{-\frac{1}{2}}n^\pm,\nabla n^\pm,u^\pm)\|_{\dot B^{\sigma_1+1}_{2,\infty}}
 \leq
\|(\kappa^{-\frac{1}{2}}n^\pm,\nabla n^\pm,u^\pm)\|_{\dot B^{\sigma_1}_{2,\infty}}^{\,1-\eta}
\|(\kappa^{-\frac{1}{2}}n^\pm,\nabla n^\pm,u^\pm)\|_{\dot B^{ \frac{d}{2}+1}_{2,1}}^{\eta}
\end{align}
for the constant $ 
\eta = \frac{1}{-\sigma_1 + \tfrac d2 + 1},
$ satisfying
 \begin{align}
     \frac{d}{2}=\eta\sigma_1+(1-\eta)(\frac{d}{2}+1),\qquad \sigma_1+1=(1-\eta)\sigma_1+\eta(\frac{d}{2}+1).
 \end{align}
Since $-\frac{d}{2}\leq \sigma_1\leq \frac{d}{2}-1$, it follows from \eqref{uv3} ($p=2,s_1=\frac{d}{2}$, $s_2=\sigma_1$) and \eqref{F0} that
 \begin{equation}
     \begin{aligned}
         \| \widetilde{\psi}(\kappa^{-\frac{1}{2}}n^+)\text{div}u^+\|_{\dot{B}^{\sigma_1}_{2,\infty}}\lesssim (1+\|n^+\|_{L^{\infty}_t(L^{\infty})})^{\frac{d}{2}+1}\|\kappa^{-\frac{1}{2}}n^+\|_{\dot{B}^{\sigma_1}_{2,\infty}}\|u^+\|_{\dot B^{ \frac{d}{2}+1}_{2,1}}
     \end{aligned}
 \end{equation}
 and 
 \begin{equation}
     \begin{aligned}
        \kappa^{-\frac{1}{2}} \| u^+\cdot\nabla n^+\|_{\dot{B}^{\sigma_1}_{2,\infty}}\lesssim \|u^+\|_{\dot{B}^{\sigma_1}_{2,\infty}}\|\kappa^{-\frac{1}{2}}n^+\|_{\dot B^{ \frac{d}{2}+1}_{2,1}}. 
     \end{aligned}
 \end{equation}
 This implies
 \begin{align}
     \|\kappa^{-\frac{1}{2}}\widetilde{F}_1\|_{\dot{B}^{\sigma_1}_{2,\infty}}\lesssim \|(\kappa^{-\frac{1}{2}}n^+,u^+)\|_{\dot{B}^{\sigma_1}_{2,\infty}}\|(\kappa^{-\frac{1}{2}}n^+,u^+)\|_{\dot B^{ \frac{d}{2}+1}_{2,1}}.
 \end{align}
 Similarly, we obtain
 \begin{align}
     \|\kappa^{-\frac{1}{2}}\widetilde{F}_3\|_{\dot{B}^{\sigma_1}_{2,\infty}}\lesssim \|(\kappa^{-\frac{1}{2}}n^-,u^-)\|_{\dot{B}^{\sigma_1}_{2,\infty}}\|(\kappa^{-\frac{1}{2}}n^-,u^-)\|_{\dot B^{ \frac{d}{2}+1}_{2,1}}.
 \end{align}
 For nonlinear terms, we start with $\| u^\pm\cdot\nabla n^\pm\|_{\dot{B}^{\sigma_1+1}_{2,\infty}}$. Using interpolation inequalities \eqref{interpo1} and \eqref{interpo2}, we have
\begin{equation}
    \begin{aligned}
        \| u^\pm\cdot\nabla n^\pm\|_{\dot{B}^{\sigma_1+1}_{2,\infty}}&\lesssim \|u^\pm\|_{\dot{B}^{\sigma_1+1}_{2,\infty}} \|\nabla n^\pm\|_{\dot B^{\frac{d}{2}}_{2,1}}\lesssim \|(\nabla n^\pm,u^\pm)\|_{\dot{B}^{\sigma_1}_{2,\infty}}\|(\nabla n^\pm,u^\pm)\|_{\dot B^{ \frac{d}{2}+1}_{2,1}},
    \end{aligned}
\end{equation}
provided $-\frac d2 \leq \sigma_1< \frac d2 -1.$
With the same computation on the Korteweg term $\nabla (|\nabla n^\pm|^2)$, by using \eqref{uv3},
\begin{equation}
    \begin{aligned}
         \|\nabla (|\nabla n^\pm|^2)\|_{\dot{B}^{\sigma_1}_{2,\infty}}&\lesssim\|\nabla n^\pm\|_{\dot{B}^{\sigma_1+1}_{2,\infty}} \|\nabla n^\pm\|_{\dot B^{\frac{d}{2}}_{2,1}}\lesssim\|\nabla n^\pm \|_{\dot{B}^{\sigma_1}_{2,\infty}}\|\nabla n^\pm \|_{\dot B^{ \frac{d}{2}+1}_{2,1}}.
    \end{aligned}
\end{equation}
Similarly, we could estimate terms $u^{+}\cdot\nabla u^+$, $\operatorname{div}\left(2\widetilde{\phi}(\kappa^{-\frac{1}{2}}n^+)\mathbb{D}u^+\right)$, $\widetilde{g}_1^+(\kappa^{-\frac{1}{2}}n^+,\kappa^{-\frac{1}{2}}n^{-}) \nabla \kappa^{-\frac{1}{2}}n^+$,\\
$\widetilde{g}_2^+(\kappa^{-\frac{1}{2}}n^+,\kappa^{-\frac{1}{2}}n^{-})\nabla \kappa^{-\frac{1}{2}}n^{-} $ in $\widetilde{F}_2$ by
  \begin{align*}
         \|u^{\pm}\cdot\nabla u^\pm\|_{\dot{B}^{\sigma_1}_{2,\infty}}
         &\lesssim\|u^\pm \|_{\dot{B}^{\sigma_1}_{2,\infty}}\|u^\pm \|_{\dot B^{ \frac{d}{2}+1}_{2,1}},\\
         \|\operatorname{div}\left(2\widetilde{\phi}(\kappa^{-\frac{1}{2}}n^\pm)\mathbb{D}u^\pm\right)\|_{\dot{B}^{\sigma_1}_{2,\infty}}&\lesssim \|  \kappa^{-\frac{1}{2}}n^\pm\|_{\dot{B}^{\sigma_1+1}_{2,\infty}} \|\mathbb{D}u^\pm\|_{\dot B^{\frac{d}{2}}_{2,1}}\\
         &\lesssim(1+\|n^\pm\|_{\dot B^{ \frac{d}{2}}_{2,1}})^{ \frac{d}{2}+1}\kappa^{-\frac12}\|\nabla n^\pm \|_{\dot{B}^{\sigma_1}_{2,\infty}}\|u^\pm \|_{\dot B^{ \frac{d}{2}+1}_{2,1}},\\  
    \|\widetilde{g}_1^+(\kappa^{-\frac{1}{2}}n^+,\kappa^{-\frac{1}{2}}n^{-})\nabla \kappa^{-\frac{1}{2}} n^+\|_{\dot{B}^{\sigma_1}_{2,\infty}} &\lesssim (1+\|(n^+,n^-)\|_{\dot B^{ \frac{d}{2}}_{2,1}})^{ \frac{d}{2}+1} \|(\kappa^{-\frac{1}{2}}n^+,\kappa^{-\frac{1}{2}}n^-)\|_{\dot{B}^{\sigma_1}_{2,\infty}}\|\kappa^{-\frac{1}{2}}n^+\|_{\dot B^{ \frac{d}{2}+1}_{2,1}}
 \end{align*}
and 
\begin{equation*}
    \begin{aligned}
        \|\widetilde{g}_2^+(\kappa^{-\frac{1}{2}}n^+,\kappa^{-\frac{1}{2}}n^{-})\nabla \kappa^{-\frac{1}{2}}n^{-}\|_{\dot{B}^{\sigma_1}_{2,\infty}} \lesssim& (1+\|(n^+,n^-)\|_{\dot B^{ \frac{d}{2}}_{2,1}})^{ \frac{d}{2}+1}\|(\kappa^{-\frac{1}{2}}n^+,\kappa^{-\frac{1}{2}}n^-)\|_{\dot{B}^{\sigma_1}_{2,\infty}}\|\kappa^{-\frac{1}{2}}n^-\|_{\dot B^{ \frac{d}{2}+1}_{2,1}}.
    \end{aligned}
\end{equation*}
For the term $\nu \widetilde{Q}( \kappa^{-\frac{1}{2}}n^\pm))\Delta u^\pm$, one  deduces from \eqref{uv3} and \eqref{F0} that
\begin{equation}
\begin{aligned}
    &\|\nu \widetilde{Q}( \kappa^{-\frac{1}{2}}n^\pm))\Delta u^\pm\|_{\dot B^{\sigma_1}_{2,\infty}}\\
    \lesssim&(1+\|n^\pm\|_{\dot B^{ \frac{d}{2}}_{2,1}})^{ \frac{d}{2}+1}\kappa^{-\frac{1}{2}} \|n^\pm\|_{\dot B^{ \frac{d}{2}}_{2,1}}\|\Delta u^\pm\|_{\dot B^{\sigma_1}_{2,\infty}}\\
    \lesssim&(1+\|n^\pm\|_{\dot B^{ \frac{d}{2}}_{2,1}})^{ \frac{d}{2}+1}\kappa^{-\frac{1}{2}} \|\nabla n^\pm\|_{\dot B^{\sigma_1}_{2,\infty}}^{\widetilde{\eta}}\|\nabla n^+\|_{\dot B^{ \frac{d}{2}+1}_{2,1}}^{1-\widetilde{\eta}}\|u^\pm\|_{\dot B^{\sigma_1}_{2,\infty}}^{1-\widetilde{\eta}}\|u^\pm\|_{\dot B^{ \frac{d}{2}+1}_{2,1}}^{\widetilde{\eta}}\\
    \lesssim&(1+\|n^\pm\|_{\dot B^{ \frac{d}{2}}_{2,1}})^{ \frac{d}{2}+1}\kappa^{-\frac{1}{2}}\|(\nabla n^\pm,u^\pm)\|_{\dot B^{\sigma_1}_{2,\infty}}\|(\nabla n^\pm,u^\pm)\|_{\dot B^{ \frac{d}{2}+1}_{2,1}}.
\end{aligned}
\end{equation}
Here, we employ the interpolation inequality, and the parameter 
$\widetilde{\eta}$ satisfies the following index relation
\begin{align*}
    \sigma_1+2=(1-\widetilde{\eta})\sigma_1+\widetilde{\eta}(\frac{d}{2}+1), \qquad \frac{d}{2}-1=\widetilde{\eta}\sigma_1+(1-\widetilde{\eta})(\frac{d}{2}+1).
\end{align*}
Similarly, we could estimate the nonlinear terms in $\widetilde{F}_4$, except for $\sqrt{\kappa}\nabla(\widetilde{\psi}(\kappa^{-\frac{1}{2}}n^-)\Delta n^-)$.
For the commutator term, applying the commutator estimate \eqref{Com:2} in Lemma \ref{lemcommutator}, we have 
\begin{equation}
    \begin{aligned}
      \sqrt{\kappa}\sup_j  2^{(\sigma_1+1)j} \|\bigl[\dot\Delta_j,   \widetilde{\psi}(\kappa^{-\frac12}n^\pm)\bigr](\operatorname{div}u^\pm+\Delta n^\pm)\|_{L^2}
\lesssim
 \|\nabla n^\pm\|_{\dot B^{\frac d2}_{2,1}}
 \| u^\pm + \nabla n^\pm\|_{\dot B^{\sigma_1+1}_{2,\infty}}.  
    \end{aligned}
\end{equation}
Then, using interpolation inequalities \eqref{interpo1} and \eqref{interpo2}, we arrive at
\begin{align}
  \sqrt{\kappa}\bigl\|\bigl[\dot\Delta_j, \widetilde{\psi}(\kappa^{-\frac12}n^\pm)\bigr]( \operatorname{div} u^\pm + \Delta n^\pm)\bigr\|_{\dot B^{\sigma_1+1}_{2,\infty}}
\lesssim
\|(\nabla n^\pm,u^\pm)\|_{\dot B^{\sigma_1}_{2,\infty}}
\|(\nabla n^\pm,u^\pm)\|_{\dot B^{ \frac{d}{2}+1}_{2,1}}.  
\end{align}
Finally, one can collect all estimates above and use the Gr\"onwall inequality to deduce \eqref{ineq-evo-1}. 
\end{proof}

\subsection{Time-weighted estimates for arbitrary-order derivatives}
In this subsection, we derive time-decay estimates for derivatives of arbitrary order of the solution. Our approach adapts the ideas of \cite{danchin5,Guo-2012,XinXu2021} and the time-weighted version \cite{CS,LiShou2023} by further exploiting the regularizing effect induced by the Korteweg terms.

For any $s>\frac{d}{2}-1$, let us first introduce the following higher-order time-weighted energy functional:
\begin{equation}\label{definition-X}
    \begin{aligned}
    \mathcal{X}_t^{M,s}:&=\sup_{t\in[0,T]}\left\| t^M \left(\kappa^{-\frac{1}{2}}n^+,\kappa^{-\frac{1}{2}}n^-,\nabla n^+,\nabla n^-, u^+, u^-\right)\right\|_{\dot{B}^{s}_{2,1}}\\
    &\qquad\qquad\qquad\qquad+ \int_0^t\left\| \tau^M \left(\kappa^{-\frac{1}{2}}n^+,\kappa^{-\frac{1}{2}}n^-,\nabla n^+,\nabla n^-,  u^+,u^-\right)\right\|_{\dot{B}^{s+2}_{2,1}}{\rm d}\tau,
\end{aligned}
\end{equation}
where $M$ is any given large constant. We will establish the expected bound of $\mathcal{X}_t^{M,s}$ without using the smallness of initial data, in the sense that $\kappa$ is suitably large and the limiting solution $v^\pm$ has the heat-like decay estimate
\begin{align}\label{definition-Y}
\sup_{\tau\in[0,t] }\|\tau^M (v^+,v^-)\|_{\dot{B}^{s}_{2,1}}+\int_0^t \|\tau^M (v^+,v^-)\|_{\dot{B}^{s+2}_{2,1}}{\rm d}\tau\leq C_1t^{M-\frac{1}{2}(s-\sigma_1)},\quad t>0,
\end{align}
where $C_1>0$ is a uniform-in-time constant. The estimate  \eqref{definition-Y} for $v^\pm$ is justified in Appendix~B.

\begin{proposition}
Let $-\frac d2 \leq \sigma_1< \frac d2 -1$. Suppose that $(n^\pm,u^\pm)$ is the global solution of the Cauchy problem  \eqref{system-perturb1} given by Theorems \ref{cor1}-\ref{cor2}. Then under the assumptions of Theorem \ref{thm:decay-transfer}, for any $M>>1$ and $t>0$, it holds that 
 \begin{equation}\label{es:timeweighted}
   \begin{aligned}
    \mathcal{X}_{t}^{M,s}  \leq &C^*   t^{M-\frac{1}{2}(s-\sigma_1)}\left(\left\|\left(\kappa^{-\frac{1}{2}}n^+,\kappa^{-\frac{1}{2}}n^-,\nabla n^+,\nabla n^-,   u^+, u^-\right)\right\|_{L_t^\infty(\dot{B}^{\sigma_1}_{2,\infty})}+\mathcal{E}_t+  \mathcal{D}_t \right)\\
    &   \quad\quad\quad +   \mathcal{W}_t\mathcal{X}_{t}^{M,s}+\Big(1+\|(n^+,n^-)\|_{L^\infty(\dot{B}^{\frac{d}{2}}_{2,1})}\Big)^{\max\{s+3,\frac{d}{2}+3\}}(\mathcal{X}_{t}^{M,s}\mathcal{D}_t+ \kappa^{-\frac12} \mathcal{E}_t\mathcal{X}_{t}^{M,s}),
\end{aligned}  
 \end{equation}
 where $\mathcal{X}_{t}^{M,s}$, $\mathcal{E}_t$,    $\mathcal{D}_t$ and $\mathcal{W}_t$  are defined by  \eqref{definition-X}, \eqref{definition-E}, \eqref{definition-D} and \eqref{definition-W}, respectively, and $C^*>0$ is a constant uniform in $t$ and $\kappa$.
\end{proposition}
\begin{proof}Let $\mathcal{E}_j(t)$ be  as in \eqref{ineqforbasicenergy0}.  Then multiplying \eqref{ineqforbasicenergy1} by $t^{M}$, we obtain
 
 \begin{equation} \label{ineq-time-2}
    \begin{aligned}
        \frac{{\rm d}}{{\rm d}t}\Big( t^{M}\mathcal{E}_j(t)\Big)+c_0 2^{2j}t^{M}\mathcal{E}_j(t)&\lesssim M t^{M-1} \mathcal{E}_j(t)+t^{M}\mathcal{R}_j\sqrt{\mathcal{E}_j(t)},
    \end{aligned}
\end{equation}
Using $M t^{M-1}\mathcal{E}_j(t)|_{t=0}=0$,  and performing direct computations on the above inequality, we have 
\begin{align}
        \mathcal{X}_t^{M,s}
        &\lesssim\, M \int_0^t \tau^{M-1} \left\|\left(\kappa^{-\frac{1}{2}}n^+,\kappa^{-\frac{1}{2}}n^-,\nabla n^+,\nabla n^-, u^+, u^-\right)\right\|_{\dot{B}^{s}_{2,1}}{\rm d}\tau\nonumber\\
        &+\int_0^t\|(v^+,v^-)\|_{B^{\frac{d}{2}+1}_{2,1}}\mathcal{X}_\tau^{M,s} {\rm d}\tau+\int_0^t \tau^M  \|  (\mathcal{Q}u^+\cdot\nabla n^+)\|_{\dot{B}^{s+1}_{2,1}}{\rm d}\tau \nonumber\\
           &+\int_0^t \tau^M  \|  (\mathcal{Q}u^-\cdot\nabla n^-)\|_{\dot{B}^{s+1}_{2,1}}{\rm d}\tau +\int_0^t \tau^M  \|  ((\mathcal{P}u^+-v^+)\cdot\nabla n^+)\|_{\dot{B}^{s+1}_{2,1}}{\rm d}\tau\nonumber\\
          &+\int_0^t \tau^M  \|  ((\mathcal{P}u^--v^-)\cdot\nabla n^-)\|_{\dot{B}^{s+1}_{2,1}}{\rm d}\tau+\int_0^t  \tau^M\bigg\|  \Big(\kappa^{-\frac{1}{2}}F_1,\kappa^{-\frac{1}{2}}F_3, \sum_{l=1}^6\mathcal{Q}F_2^l ,\sum_{l=1}^6\mathcal{Q}F_4^l\Big)\bigg\|_{\dot{B}^{s}_{2,1}}{\rm d}\tau\nonumber\\ 
         &+\, \sqrt{\kappa}\sum_{j}2^{j(s+1)}\int_0^t   \|[ \dot{\Delta}_j,  \widetilde{\psi}(\kappa^{-\frac12}n^+)](\operatorname{div}u^++\Delta n^+)\|_{L^2} {\rm d}\tau\nonumber\\ 
         &+\,\sqrt{\kappa}\sum_{j}2^{j(s+1)}\int_0^t   \|[ \dot{\Delta}_j,  \widetilde{\psi}(\kappa^{-\frac12}n^-)](\operatorname{div}u^-+\Delta n^-)\|_{L^2} {\rm d}\tau.\label{ineq-time-6}
    \end{align}

We are in a position to estimate the nonlinear terms in \eqref{ineq-time-6}.  
First, we estimate the first term on the right–hand side of \eqref{ineq-time-6}. Let $\eta_0\in(0,1)$ be such that $\sigma_1(1-\eta_0)+(s+2)\eta_0=s$. It follows from the interpolation inequality and Young's inequality that 
\begin{equation}\label{ineq-time-7}
    \begin{aligned}
         &\int_0^t \tau^{M-1} \left\|\left(\kappa^{-\frac{1}{2}}n^\pm,\nabla n^\pm,   u^\pm\right)\right\|_{\dot{B}^{s}_{2,1}}{\rm d}\tau\\
         \leq&  \int_0^t \tau^{M-1}\left\|\left(\kappa^{-\frac{1}{2}}n^\pm,\nabla n^\pm,   u^\pm\right)\right\|^{1-\eta_0}_{\dot{B}^{\sigma_1}_{2,\infty}}\left\|\left(\kappa^{-\frac{1}{2}}n^\pm,\nabla n^\pm,   u^\pm\right)\right\|^{\eta_0}_{\dot{B}^{s+2}_{2,1}}{\rm d}\tau\\
         \leq& \left( \int_0^t  \tau^{M-\frac{1}{1-\eta_0}}\left\|\left(\kappa^{-\frac{1}{2}}n^\pm,\nabla n^\pm,   u^\pm\right)\right\|_{\dot{B}^{\sigma_1}_{2,\infty}} {\rm d}\tau\right)^{1-\eta_0}\left\|\tau^M\left(\kappa^{-\frac{1}{2}}n^\pm,\nabla n^\pm,   u^\pm\right)\right\|^{\eta_0}_{L_t^1(\dot{B}^{s+2}_{2,1})}\\
         \leq& \epsilon \left\|\tau^M\left(\kappa^{-\frac{1}{2}}n^\pm,\nabla n^\pm,   u^\pm\right)\right\| _{L_t^1(\dot{B}^{s+2}_{2,1})}+\epsilon^{-1}t^{M-\frac{1}{2}(s-\sigma_1)}\left\|\left(\kappa^{-\frac{1}{2}}n^\pm,\nabla n^\pm,   u^\pm\right)\right\|_{L_t^\infty(\dot{B}^{\sigma_1}_{2,\infty})}.
    \end{aligned} 
\end{equation}
where $\epsilon>0$ is a small constant to be chosen later.

The nonlinear terms on the right-hand side of \eqref{ineq-time-6}, namely $\mathcal{Q}u^\pm\cdot\nabla n^\pm$ and  $(\mathcal{P}u^\pm-v^\pm)\cdot\nabla n^\pm$  can be estimated  by \eqref{uv1} as follows:
\begin{equation}\label{ineq-time-8}
    \begin{aligned}
        & \int_0^{t}  \tau^M\|   \mathcal{Q}u^\pm\cdot\nabla n^\pm\|_{\dot{B}^{s+1}_{2,1}}{\rm d}\tau\\
        \lesssim&  \|\tau^M\mathcal{Q}u^\pm \|_{L^2_t(\dot{B}^{s+1}_{2,1})}\| \nabla n^\pm\|_{L^2_t(L^\infty)}+\|\mathcal{Q}u^\pm\|_{L^2_t(L^\infty)}\|\tau^M\nabla n^\pm\|_{L^2_t(\dot{B}^{s+1}_{2,1})}\\
        \lesssim& \mathcal{X}_{t}^{M,s}\mathcal{D}_t
    \end{aligned}
\end{equation}
and 
\begin{equation}\label{ineq-time-9}
    \begin{aligned}
        &\int_0^t \tau^M  \|  (\mathcal{P}u^\pm-v^\pm)\cdot\nabla n^\pm\|_{\dot{B}^{s+1}_{2,1}}{\rm d}\tau\\
        \lesssim&\| \tau^M (\mathcal{P}u^\pm-v^\pm)\|_{L^2_t(\dot{B}^{s+1}_{2,1})}\| \nabla n^\pm\|_{L_t^2(L^\infty)}+\|\mathcal{P}u^\pm-v^\pm\|_{L_t^2(L^\infty)}\|\tau^M\nabla n^\pm\|_{L_t^2(\dot{B}^{s+1}_{2,1})}\\
        \lesssim&  \Big(\mathcal{X}_{t}^{M,s}+t^{M-\frac{1}{2}(s-\sigma_1)}\Big) \mathcal{D}_t+\mathcal{W}_t\mathcal{X}_{t}^{M,s},
    \end{aligned}
\end{equation}
where \eqref{definition-Y} has been used in the final step of \eqref{ineq-time-9}.

Next, recall that $F_1=-\sqrt{\kappa}\widetilde{\psi}(\kappa^{-\frac{1}{2}}n^+)\operatorname{div}u^+-\mathcal{Q}u^+\cdot\nabla n^+-(\mathcal{P}u^+-v^+)\cdot\nabla n^+$.  In view of \eqref{uv1} and \eqref{definition-Y}, each term is then bounded by
\begin{equation*}
    \begin{aligned}
          \int_0^{t}  \|   \widetilde{\psi}(\kappa^{-\frac{1}{2}}n^+)\text{div}u^+ \|_{\dot{B}^{s}_{2,1}}{\rm d}\tau
        \lesssim& \int_0^{t}\|\tau^M \kappa^{-\frac{1}{2}}n^+\|_{\dot{B}^{s}_{2,1}} \|\text{div}u^+ \|_{L^\infty}+\|\kappa^{-\frac{1}{2}}n^+\|_{L^\infty}\|\tau^M\text{div}u^+\|_{\dot{B}^{s}_{2,1}}{\rm d}\tau\\
        \lesssim&\int_0^{t}\| u^+ \|_{\dot{B}^{\frac{d}{2}+1}_{2,1}}\|\tau^M \kappa^{-\frac{1}{2}}n^+\|_{\dot{B}^{s}_{2,1}} {\rm d}\tau+  \mathcal{D}_t\mathcal{X}_{t}^{M,s},
    \end{aligned}
\end{equation*}
\begin{equation*}
    \begin{aligned}
         \int_0^{t}  \|\tau^M(\kappa^{-\frac{1}{2}}\mathcal{Q}u^+\cdot\nabla n^+)\|_{\dot{B}^{s}_{2,1}}{\rm d}\tau
        \lesssim&  \int_0^{t}  \|\mathcal{Q}u^+ \|_{L^\infty}\|\tau^M\kappa^{-\frac{1}{2}}\nabla n^+\|_{\dot{B}^{s}_{2,1}}+\|\kappa^{-\frac{1}{2}}\nabla n^+\|_{L^\infty}\|\tau^M\mathcal{Q}u^+\|_{\dot{B}^{s}_{2,1}}{\rm d}\tau\\
        \lesssim&\int_0^t\|\kappa^{-\frac{1}{2}}\nabla n^+\|_{L^\infty}\|\tau^M\mathcal{Q}u^+ \|_{\dot{B}^{s}_{2,1}}{\rm d}\tau +   \mathcal{D}_t\mathcal{X}_{t}^{M,s}
    \end{aligned}
\end{equation*}
and
\begin{equation*}
    \begin{aligned}
        & \int_0^{t}  \|\tau^M(\kappa^{-\frac{1}{2}}(\mathcal{P}u^+-v^+)\cdot\nabla n^+)\|_{\dot{B}^{s}_{2,1}}{\rm d}\tau\\
        \lesssim&  \int_0^{t}  \|\mathcal{P}u^+-v^+ \|_{L^\infty}\|\tau^M\kappa^{-\frac{1}{2}}\nabla n^+\|_{\dot{B}^{s}_{2,1}}+\int_0^t\|\kappa^{-\frac{1}{2}}\nabla n^+\|_{L^\infty}\|\tau^M(\mathcal{P}u^+-v^+) \|_{\dot{B}^{s}_{2,1}}{\rm d}\tau\\
        \lesssim&  \mathcal{W}_t\mathcal{X}_{t}^{M,s}+ \int_0^t\|\kappa^{-\frac{1}{2}}\nabla n^+\|_{L^\infty}\|\tau^M u^+ \|_{\dot{B}^{s}_{2,1}}{\rm d}\tau+  \mathcal{E}_t t^{M-\frac{1}{2}(s-\sigma_1)}.
    \end{aligned}
\end{equation*}
Summing these estimates yields
\begin{equation}\label{ineq-time-10}
    \begin{aligned}
         \int_0^{t}  \| \kappa^{-\frac{1}{2}}F_1\|_{\dot{B}^{s}_{2,1}}{\rm d}\tau 
        & \lesssim \int_0^t\|(\kappa^{-\frac{1}{2}}n^+,u^+)\|_{\dot{B}^{\frac{d}{2}+1}_{2,1}}\|\tau^M(\kappa^{-\frac{1}{2}}n^+, u^+)\|_{\dot{B}^{s}_{2,1}}{\rm d}\tau\\
        &\quad + \mathcal{W}_t\mathcal{X}_{t}^{M,s}+   \mathcal{D}_t\mathcal{X}_{t}^{M,s}+\mathcal{E}_t t^{M-\frac{1}{2}(s-\sigma_1)}.
    \end{aligned}
\end{equation}
In the same way, one obtains the analogous bound for $F_3$:
\begin{equation}\label{ineq-time-11}
    \begin{aligned}
         \int_0^{t}  \| \kappa^{-\frac{1}{2}}F_3\|_{\dot{B}^{s}_{2,1}}{\rm d}\tau
         &\lesssim \int_0^t\|(\kappa^{-\frac{1}{2}}n^-,u^-)\|_{\dot{B}^{\frac{d}{2}+1}_{2,1}}\|\tau^M(\kappa^{-\frac{1}{2}}n^-, u^- )\|_{\dot{B}^{s}_{2,1}}{\rm d}\tau\\
         &\quad+ \mathcal{W}_t\mathcal{X}_{t}^{M,s}+   \mathcal{D}_t\mathcal{X}_{t}^{M,s}+\mathcal{E}_t t^{M-\frac{1}{2}(s-\sigma_1)}.
    \end{aligned}
\end{equation}
We now turn to the remaining nonlinear contributions one by one, showing that each can be absorbed into our overall energy–dissipation framework:

For the convection term $u^{\pm}\cdot\nabla u^\pm,$ using the standard product estimate \eqref{uv1}, one obtains 
\begin{equation}
    \begin{aligned}
        & \int_0^{t}  \|\tau^M (u^{\pm}-v^\pm)\cdot\nabla u^\pm\|_{\dot{B}^{s}_{2,1}}{\rm d}\tau
        \lesssim\int_0^t\| u^\pm \|_{\dot{B}^{\frac{d}{2}+1}_{2,1}}\|\tau^M u^\pm \|_{\dot{B}^{s}_{2,1}}{\rm d}\tau + \mathcal{D}_t\mathcal{X}_{t}^{M,s}+\mathcal{W}_t\mathcal{X}_{t}^{M,s}.
    \end{aligned}
\end{equation}
For the Korteweg-type capillarity term, we have
\begin{equation}
    \begin{aligned}
        & \int_0^{t}  \|\tau^M \nabla (|\nabla n^\pm|^2)\|_{\dot{B}^{s}_{2,1}}{\rm d}\tau
        \lesssim   \mathcal{D}_t\mathcal{X}_{t}^{M,s}.
    \end{aligned}
\end{equation}
Using composite-function estimate \eqref{F0} and  the standard product estimate \eqref{uv1}, we have 
\begin{equation}
    \begin{aligned}
        & \int_0^{t}  \|\tau^M \operatorname{div}\left(2\widetilde{\phi}(\kappa^{-\frac{1}{2}}n^+)\mathbb{D}u^+\right)\|_{\dot{B}^{s}_{2,1}}{\rm d}\tau\\
        \lesssim&(1+\|(n^+,n^-)\|_{L_t^\infty(\dot{B}^{\frac{d}{2}}_{2,1})})^{s+2}\|\tau^M  \kappa^{-\frac{1}{2}}n^+\|_{L_t^\infty(\dot{B}^{s+1}_{2,1})} \|\mathbb{D}u^+\|_{L_t^1(L^\infty)}\\
        &+(1+\|(n^+,n^-)\|_{L_t^\infty(\dot{B}^{\frac{d}{2}}_{2,1})})^{\frac{d}{2}+1} \| \kappa^{-\frac{1}{2}}n^+\|_{L^\infty_t(\dot{B}^{\frac{d}{2}}_{2,1})}\|\mathbb{D}u^+\|_{L^1_t(\dot{B}^{s+1}_{2,1})}\\
        \lesssim&\kappa^{-\frac12}(1+\|(n^+,n^-)\|_{L_t^\infty(\dot{B}^{\frac{d}{2}}_{2,1})})^{\max\{s+2,\frac{d}{2}+1\}}\mathcal{E}_t\mathcal{X}_{t}^{M,s}.
    \end{aligned}
\end{equation}
For  $\widetilde{g}_1^+(\kappa^{-\frac{1}{2}}n^+,\kappa^{-\frac{1}{2}}n^{-})\kappa^{-\frac{1}{2}} \nabla n^+$ and $\widetilde{g}_2^+(\kappa^{-\frac{1}{2}}n^+,\kappa^{-\frac{1}{2}}n^{-})\nabla \kappa^{-\frac{1}{2}}n^{-}$. We only estimate $\widetilde{g}_1^+(\kappa^{-\frac{1}{2}}n^+,\kappa^{-\frac{1}{2}}n^{-})\kappa^{-\frac{1}{2}} \nabla n^+$, since the other can be addressed similarly.
\begin{equation}
    \begin{aligned}
        & \int_0^{t}  \|\tau^M \widetilde{g}_1^+(\kappa^{-\frac{1}{2}}n^+,\kappa^{-\frac{1}{2}}n^{-})\kappa^{-\frac{1}{2}} \nabla n^+\|_{\dot{B}^{s}_{2,1}}{\rm d}\tau\\
        \lesssim &\int_0^t(1+\|(n^+,n^-)\|_{\dot{B}^{\frac{d}{2}}_{2,1}})^{s+1}\|\tau^M \kappa^{-\frac{1}{2}}(n^+,n^-)\|_{ \dot{B}^{s}_{2,1}} \|\kappa^{-\frac{1}{2}} \nabla n^+\|_{L_t^1(L^\infty)} {\rm d}\tau\\
        &+ (1+\|(n^+,n^-)\|_{L_t^\infty(\dot{B}^{\frac{d}{2}}_{2,1})})^{\frac{d}{2}+1}\|\kappa^{-\frac{1}{2}}(n^+,n^-)\|_{L^2_t(\dot{B}^{\frac{d}{p}}_{p,1})} )\|\kappa^{-\frac{1}{2}} \nabla n^+\|_{L^2_t(\dot{B}^{s}_{2,1})}\\
        \lesssim&\int_0^t(1+\|(n^+,n^-)\|_{\dot{B}^{\frac{d}{2}}_{2,1}})^{s+1}\|\tau^M \kappa^{-\frac{1}{2}}(n^+,n^-)\|_{ \dot{B}^{s}_{2,1}} \| \kappa^{-\frac{1}{2}}n^+\|_{\dot{B}^{\frac{d}{2}+1}_{2,1}}{\rm d}\tau\\
        &+  (1+\|(n^+,n^-)\|_{L_t^\infty(\dot{B}^{\frac{d}{2}}_{2,1})})^{\frac{d}{2}+1}\mathcal{D}_t\mathcal{X}_{t}^{M,s}.
    \end{aligned}
\end{equation}
The most challenging term is $\widetilde{Q}( \kappa^{-\frac{1}{2}}n^\pm))\Delta u^\pm$.  One rewrite  it  as
$$\widetilde{Q}( \kappa^{-\frac{1}{2}}n^\pm))\Delta u^\pm=\Delta (\widetilde{Q}( \kappa^{-\frac{1}{2}}n^\pm)u^\pm)-\Delta \widetilde{Q}( \kappa^{-\frac{1}{2}}n^\pm) u^\pm-2\nabla \widetilde{Q}( \kappa^{-\frac{1}{2}}n^\pm)\nabla u^\pm.$$
Then we arrive at 
\begin{equation}
    \begin{aligned}
        &\int_0^t\|\tau^M\widetilde{Q}( \kappa^{-\frac{1}{2}}n^\pm))\Delta u^\pm\|_{\dot{B}^{s}_{2,1}}{\rm d}\tau\\
        \lesssim&\int_0^t\|\tau^M \Delta (\widetilde{Q}( \kappa^{-\frac{1}{2}}n^\pm)u^\pm)\|_{\dot{B}^{s}_{2,1}}+\|\tau^M \Delta \widetilde{Q}( \kappa^{-\frac{1}{2}}n^\pm) u^\pm\|_{\dot{B}^{s}_{2,1}}+\|\tau^M\nabla \widetilde{Q}( \kappa^{-\frac{1}{2}}n^\pm)\nabla u^\pm\|_{\dot{B}^{s}_{2,1}} {\rm d}\tau\\
        \quad:=& I_1+I_2+I_3.
    \end{aligned}
\end{equation}
For $I_1$, we estimate  by 
\begin{equation*}
    \begin{aligned}
        I_1&\lesssim   \|\tau^{M}\widetilde{Q}( \kappa^{-\frac{1}{2}}n^\pm)\|_{L^2_t(\dot{B}^{s+2}_{2,1})}\|u^\pm\|_{L^2(L^\infty)}+\|\tau^M u^\pm\|_{L_t^1(\dot{B}^{s+2}_{2,1})}\|\widetilde{Q}( \kappa^{-\frac{1}{2}}n^\pm)\|_{L^{\infty}_t(L^\infty)}\\
        &\lesssim\kappa^{-\frac12}(1+\|(n^+,n^-)\|_{L_t^\infty(\dot{B}^{\frac{d}{2}}_{2,1})})^{s+3}\|\tau^{M} n^\pm\|_{L^2_t(\dot{B}^{s+2}_{2,1})}\\
        &\quad+(1+\|(n^+,n^-)\|_{L_t^\infty(\dot{B}^{\frac{d}{2}}_{2,1})})^{\frac{d}{2}+1}\|\tau^M u^\pm\|_{L^1_t(\dot{B}^{s+2}_{2,1})}\|  \kappa^{-\frac{1}{2}}n^\pm\|_{L^{\infty}_t(\dot{B}^{\frac{d}{p}}_{p,1} )}\\
        &\lesssim(1+\|(n^+,n^-)\|_{L_t^\infty(\dot{B}^{\frac{d}{2}}_{2,1})})^{\max\{s+3,\frac{d}{2}+1\}}(\kappa^{-\frac{1}{2}}\mathcal{E}_t\mathcal{X}_{t}^{M,s}+ \mathcal{D}_t\mathcal{X}_{t}^{M,s}).
    \end{aligned}
\end{equation*}
For $I_2$, we estimate  by 
\begin{equation*}
    \begin{aligned}
      I_2&\lesssim\|\tau^M\Delta \widetilde{Q}( \kappa^{-\frac{1}{2}}n^\pm)\|_{L^2_t(\dot{B}^{s}_{2,1})}\|u^\pm\|_{L_t^2(L^\infty)}+\|\Delta \widetilde{Q}( \kappa^{-\frac{1}{2}}n^\pm)\|_{L^{1}_t(L^{\infty})}\|\tau^M u^\pm\|_{L^\infty_t(\dot{B}^{s}_{2,1})}\\
      &\lesssim(1+\|(n^+,n^-)\|_{L^\infty_t(\dot{B}^{\frac{d}{2}}_{2,1})})^{s+3}\|\tau^M   \kappa^{-\frac{1}{2}}n^\pm\|_{L^2_t(\dot{B}^{s+2}_{2,1})}\|u^\pm\|_{L^2_t(L^\infty)}\\
      &\quad+(1+\|(n^+,n^-)\|_{L_t^\infty(\dot{B}^{\frac{d}{2}}_{2,1})})^{\frac{d}{2}+3}\|  \kappa^{-\frac{1}{2}}n^\pm\|_{L^{1}_t(\dot{B}^{\frac{d}{2}+2}_{2,1})}\|\tau^M u^\pm\|_{L^\infty_t(\dot{B}^{s}_{2,1})}\\
       &\lesssim \kappa^{-\frac12}(1+\|(n^+,n^-)\|_{L_t^\infty(\dot{B}^{\frac{d}{2}}_{2,1})})^{\max\{s+3,\frac{d}{2}+3\}}\mathcal{E}_t\mathcal{X}_{t}^{M,s}
    \end{aligned}
\end{equation*}
and for $I_3$, we estimate  by 
\begin{equation*}
    \begin{aligned}
        I_3&\lesssim\|\tau^M \nabla\widetilde{Q}( \kappa^{-\frac{1}{2}}n^\pm) \|_{L^{\infty}_t(\dot{B}^{s}_{2,1})}\|\nabla u^\pm\|_{L^1_t(L^\infty)}+\|\nabla\widetilde{Q}( \kappa^{-\frac{1}{2}}n^\pm)\|_{L^2_t(L^\infty)}\|\nabla u^\pm\|_{L^2_t(\dot{B}^{s}_{2,1})}\\
        &\lesssim (1+\|(n^+,n^-)\|_{L_t^\infty(\dot{B}^{\frac{d}{2}}_{2,1})})^{s+2} \|\tau^M   \kappa^{-\frac{1}{2}}n^\pm \|_{L^{\infty}_t(\dot{B}^{s+1}_{2,1})}\|\nabla u^\pm\|_{L^1_t(L^\infty)}\\
        &\quad+(1+\|(n^+,n^-)\|_{L_t^\infty(\dot{B}^{\frac{d}{2}}_{2,1})})^{\frac{d}{2}+2} \| \kappa^{-\frac{1}{2}}n^\pm\|_{L^2_t( \dot{B}^{\frac{d}{2}+1}_{2,1})}\|\nabla u^\pm\|_{L^2_t(\dot{B}^{s}_{2,1})}  \\
        &\lesssim\kappa^{-\frac12} (1+\|(n^+,n^-)\|_{L_t^\infty(\dot{B}^{\frac{d}{2}}_{2,1})})^{\max\{s+2,\frac{d}{2}+2\}}\mathcal{E}_t\mathcal{X}_{t}^{M,s}.
    \end{aligned}
\end{equation*}
Combining these three bounds delivers the desired control
\begin{equation}\label{ineq-time-12}
    \begin{aligned}
        &\int_0^t\|\tau^M\widetilde{Q}( \kappa^{-\frac{1}{2}}n^\pm))\Delta u^\pm\|_{\dot{B}^{s}_{2,1}}{\rm d}\tau\\
        \lesssim &\kappa^{-\frac12}(1+\|(n^+,n^-)\|_{L_t^\infty(\dot{B}^{\frac{d}{2}}_{2,1})})^{\max\{s+3,\frac{d}{2}+3\}}(\kappa^{-\frac12} \mathcal{E}_t\mathcal{X}_{t}^{M,s}+ \mathcal{D}_t\mathcal{X}_{t}^{M,s}).
    \end{aligned}
\end{equation}
Finally, we deal with the commutator contributions. We consider two cases. First, in the low–regularity regime $\sigma_1<s\leq \frac{d}{p}$, using \eqref{Com:1} and \eqref{F0}, we obtain
\begin{equation}\label{ineq-time-13}
    \begin{aligned}
&\sqrt{\kappa} \sum_{j}2^{j(1+s)}  \int_0^{t} \tau^M\|[ \dot{\Delta}_j,  \widetilde{\psi}(\kappa^{-\frac12}n^\pm)]\operatorname{div}u^\pm\|_{L^2}{\rm d}\tau\\
\lesssim\,& \|\nabla \widetilde{\psi}(\kappa^{-\frac12}n^\pm)\|_{L^2_t(\dot{B}^{\frac{d}{p}}_{p,1})}\|u^\pm\|_{L^2_t(\dot{B}^{s+1}_{2,1})} 
\lesssim\,  \Big(1+\|n^\pm\|_{L_t^\infty(\dot{B}^{\frac{d}{2}}_{2,1})}\Big)^{\frac{d}{2}+1} \mathcal{D}_t\mathcal{X}_{t}^{M,s}
 \end{aligned}
\end{equation}
and, by the very same argument,
\begin{equation}\label{ineq-time-14}
    \begin{aligned}
&\sqrt{\kappa} \sum_{j}2^{j(1+s)}  \int_0^{t} \tau^M\|[ \dot{\Delta}_j,  \widetilde{\psi}(\kappa^{-\frac12}n^\pm)]\Delta n^\pm\|_{L^2}{\rm d}\tau\\
\lesssim\,& \|\nabla\widetilde{\psi}(\kappa^{-\frac12}n^\pm)\|_{L^2_t(\dot{B}^{\frac{d}{p}}_{p,1})}\|\nabla n^\pm\|_{L^2_t(\dot{B}^{s+1}_{2,1})} 
\lesssim\, \Big(1+\|n^\pm\|_{L_t^\infty(\dot{B}^{\frac{d}{2}}_{2,1})}\Big)^{\frac{d}{2}+1} \mathcal{D}_t\mathcal{X}_{t}^{M,s}.
 \end{aligned}
\end{equation}

Next, in the high–regularity regime $s>\frac{d}{p}$, we apply \eqref{Com:3} together with \eqref{F0}  to get
    \begin{align}\label{ineq-time-15}
&\sqrt{\kappa} \sum_{j}2^{j(1+s)}  \int_0^{t} \tau^M\|[ \dot{\Delta}_j,  \widetilde{\psi}(\kappa^{-\frac12}n^\pm)]\operatorname{div}u^\pm\|_{L^2}{\rm d}\tau\nonumber\\
\lesssim\,& \sqrt{\kappa}\|\nabla \widetilde{\psi}(\kappa^{-\frac12}n^\pm)\|_{L^2_t(\dot{B}^{\frac{d}{p}}_{p,1})}\|u^\pm\|_{L^2_t(\dot{B}^{s+1}_{2,1})}+\sqrt{\kappa}\int_0^{t}\|\nabla u^\pm\|_{L^\infty}\|\widetilde{\psi}(\kappa^{-\frac12}n^\pm)\|_{\dot{B}^{s+1}_{2,1}}{\rm d}\tau\nonumber\\
\lesssim\,& \Big(1+\|n^\pm\|_{L_t^\infty(\dot{B}^{\frac{d}{2}}_{2,1})}\Big)^{\frac{d}{2}+1}\|\nabla n^\pm\|_{L^2_t(\dot{B}^{\frac{d}{p}}_{p,1})}\|u^\pm\|_{L^2_t(\dot{B}^{s+1}_{2,1})} \\
&+  \int_0^{t}\Big(1+\|n^\pm\|_{L^\infty(\dot{B}^{\frac{d}{2}}_{2,1})}\Big)^{s+2}\|\nabla u^\pm\|_{\dot{B}^{\frac{d}{2}}_{2,1}}\| n^\pm\|_{\dot{B}^{s+1}_{2,1}}{\rm d}\tau\nonumber\\
\lesssim\,& \Big(1+\|n^\pm\|_{L_t^\infty(\dot{B}^{\frac{d}{2}}_{2,1})}\Big)^{\frac{d}{2}+1} \mathcal{D}_t\mathcal{X}_{t}^{M,s}+\int_0^{t}(1+\|(n^+,n^-)\|_{L^\infty(\dot{B}^{\frac{d}{2}}_{2,1})})^{s+2}\|\nabla u^\pm\|_{\dot{B}^{\frac{d}{2}}_{2,1}}\mathcal{X}_\tau^{M,s}{\rm d}\tau\nonumber
 \end{align}
and 
\begin{equation}\label{ineq-time-16}
    \begin{aligned}
&\sqrt{\kappa} \sum_{j}2^{j(1+s)}  \int_0^{t} \tau^M\|[ \dot{\Delta}_j,  \widetilde{\psi}(\kappa^{-\frac12}n^\pm)]\Delta n^\pm\|_{L^2}{\rm d}\tau\\
\lesssim\,& \sqrt{\kappa}\|\nabla \widetilde{\psi}(\kappa^{-\frac12}n^\pm)\|_{L^2_t(\dot{B}^{\frac{d}{p}}_{p,1})}\|\nabla n^\pm\|_{L^2_t(\dot{B}^{s+1}_{2,1})}+\sqrt{\kappa}\int_0^{t}\|\Delta n^\pm\|_{L^\infty}\|\widetilde{\psi}(\kappa^{-\frac12}n^\pm)\|_{\dot{B}^{s+1}_{2,1}}{\rm d}\tau\\
\lesssim\,& \Big(1+\|n^\pm\|_{L_t^\infty(\dot{B}^{\frac{d}{2}}_{2,1})}\Big)^{\frac{d}{2}+1}\|\nabla n^\pm\|_{L^2_t(\dot{B}^{\frac{d}{p}}_{p,1})}\|\nabla n^\pm\|_{L^2_t(\dot{B}^{s+1}_{2,1})} \\
&+  \int_0^{t}\Big(1+\|n^\pm\|_{L^\infty(\dot{B}^{\frac{d}{2}}_{2,1})}\Big)^{s+2}\|\Delta n^\pm\|_{\dot{B}^{\frac{d}{2}}_{2,1}}\| n^\pm\|_{\dot{B}^{s+1}_{2,1}}{\rm d}\tau\\
\lesssim\,& \Big(1+\|n^\pm\|_{L_t^\infty(\dot{B}^{\frac{d}{2}}_{2,1})}\Big)^{\frac{d}{2}+1} \mathcal{D}_t\mathcal{X}_{t}^{M,s}+\int_0^{t}\Big(1+\|n^\pm\|_{L^\infty(\dot{B}^{\frac{d}{2}}_{2,1})}\Big)^{s+2}\|\nabla n^\pm\|_{\dot{B}^{\frac{d}{2}+1}_{2,1}}\mathcal{X}_\tau^{M,s}{\rm d}\tau.
 \end{aligned}
\end{equation}

Collecting \eqref{ineq-time-7}-\eqref{ineq-time-16},  inserting them into\eqref{ineq-time-6}, letting $\epsilon$ be sufficiently small and applying Gr\"onwall’s lemma then yields \eqref{es:timeweighted},
which completes the proof.
\end{proof}

\vspace{2mm}

\noindent
\textbf{Proof of Theorem \ref{thm:decay-transfer}}.  
Since $\mathcal E_t, \kappa^{\delta}\mathcal D_t$ and $\kappa^{\delta}\mathcal W_t$ remain uniformly bounded (see \eqref{253}), by choosing $\kappa$ sufficiently large,
$$
 \mathcal{W}_t+(1+\|(n^+,n^-)\|_{L_t^\infty(\dot{B}^{\frac{d}{2}}_{2,1})})^{\max\{s+3,\frac{d}{2}+3\}}(\mathcal{D}_t+ \kappa^{-\frac12} \mathcal{E}_t )\leq C\kappa^{-\delta}\leq\frac{1}{2C^*},
$$
the last three terms in the time‐weighted estimate \eqref{es:timeweighted} can be absorbed into the left–hand side.  Consequently, we obtain 
\begin{equation}
    \begin{aligned}
        \mathcal{X}_{t}^{M,s} \lesssim  t^{M-\frac{1}{2}(s-\sigma_1)}\left(\left\|\left(\kappa^{-\frac{1}{2}}n^+,\kappa^{-\frac{1}{2}}n^-,\nabla n^+,\nabla n^-,   u^+, u^-\right)\right\|_{\dot{B}^{\sigma_1}_{2,\infty}}+(\mathcal{E}_t+ \mathcal{D}_t)\|(v^+,v^-)\|_{\dot{B}^{\sigma_1}_{2,\infty}}\right).
    \end{aligned}
\end{equation}
By the lower-order regularity evolution \eqref{3.1} and the definition of $ \mathcal{X}_{t}^{M,s} $, this in turn yields, for any $s>\sigma_1$ that
\begin{equation}
    \begin{aligned}
        &\left\| \left(\kappa^{-\frac{1}{2}}n^+,\kappa^{-\frac{1}{2}}n^-,\nabla n^+,\nabla n^-,   u^+, u^-\right)\right\|_{\dot{B}^{s}_{2,1}}\\
        \lesssim&  t^{-\frac{s-\sigma_1}{2}}\left(\left\|\left(\kappa^{-\frac{1}{2}}n^+,\kappa^{-\frac{1}{2}}n^-,\nabla n^+,\nabla n^-,   u^+, u^-\right)\right\|_{\dot{B}^{\sigma_1}_{2,\infty}}+(\mathcal{E}_t+ \mathcal{D}_t) \|(v^+,v^-)\|_{\dot{B}^{\sigma_1}_{2,\infty}}\right),
    \end{aligned}
\end{equation}
which, combined with the uniform evolution \eqref{3.1}, gives rise to \eqref{eq:cns-decay}. This completes the proof of Theorem \ref{thm:decay-transfer}.
\vspace{3mm}



\appendix
\titleformat{\section}
  {\normalfont\Large\bfseries}
  {Appendix~\thesection} 
  {1em}{}

\section[Appendix~\thesection: Technical lemmas in Besov spaces]
        {Technical lemmas in Besov spaces}\label{appendixA}

We begin by recalling basic properties of Besov spaces and product estimates that will be used repeatedly throughout the paper.  The same statements remain valid in the Chemin–Lerner spaces, with the caveat that the time exponents must satisfy the H${\rm{\ddot{o}}}$lder  inequality in the time variable. 

Our first lemma is the classical Bernstein inequality. In particular, it implies that, for any $u$ belonging to a Besov space, each dyadic block $\dot{\Delta}_{j}u$ is smooth. Consequently, after applying the localization operator $\dot{\Delta}_{j}$, we may carry out direct  computations on the linear equations.
\begin{lemma}
	Let $0<r<R$, $1\leq p\leq q\leq \infty$ and $k\in \mathbb{N}$. For any $u\in L^p$ and $\lambda>0$, it holds
	\begin{equation}\notag
		\left\{
		\begin{aligned}
			&{\rm{Supp}}~ \mathcal{F}(u) \subset \{\xi\in\mathbb{R}^{d}~| ~|\xi|\leq \lambda R\}\Rightarrow \|D^{k}u\|_{L^q}\lesssim\lambda^{k+d(\frac{1}{p}-\frac{1}{q})}\|u\|_{L^p},\\
			&{\rm{Supp}}~ \mathcal{F}(u) \subset \{\xi\in\mathbb{R}^{d}~|~ \lambda r\leq |\xi|\leq \lambda R\}\Rightarrow \|D^{k}u\|_{L^{p}}\sim\lambda^{k}\|u\|_{L^{p}}.
		\end{aligned}
		\right.
	\end{equation}
\end{lemma}

As a consequence of the Bernstein inequalities, Besov spaces enjoy the following properties.
\begin{lemma}
	The following properties hold{\rm:}
	\begin{itemize}
		\item{} For $s\in\mathbb{R}$, $1\leq p_{1}\leq p_{2}\leq \infty$ and $1\leq r_{1}\leq r_{2}\leq \infty$, it holds
		\begin{equation}\notag
			\begin{aligned}
				\dot{B}^{s}_{p_{1},r_{1}}\hookrightarrow \dot{B}^{s-d(\frac{1}{p_{1}}-\frac{1}{p_{2}})}_{p_{2},r_{2}}.
			\end{aligned}
		\end{equation}
		\item{} For $1\leq p\leq q\leq\infty$, we have the following chain of continuous embedding:
		\begin{equation}\nonumber
			\begin{aligned}
				\dot{B}^{0}_{p,1}\hookrightarrow L^{p}\hookrightarrow \dot{B}^{0}_{p,\infty}\hookrightarrow \dot{B}^{\sigma}_{q,\infty},\quad \sigma=-d(\frac{1}{p}-\frac{1}{q})<0.
			\end{aligned}
		\end{equation}
		\item{} If $p<\infty$, then $\dot{B}^{d/p}_{p,1}$ is continuously embedded in the set of continuous functions decaying to 0 at infinity;
		\item{} The following real interpolation property is satisfied for $1\leq p\leq\infty$, $s_{1}<s_{2}$ and $\theta\in(0,1)$:
		\begin{equation}
			\begin{aligned}
				&\|u\|_{\dot{B}^{\theta s_{1}+(1-\theta)s_{2}}_{p,1}}\lesssim \frac{1}{\theta(1-\theta)(s_{2}-s_{1})}\|u\|_{\dot{B}^{ 
s_{1}}_{p,\infty}}^{\theta}\|u\|_{\dot{B}^{s_{2}}_{p,\infty}}^{1-\theta}.\label{inter}
			\end{aligned}
		\end{equation}
		\item{}
		Let $\Lambda^{\sigma}$ be defined by $\Lambda^{\sigma}=(-\Delta)^{\frac{\sigma}{2}}u:=\mathcal{F}^{-}\big{(} |\xi|^{\sigma}\mathcal{F}(u) \big{)}$ for $\sigma\in \mathbb{R}$ and 
$u\in\dot{S}^{'}_{h}$, then $\Lambda^{\sigma}$ is an isomorphism from $\dot{B}^{s}_{p,r}$ to $\dot{B}^{s-\sigma}_{p,r}$;
	\end{itemize}
\end{lemma}

The following Morse-type product estimates in Besov spaces play a fundamental role in the analysis of nonlinear terms:
\begin{lemma}\label{Lemma5-5}
	The following statements hold:
	\begin{itemize}
		\item{} Let $s>0$ and $1\leq p,r\leq \infty$. Then $\dot{B}^{s}_{p,r}\cap L^{\infty}$ is a algebra and
		\begin{equation}\label{uv1}
			\|uv\|_{\dot{B}^{s}_{p,r}}\lesssim \|u\|_{L^{\infty}}\|v\|_{\dot{B}^{s}_{p,r}}+ \|v\|_{L^{\infty}}\|u\|_{\dot{B}^{s}_{p,r}}.
		\end{equation}
		\item{}
		Let $s_{1}$, $s_{2}$ and $p$ satisfy $2\leq p\leq \infty$, $-\frac{d}{p}<s_{1},s_{2}\leq \frac{d}{p}$ and $s_{1}+s_{2}>0$. Then we have
		\begin{equation}\label{uv2}
			\|uv\|_{\dot{B}^{s_{1}+s_{2}-\frac{d}{p}}_{p,1}}\lesssim \|u\|_{\dot{B}^{s_{1}}_{p,1}}\|v\|_{\dot{B}^{s_{2}}_{p,1}}.
		\end{equation}
		\item{} Assume that $s_{1}$, $s_{2}$ and $p$ satisfy $2\leq p\leq \infty$, $-\frac{d}{p}<s_{1}\leq \frac{d}{p}$, $-\frac{d}{p}\leq s_{2}<\frac{d}{p}$ and $s_{1}+s_{2}\geq0$. Then 
it holds
		\begin{equation}\label{uv3}
			\|uv\|_{\dot{B}^{s_{1}+s_{2}-\frac{d}{p}}_{p,\infty}}\lesssim \|u\|_{\dot{B}^{s_{1}}_{p,1}}\|v\|_{\dot{B}^{s_{2}}_{p,\infty}}.
		\end{equation}
        \item {} Let $s_{1}$, $s_{2}$, $p_1$ and $p_2$ be such that
$$
s_1+s_2>0,\quad s_1\leq \frac{d}{p_1},\quad s_2\leq \frac{d}{p_2},\quad s_1\geq s_2,\quad \frac{1}{p_1}+\frac{1}{p_2}\leq 1.
$$
Then we have
\begin{align}\label{uv4}
\|uv\|_{\dot{B}^{s_{2}}_{q,1}}\lesssim \|u\|_{\dot{B}^{s_{1}}_{p_1,1}}\|v\|_{\dot{B}^{s_{2}}_{p_2,1}}\quad \text{with}\quad \frac{1}{q}=\frac{1}{p_1}+\frac{1}{p_2}-\frac{s_1}{d}.
\end{align}
	\end{itemize}

\end{lemma}

The following commutator estimates will be used to control some nonlinearities in high frequencies.
\begin{lemma}\label{lemcommutator}
	Let $p,p_1\in [1,\infty ]$  and $p'=\frac{p}{p-1}$. Denote by $[A,B]=AB-BA.$ the commutator. Then the following two statements hold:   
    \begin{itemize}
        \item For $-\min\{\frac{d}{p_1},\frac{d}{p'}\}< s\leq \min\{\frac{d}{p},\frac{d}{p_1}\}+1$, it holds 
        \begin{align}\label{Com:1}
		\sum_{j\in \mathbb{Z}}2^{js}\|[v,\dot{\Delta}_j]\partial_i u\|_{L^p}\lesssim\|\nabla v\|_{\dot{B}^{\frac{d}{p_1}}_{p_1,1}}\|u\|_{\dot{B}^{s}_{p,1}}, \ i=1,2,\cdots,d,
	\end{align}
        \item For $-\min\{\frac{d}{p_1},\frac{d}{p'}\}\leq  s\leq \min\{\frac{d}{p},\frac{d}{p_1}\}+1$, it holds 
        \begin{align}\label{Com:2}
		\sup_{j\in \mathbb{Z}}2^{js}\|[v,\dot{\Delta}_j]\partial_i u\|_{L^p}\lesssim\|\nabla v\|_{\dot{B}^{\frac{d}{p_1}}_{p_1,1}}\|u\|_{\dot{B}^{s}_{p,\infty}}, \ i=1,2,\cdots,d.
	\end{align}

      \item For $ s>\min\{\frac{d}{p},\frac{d}{p_1}\}+1$, it holds 
        \begin{align}\label{Com:3}
		\sum_{j\in \mathbb{Z}} 2^{js}\|[v,\dot{\Delta}_j]\partial_i u\|_{L^p}\lesssim\|\nabla v\|_{\dot{B}^{\frac{d}{p_1}}_{p_1,1}}\|u\|_{\dot{B}^{s}_{p,1}}+\|\nabla u\|_{L^{\infty}}\|v\|_{\dot{B}^{s}_{p,1}}, \ i=1,2,\cdots,d.
	\end{align}

        \end{itemize}
\end{lemma}

\begin{proof}
The first estimate \eqref{Com:1} is classical (see \cite{danchin5}). The second estimate \eqref{Com:1} follows by the same argument as in \cite{danchin5}, with a minor modification to handle the endpoint $r=\infty$; we therefore omit the details. To prove \eqref{Com:3}, we recall Bony's paraproduct decomposition (see \cite{bahouri1}):
\begin{equation}\nonumber
\begin{aligned}
fg=\dot{T}_{f}g+\dot{R}[f,g]+\dot{T}_{g}f\quad \mbox{with}\quad
\dot{T}_{f}g:=  \sum_{j'\in \mathbb{Z}} \dot S_{j'-1}f \dot{\Delta}_{j'} g\quad\mbox{and}\quad
\dot{R}[f,g]:=  \sum_{|j'-j''|\leq 1} \dot\Delta_{j''} f \dot{\Delta}_{j'} g.
\end{aligned}
\end{equation}
Accordingly, $[v,\dot{\Delta}_{j}]\partial_i u$ can be decomposed as
\begin{equation}\nonumber
\begin{aligned}
\relax[v,\dot{\Delta}_{j}] \partial_i u&=\sum_{|j'-j|\leq 4}[\dot{S}_{j'-1}v,\dot{\Delta}_{j}]  \dot{\Delta}_{j'} \partial_i u-\dot{\Delta}_{j} \dot{T}_{ \partial_i u}v-\dot{\Delta}_{j} \dot{R}(v, \partial_i u)\\
&\quad+({\rm Id}-\dot{S}_{j-1}) v  {\Delta}_{j} \partial_i u+\sum_{|j-j'|\leq1} (\dot{S}_{j-1}v-\dot{S}_{j'-1}v)
\dot{\Delta}_{j}\dot{\Delta}_{j'}\partial_i u.
\end{aligned}
\end{equation}

First, the classical commutator estimate
(see \cite{bahouri1}[Proposition 2.97]) leads to
\begin{equation}\nonumber
\begin{aligned}
\sum_{j\in\mathbb{Z}}2^{sj} \sum_{|j'-j|\leq 4}\|[\dot{S}_{j'-1}v,\dot{\Delta}_{j}]  \dot{\Delta}_{j'} \partial_i u\|_{L^p}&\lesssim \sum_{j\in\mathbb{Z}}2^{sj} \sum_{|j'-j|\leq 4}\|\nabla \dot{S}_{j'-1} w\|_{L^{\infty}} 2^{-j}\|\dot{\Delta}_{j'} \partial_i  u\|_{L^p} \\
&\lesssim \|\nabla v\|_{L^{\infty}} \|u\|_{\dot{B}^{s}_{p,1}}.
\end{aligned}
\end{equation}
Employing the estimate for $\dot{T}$ and $\dot{R}$ (see \cite{bahouri1}[Theorems 2.82 and 2.85]), we have
\begin{equation}\nonumber
\begin{aligned}
\|\dot{T}_{ \partial_i u}v\|_{\dot{B}^{s}_{p,1}}&\lesssim \|\partial_i u\|_{L^{\infty}} \|v\|_{\dot{B}^{s}_{p,1}}.
\end{aligned}
\end{equation}
and
\begin{equation}\nonumber
\begin{aligned}
\|\dot{R}(v, \partial_i u)\|_{\dot{B}^{s}_{p,1}}&\lesssim \|v\|_{\dot{B}^{1}_{\infty,1}}\|\partial_i u\|_{\dot{B}^{s-1}_{p,1}}\lesssim \|\nabla v\|_{\dot{B}^{p_1}_{p_1,1}} \|u\|_{\dot{B}^{s}_{p,1}}.
\end{aligned}
\end{equation}

Next, by Bernstein’s and Young’s inequalities,
\begin{equation}\nonumber
\begin{aligned}
\sum_{j\in\mathbb{Z}}2^{sj}  \|({\rm Id}-\dot{S}_{j-1}) v {\Delta}_{j}\partial_i u\|_{L^p}&\lesssim \sum_{j\in\mathbb{Z}} \sum_{j'\geq j}  2^{sj} \|\dot{\Delta}_{j} u\|_{L^{p}}  2^{j'}\|\dot{\Delta}_{j'}v\|_{L^\infty} 2^{(j-j')}\\
& \lesssim   \|v\|_{\dot{B}^{1}_{\infty,1}} \|u\|_{\dot{B}^{s}_{p,1}}\lesssim \|\nabla v\|_{\dot{B}^{\frac{d}{p_1}}_{p_1,1}}\|u\|_{\dot{B}^{s}_{p,1}}.
\end{aligned}
\end{equation}

Finally, we also have
\begin{equation}\nonumber
\begin{aligned}
\sum_{j\in\mathbb{Z}} 2^{js}\sum_{|j-j'|\leq1} \|(\dot{S}_{j'-1}v-\dot{S}_{j-1}v)\dot{\Delta}_{j}\dot{\Delta}_{j'} \partial_i u\|_{L^p} 
& \lesssim\sum_{j\in\mathbb{Z}} 2^{j}  (\|\dot{\Delta}_{j- 1}v\|_{L^{\infty}}+\|\dot{\Delta}_{j- 2}v\|_{L^{\infty}}) 2^{j(s-1)} \| \partial_i u\|_{L^p}\\
&\lesssim \|v\|_{\dot{B}^{1}_{\infty,\infty}} \|u\|_{\dot{B}^{s}_{p,1}}\lesssim  \|\nabla v\|_{\dot{B}^{\frac{d}{p_1}}_{p,1}}\|u\|_{\dot{B}^{s}_{p,1}}.
\end{aligned}
\end{equation}
\end{proof}

We state the following result concerning the continuity of composition functions.
\begin{lemma}\label{Lemma5-7}
	Let $G:I\to \mathbb{R}$ be a  smooth function satisfying $G(0)=0$.  Let $1\leq p\leq \infty$. There exists a constant $C>0$ depending only on $G'$, $s $ and $d$ such that the following assertions hold{\rm:}
    
    \begin{itemize}
        \item For $s>0$ and $g\in \dot{B}^{s}_{2,1}\cap L^{\infty}$, it holds that $G(g)\in 
\dot{B}^{s}_{p,r}\cap L^{\infty}$ and
	\begin{align}
		\|G(g)\|_{\dot{B}^{s}_{p,r}}\leq C(1+\|g\|_{L^{\infty}})^{s+1}\|g\|_{\dot{B}^{s}_{p,r}}.\label{F0}
	\end{align}

        \item  For $-\frac{d}{2}<s\leq 0$, if  $g\in \dot{B}^{s}_{p,r}\cap \dot{B}^{\frac{d}{p}}_{p,1}$, then we have $G(g)\in 
\dot{B}^{s}_{p,r}\cap \dot{B}^{\frac{d}{p}}_{p,1}$ and
	\begin{align}
		\|G(g)\|_{\dot{B}^{s}_{p,r}}\leq C\Big(1+\|g\|_{\dot{B}^{\frac{d}{p}}_{p,1}}\Big)^{\frac{d}{p}+1}\|g\|_{\dot{B}^{s}_{p,r}}.\label{F1}
	\end{align}

        \item For $-\frac{d}{2}<s\leq \frac{d}{2}$, if $g_{1}, g_{2}\in \dot{B}^{s}_{p,1}\cap \dot{B}^{\frac{d}{p}_{p,1}}$, then we have
	\begin{align}
		&\|G(g_{1})-G(g_{2})\|_{\dot{B}^{s}_{p,1}}\leq C\Big(1+\|(g_{1},g_{2})\|_{\dot{B}^{\frac{d}{p}}_{p,1}}\Big)^{\frac{d}{p}+1}\|g_{1}-g_{2}\|_{\dot{B}^{s}_{p,1}},\quad ~s\in 
(-\frac{d}{p},\frac{d}{p}].\label{F2}
        	\end{align}
    \end{itemize}
\end{lemma}

\begin{proof}
The first estimate \eqref{F0} is well-known; e.g., cf. \cite{bahouri1}. By 
$$
G(g)=(G'(0)+G^*(g))g\quad\text{with}\quad G^*(g):=\int_0^1 \big(G'(\theta g)-G'(0)\big)  {\rm d}\theta,
$$
the second estimate \eqref{F1} can be shown directly from \eqref{uv2} ($s_1=\frac{d}{p}, s_2=s$), \eqref{F0} and the embedding $\dot{B}^{\frac{d}{p}}_{p,1}\hookrightarrow L^{\infty}$ as in \cite{danchin5}. To show \eqref{F2}, one notes the following formula
$$
G(g_1)-G(g_2)=\big(G'(0)+\widetilde{G}^*(g_1,g_2)\big)(g_1-g_2)\quad\text{with}\quad \widetilde{G}^*(g_1,g_2):=\int_0^1 \big(G'(g_1+\theta(g_1-g_2))-G'(0)\big) {\rm d}\theta
$$
and applies \eqref{uv2} and \eqref{F0} again.
\end{proof}

By similar proofs in Lemma \ref{Lemma5-7}, using Taylor's formula for multicomponent functions (or see \cite{runst1}[Pages 387-388] and \cite{CB3}), one has the following lemma. 

\begin{lemma}\label{lemma64}
	Let $m\in \mathbb{N}$, $1\leq p,r\leq \infty$, and $G\in C^{\infty}(\mathbb{R}^{m})$ satisfy $G(0,...,0)=0$. Then there exists a constant $C_{f}>0$ depending on $F$, $s$, $m$ and $d$ such that
\begin{itemize}
\item  For $s>0$ and any $f_{i}\in\dot{B}_{p,r}^{s}\cap L^{\infty}$ 
$(i=1,...,m)$, we have $G(f_{1},...,f_{m})\in \dot{B}^{s}_{p,r}\cap L^{\infty}$ and
	\begin{equation}
		\begin{aligned}
			\|G(f_{1},...,f_{m})\|_{\dot{B}^{s}_{p,r} }\leq C\Big(1+\sum_{i=1}^{m}\|f_{i}\|_{L^{\infty}}\Big)^{s+1}\sum_{i=1}^{m}\|f_{i}\|_{\dot{B}^{s}_{p,r} }.\label{F1:m}
		\end{aligned}
	\end{equation}

\item  For $-\frac{d}{2}<s\leq 0$ and any $f_{i}\in\dot{B}_{p,r}^{s}\cap \dot{B}^{\frac{d}{p}}_{p,1}$ 
$(i=1,...,m)$, we have $G(f_{1},...,f_{m})\in \dot{B}^{s}_{p,r}\cap \dot{B}^{\frac{d}{p}}_{p,1}$ and
	\begin{equation}
		\begin{aligned}
		\|G(f_{1},...,f_{m})\|_{\dot{B}^{s}_{p,r} }\leq C\Big(1+\sum_{i=1}^{m}\|f_{i}\|_{\dot{B}^{\frac{d}{p}}_{p,1}}\Big)^{\frac{d}{p}+1}\sum_{i=1}^{m}\|f_{i}\|_{\dot{B}^{s}_{p,r} }.\label{F1:m00}
		\end{aligned}
	\end{equation}

\item For $-\frac{d}{2}<s\leq \frac{d}{2}$ and any $f^{1}_{i}, f^{2}_{i}\in\dot{B}_{p,r}^{s}\cap \dot{B}_{p,r}^{\frac{d}{p}}$ $(i=1,...,m)$, we have
	\begin{equation}\label{F3:m}
		\begin{aligned}
			&\|G(f^{1}_{1},...,f^{1}_{m})-G(f^{1}_{1},...,f^{1}_{m})\|_{\dot{B}^{s}_{p,r} }\leq C\Big(1+\sum_{i=1}^{m}\|(f^1_{i},f_{i}^2)\|_{\dot{B}^{\frac{d}{p}}_{p,1}} 
\Big)^{\frac{d}{2}+1}\sum_{i=1}^{m}\|f_{i}^{1}-f_{i}^{2}\|_{\dot{B}^{s}_{p,r}}.
		\end{aligned}
	\end{equation}
\end{itemize}
\end{lemma}

\section[Appendix~\thesection: Dispersive and Strichartz estimates]
        {Dispersive and Strichartz estimates}\label{appendixB}

We need the following proposition concerning the dispersive estimates with the parameter $\kappa$. Note that the case $\kappa=1$ corresponds to the classical Strichartz estimates for the Gross-Pitaevskii equation \cite{GNT206,GNT209}. The proof of  Proposition \ref{propA.2} with $\kappa>0$ can be proved similarly as in \cite{CV-Song-2024} using the stationary phase method.

\begin{proposition} \label{propA.2}
Let $\kappa>0$ and
$$
H^{\kappa}:= -r_1 \Delta -i\sqrt{-\Delta(r_2 -r_3\kappa \Delta)}
$$
with fixed constants $r_1, r_2,r_3>0$. Let $p\in[2,\infty]$ and 
$1=\;\frac1p + \frac1{p'}$. 
For any distribution $f$, there exists a constant $C>0$ independent of $t$ and $\kappa$ such that for any $f\in L^{p'}$,
\begin{equation} 
\bigl\|e^{\, i\,H^{\kappa}\,t} f \bigr\|_{L^{p}}
\;\leq C\;
\bigl(\sqrt{\kappa}\,t\bigr)^{-(\frac d2 - \frac dp)}\,
e^{-2^{2j}t}\,
\|f \|_{L^{p'}},
\end{equation}
for all $t>0$.
\end{proposition}

Next, we prove Strichartz estimates in Besov spaces for general {\emph{vector–matrix}} forms based on Proposition \ref{propA.2},  which will be used in the proof of the dispersive estimates.
\begin{lemma}\label{LemA.2} For $\kappa>0$ and $n\geq1$, the operator $\mathcal{H}^{\kappa}$ is defined by
\[
\mathcal{H}^\kappa =
\left(
\begin{matrix}
\mathcal{H}_1^\kappa & 0 & \cdots & 0 \\
0 & \mathcal{H}_2^\kappa & \cdots & 0 \\
\vdots & \vdots & \ddots & \vdots \\
0 & 0 & \cdots & \mathcal{H}_n^\kappa
\end{matrix}
\right)
\]
with $\mathcal{H}_i^\kappa=- r_{1,i}\Delta - i\sqrt{-\Delta\big(r_{2,i} - r_{3,i}\,\kappa \Delta\big)}$. Here $r_{1,i}, r_{2,i}, r_{3,i}>0$ are given constants.
Set $q,r\in[1,\infty]$, $p\in[2,\infty]$ and $s\in\mathbb{R}$. 
Let $(r,p)$ satisfy the condition
\begin{equation} 
\begin{cases}
\frac{d}{2}  - \frac dp \leq\;\frac2r \leq\;\frac dp -\frac d2 +2,
& d\geq3,\\[1ex]
1  - \frac{2}{p} \leq\;\frac{2}{r}<\;\frac{2}{p}+1,
& d=2.
\end{cases}
\end{equation}
For any $k$ satisfying $k =\frac2r + \frac dp - \frac d2$,
there exists a constant $C>0$ independent of $t$ and $\kappa$ such that for any $n$-vector–valued functions $\vec{u}\in \dot{B}^{s}_{2,q}$ and $\vec{f}\in L^1_t(\dot{B}^{s}_{2,r})$, 
\begin{align}\label{A.8}
\bigl\|e^{\,i \mathcal{H}^{\kappa}\,t} \vec{u}\bigr\|_{\widetilde L^r_t(\dot B^{\,s+k}_{p,q})}
&\leq C\;
\kappa^{\tfrac14\,(k-\tfrac2r)}\,
\|\vec{u}\|_{\dot B^s_{2,q}},
\end{align}
and
\begin{align}\label{A.9}
\left\|\int_{0}^{t}e^{\, i\mathcal{H}^{\kappa}\,(t-s)}\, \vec{f}(s)\,\mathrm{d}s\right\|_{\widetilde L^r_t(\dot B^{\,s+k}_{p,q})}
&\leq C\;
\kappa^{\tfrac14\,(k-\tfrac2r)}\,
\|\vec{f}\|_{\widetilde L^1_t(\dot B^s_{2,q})}.
\end{align}
\end{lemma}
\begin{proof}

Choosing $\vec{u}=(u_1,u_2,\cdots,u_n)^{\top}$,  we get
 \begin{equation}
     \begin{aligned}
        e^{\,i  \mathcal{H}^{\kappa}\,t} \vec{u}=(
e^{\,i  \mathcal{H}_1^{\kappa}\,t}u_1,
 e^{\,i  \mathcal{H}_2^{\kappa}\,t}u_2, 
\cdots ,   
  e^{\,i  \mathcal{H}_n^{\kappa}\,t}u_n)^{\top}.
 \end{aligned}
 \end{equation}
When $p=2$, one uses the optimal smoothing effect of heat kernels to derive 
 \begin{equation}
     \begin{aligned}
      \| e^{\,i  \mathcal{H}^{\kappa}\,t} \vec{u} \|_{\widetilde{L}^{r}(\dot{B}_{2,q}^{s+\frac2r})}  =\sum_{l=1}^{n} \| e^{\,i  \mathcal{H}_l^{\kappa}\,t} u_l \|_{\widetilde{L}^{r}(\dot{B}_{2,q}^{s+\frac2r})} \leq C\sum_{l=1}^{n}\|u_l\|_{\dot{B}_{2,q}^{s}}=C\|\vec{u}\|_{\dot{B}_{2,q}^{s}}
     \end{aligned}
 \end{equation} 
and 
 \begin{equation}
     \begin{aligned}
      \left\|\int_{0}^{t}e^{\, i\mathcal{H}^{\kappa}\,(t-s)}\, \vec{f}(s)\,\mathrm{d}s\right\|_{\widetilde{L}^{r}(\dot{B}_{2,q}^{s+\frac2r})}  &=\sum_{l=1}^{n} \left\|\int_{0}^{t}e^{\, i\mathcal{H}_l^{\kappa}\,(t-s)}\, f_l(s)\,\mathrm{d}s\right\|_{\widetilde{L}^{r}(\dot{B}_{2,q}^{s+\frac2r})} \\
      &\leq C\sum_{l=1}^{n}\|f_l(s)\|_{\widetilde{L}^{1}(\dot{B}_{2,q}^{s})}=C\|\vec{f}\|_{\widetilde{L}^{1}(\dot{B}_{2,q}^{s})}.
     \end{aligned}
 \end{equation} 
When $p>2,$  we consider the case $d\geq 3$. Assume that $\frac2r=\frac{d}{2}-\frac{d}{p}$, by using the classical $TT^*$ method, we have, for all $l=1,2,\cdots,n,$
\begin{equation}
    \begin{aligned}
       \left\|\int_{0}^{t}e^{\, i\mathcal{H}_l^{\kappa}\,(t-s)}\, \dot{\Delta}_jf_l(s)\,\mathrm{d}s\right\|_{L^r(L^p)} \leq C\kappa^{-\frac{1}{2r}}\|\dot{\Delta}_jf_l\|_{L^{r'}(L^{p'})},
    \end{aligned}
\end{equation}
and 
\begin{equation}
    \begin{aligned}
        \|e^{\, i\mathcal{H}_l^{\kappa}\,t}\dot{\Delta}_ju_l\|_{L^r(L^p)}\leq C\kappa^{-\frac{1}{2r}} \|\dot{\Delta}_j u_l\|_{L^2}.
    \end{aligned}
\end{equation}
The case $ k =\frac2r + \frac dp - \frac d2>0$ will be obtained by interpolation between $p=2$ and $p>2$, that is 
\begin{equation}
    \begin{aligned}
       \|e^{\, i\mathcal{H}_l^{\kappa}\,t} u_l\|_{\widetilde L^r_t(\dot B^{\,s+k}_{p,q})}\leq C\kappa^{-\frac{\theta}{2r_1}} \|u_l\|_{\dot B^s_{2,q}} \leq C\;
\kappa^{\tfrac14\,(k-\tfrac2r)}\,
\|u_l\|_{\dot B^s_{2,q}}, \quad\mbox{for}\quad l=1,2,\cdots,n,
    \end{aligned}
\end{equation}
where $\theta=1-\frac{kr_1}{2}$.  Similarly, we can get \begin{equation}
    \begin{aligned}
    \left\|\int_{0}^{t}e^{\, i\mathcal{H}_l^{\kappa}\,(t-s)}\,  f_l(s)\,\mathrm{d}s\right\|_{\widetilde L^r_t(\dot B^{\,s+k}_{p,q})} \leq C\;
\kappa^{\tfrac14\,(k-\tfrac2r)}\,
\|u_l\|_{\dot B^s_{2,q}}, \quad\mbox{for}\quad l=1,2,\cdots,n.
  \end{aligned}
\end{equation}
Then, putting it together,  \eqref{A.8} and \eqref{A.9} hold.

For the  case $d=2$, the argument mirrors the one for $d\geq3$, except that the critical admissible triple $(r,p,k)=(2,\infty,0)$ is not allowed.
\end{proof}

\section[Appendix~\thesection: Decay and smoothness for the incompressible Naver-Stokes equations]{Decay and smoothness for the incompressible Naver-Stokes equations}\label{appendixC}


We present a direct weighted energy method to obtain the time-decay estimates of any-order derivatives for the incompressible Navier-Stokes equations, provided that the global solution exists in critical spaces without smallness. This method is different from the classical work of Oliver and Titi \cite{OT} based on analyticity.

\begin{proposition}\label{PropA.3}
 Suppose that $v^\pm$ is a global solution of the systems \eqref{INS} and satisfies
      \begin{align}\label{es:in0}
     \|v^\pm\|_{\widetilde{L}^{\infty}(\mathbb{R}_+;\dot{B}^{\frac{d}{2}-1}_{2,1})}+\|v^\pm\|_{L^1(\mathbb{R}_+;\dot{B}^{\frac{d}{2}+1}_{2,1})}<\infty.
     \end{align}
 Then, if $v_0^\pm\in\dot{B}^{\sigma_1}_{2,\infty} $ for $-\frac d2 \leq \sigma_1< \frac d2 -1$ and $M>\frac{1}{2}(s-\sigma_1)$, it holds that 
 \begin{equation}\label{A.6}
        \begin{aligned}
          \sup_{\tau\in[0,t]}\|\tau^M v^\pm\|_{\dot{B}^{s}_{2,1}}+\int_0^t \|\tau^M v^\pm\|_{\dot{B}^{s+2}_{2,1}}{\rm d}\tau  \lesssim t^{M-\frac{1}{2}(s-\sigma_1)}\|v^\pm\|_{L^\infty_t(\dot{B}^{\sigma_1}_{2,\infty})},\quad t>0,
        \end{aligned}
    \end{equation}
where $\mathcal{Y}_{t}^{M,s}$ is defined by \eqref{definition-Y}. 

Consequently, we have 
\begin{equation}
\big\|D^\alpha v^\pm (t)\big\|_{L^2}
\;\leq C\;(1+t)^{-\frac{\lvert\alpha\rvert}{2}+\frac{\sigma_1}{2}}
\quad\text{for all }~~\lvert\alpha\rvert>\sigma_1.
\end{equation}
\end{proposition}
\begin{proof}
Arguing as in Lemma \ref{Lem3.1}, we obtain 
    \begin{align}\label{A.1111}
\|v^\pm\|_{L_t^\infty(\dot{B}^{\sigma_1}_{2,\infty})}\lesssim {\rm exp}\Big\{\int_0^t \|v^\pm\|_{\dot{B}^{\frac{d}{2}+1}_{2,1}}{\rm d}\tau\Big\}\|v^\pm_0\|_{ \dot{B}^{\sigma_1}_{2,\infty}}.
    \end{align} 
By direct computations for  \eqref{INS}, there exists a uniform constant $c>0$ such that
\begin{equation}\label{ineq-time-v-1}
    \begin{aligned}
        \frac{1}{2}\frac{{\rm d}}{{\rm d}t}\|  v_j^\pm\|_{L^2}^2+c 2^{2j}\| v_j^\pm\|_{L^2}\leq \|\dot{\Delta}_j(v^\pm\cdot\nabla v^\pm)\|_{L^2}\| v_j^\pm\|_{L^2}.
    \end{aligned}
\end{equation}
Multiplying \eqref{ineq-time-v-1} by $t^{M}$ yields
\begin{equation}\label{ineq-time-v-2}
    \begin{aligned}
      &  \frac{1}{2}\frac{{\rm d}}{{\rm d}t}\|t^M v_j^\pm\|_{L^2}^2+ c 2^{2j}\|t^M v_j^\pm\|_{L^2}^2\leq\Big( M t^{M-1} \| v_j^\pm\|_{L^2}+\|t^M\dot{\Delta}_j(v^\pm\cdot\nabla v^\pm)\|_{L^2}\Big) \| v_j^\pm\|_{L^2}.
    \end{aligned}
\end{equation}
Then, integrating \eqref{ineq-time-v-2} from $0$ to $t$, we obtain 
\begin{equation}\label{ineq-time-v-3}
    \begin{aligned}
        \|t^M v_j^\pm\|_{L^2}&+c 2^{2j}\int_0^t\|\tau^M v_j^\pm\|_{L^2}{\rm d}\tau\leq M\int_0^t \tau^{M-1} \| v_j^\pm\|_{L^2}{\rm d}\tau+\int_0^t\|\tau ^M\dot{\Delta}_j(v^\pm\cdot\nabla v^\pm)\|_{L^2}{\rm d}\tau.
    \end{aligned}
\end{equation}
Next, multiply \eqref{ineq-time-v-3} by $2^{js}$ and sum over $j\in \mathbb{Z}$ to get
\begin{equation}\label{ineq-time-v-4}
    \begin{aligned}
       t^{M }\|v^\pm\|_{\dot{B}^{s}_{2,1}} &+c\int_0^t \|\tau^M v^\pm\|_{\dot{B}^{s+2}_{2,1}}{\rm d}\tau\leq  M\int_0^t \tau^{M-1}  \| v^\pm \|_{\dot{B}^{s}_{2,1}}{\rm d}\tau+\int_0^t\|\tau ^M v^\pm\cdot\nabla v^\pm\|_{\dot{B}^{s}_{2,1}}{\rm d}\tau.
    \end{aligned}
\end{equation}
Let $\eta_0\in(0,1)$ be such that $\sigma_1(1-\eta_0)+(s+2)\eta_0=s$. By \eqref{inter} and Young's inequality, the right-hand side of \eqref{ineq-time-v-4} is estimated as
\begin{equation} \label{A.12}
    \begin{aligned}
         M\int_0^t \tau^{M-1}  \| v^\pm \|_{\dot{B}^{s}_{2,1}}{\rm d}\tau
          \leq& C  \int_0^t \tau^{M-1}\| v^\pm \|^{1-\eta_0}_{\dot{B}^{\sigma_1}_{2,\infty}}\| v^\pm \|^{\eta_0}_{\dot{B}^{s+2}_{2,1}}{\rm d}\tau\\
         \leq& C \left( \int_0^t  \tau^{M-\frac{1}{1-\eta_0}}\| v^\pm \|_{\dot{B}^{\sigma_1}_{2,\infty}} {\rm d}\tau\right)^{1-\eta_0}\| t^{M} v^\pm \|^{\eta_0}_{L^1_t(\dot{B}^{s+2}_{2,1})}\\
         \leq& \frac{c}{4} \|\tau^M  v^\pm \| _{L^1_t(\dot{B}^{s+2}_{2,1})}+Ct^{M-\frac{1}{2}(s-\sigma_1)}\| v^\pm \|_{L_t^\infty(\dot{B}^{\sigma_1}_{2,\infty})}.
    \end{aligned} 
\end{equation}

We first derive decay in the low-regularity range $\sigma_1<s\leq \frac{d}{2}$. The nonlinear terms can be directly handled by the product law \eqref{uv2} ($p=2, s_1=\frac{d}{2}, s_2=s$):  
\begin{equation}\label{A.13}
    \begin{aligned}
     \int_0^t\|\tau ^M v^\pm\cdot\nabla v^\pm\|_{\dot{B}^{s}_{2,1} }{\rm d}\tau&\lesssim  \int_0^t \|\tau^M v^\pm\|_{\dot{B}^{s}_{2,1}}  \| v^\pm\|_{\dot{B}^{\frac{d}{2}+1}_{2,1}}{\rm d}\tau.
    \end{aligned}
\end{equation} 
Combining \eqref{ineq-time-v-4} - \eqref{A.13} gives
\begin{equation*}
    \begin{aligned}
       t^{M }\|v^\pm\|_{\dot{B}^{s}_{2,1}} &+\int_0^t \|\tau^M v^\pm\|_{\dot{B}^{s+2}_{2,1}}{\rm d}\tau\lesssim t^{M-\frac{1}{2}(s-\sigma_1)}\| v^\pm \|_{L^\infty(\dot{B}^{\sigma_1}_{2,\infty})}+\int_0^t \|\tau^M v^\pm\|_{\dot{B}^{s}_{2,1}}  \| v^\pm\|_{\dot{B}^{\frac{d}{2}+1}_{2,1}}{\rm d}\tau.
    \end{aligned}
\end{equation*}
Applying Gr\"onwall's lemma together with \eqref{es:in0} and \eqref{A.1111} yields \eqref{A.6} for $\sigma_1<s\leq \frac{d}{2}$.

Finally, in the higher-regularity case $s>\frac{d}{2}$, we use the Moser-type product law \eqref{uv1}:   
\begin{equation}\label{A.14}
    \begin{aligned}
     \int_0^t\|\tau ^M v^\pm\cdot\nabla v^\pm\|_{\dot{B}^{s}_{2,1} }{\rm d}\tau&\leq C  \int_0^t \|\tau^M v^\pm\|_{\dot{B}^{s}_{2,1}}  \|\nabla v^\pm\|_{L^{\infty}}{\rm d}\tau+  \int_0^t \|v^\pm\|_{L^{\infty}}\|\tau^M v^\pm\|_{\dot{B}^{s+1}_{2,1}}{\rm d}\tau\\
     &\leq C \int_0^t \|\tau^M v^\pm\|_{\dot{B}^{s}_{2,1}}  \|v^\pm\|_{\dot{B}^{\frac{d}{2}+1}_{2,1}}{\rm d}\tau\\
     &\quad+ C\int_0^t \|v^\pm\|_{\dot{B}^{\frac{d}{2}}_{2,1}}\|\tau^M v^\pm\|_{\dot{B}^{s}_{2,1}}^{\frac{1}{2}} \|\tau^M v^\pm\|_{\dot{B}^{s+2}_{2,1}}^{\frac{1}{2}} {\rm d}\tau\\
     &\leq \frac{c}{4} \|\tau^M  v^\pm \| _{L^1_t(\dot{B}^{s+2}_{2,1})}+C\int_0^t \Big(\|v^\pm\|_{\dot{B}^{\frac{d}{2}+1}_{2,1}}+\|v^\pm\|_{\dot{B}^{\frac{d}{2}}_{2,1}}^2\Big) \|\tau^M v^\pm\|_{\dot{B}^{s}_{2,1}}{\rm d}\tau. 
    \end{aligned}
\end{equation}
Since $v^\pm \in L^2(\mathbb{R}_+;\dot{B}^{\frac{d}{2}}_{2,1})\cap L^1(\mathbb{R}_+;\dot{B}^{\frac{d}{2}+1}_{2,1})$, substituting \eqref{A.12} and \eqref{A.14} into \eqref{ineq-time-v-4} and applying Gr\"onwall's lemma together with \eqref{es:in0} gives \eqref{A.6} for $s>\frac{d}{2}$.
\end{proof}

\noindent \textbf{Acknowledgments} 
The authors would like to thank Professor Huanyao Wen for his helpful discussions and suggestions. L.-Y. Shou is supported by National Natural Science Foundation of China $\#$12301275.  L. Yao's research is partially supported by National Natural Science Foundation of China $\#$12571250, 12171390,  and the Fundamental Research Funds for the Central Universities under Grant: G2025KY05134.   Y. Zhang' research is partially supported by National Natural Science  Foundation of China $\#$12271114, Guangxi Natural Science Foundation $\#$2024GXNSFDA010071, $\#$2019JJG110003, Science and Technology Project of Guangxi $\#$GuikeAD21220114, the Innovation Project of Guangxi  Graduate Education $\#$JGY2023061, Center for Applied Mathematics of Guangxi (Guangxi Normal University) and the Key Laboratory of Mathematical Model and Application (Guangxi Normal  University), Education Department of Guangxi Zhuang Autonomous Region.
	
\vspace{2mm}
	
\noindent \textbf{Conflict of interest.} The authors do not have any possible conflicts of interest.
	
\vspace{2mm}
	
\noindent \textbf{Data availability statement.} Data sharing is not applicable to this article as no data sets were generated or analyzed during the current study.

\vspace{0.5cm}
    


(L.-Y. Shou)\par\nopagebreak
\noindent\textsc{School of Mathematical Sciences and Ministry of Education Key Laboratory of NSLSCS, Nanjing Normal University, Nanjing 210023, P.R. China}

Email address: {\texttt{shoulingyun11@gmail.com}}

\vspace{3ex}

(Jiayan Wu)\par\nopagebreak
\noindent\textsc{School of Mathematics, South China University of Technology, Guangzhou 510641, P.R. China}

Email address: {\texttt{wujiayan@scut.edu.cn}}

\vspace{3ex}

(Lei Yao)\par\nopagebreak
\noindent\textsc{School of Mathematics and Statistics, Northwestern Polytechnical University, Xi'an 710129, P.R. China}

Email address: {\texttt{yaolei1056@hotmail.com}}

\vspace{3ex}

(Yinghui Zhang)\par\nopagebreak
\noindent\textsc{School of Mathematics and Statistics, Guangxi Normal University, Guilin, Guangxi 541004, P.R. China}

Email address: {\texttt{yinghuizhang@mailbox.gxnu.edu.cn}}

\vspace{3ex}


\begin{thebibliography}{99}


\bibliographystyle{plain}
	\parskip=0pt
	\small
\bibitem{AH2017}
C. Audiard, B. Haspot,
Global well-posedness of the Euler--Korteweg system for small irrotational data,
\emph{Comm. Math. Phys.} \textbf{351}(1) (2017), 201--247.

\bibitem{AHM2020}
P. Antonelli, L. Hientzsch, P. Marcati,
On the low Mach number limit for quantum Navier--Stokes equations,
\emph{SIAM J. Math. Anal.} \textbf{52}(6) (2020), 6105--6139.

\bibitem{bahouri1}
H. Bahouri, J.-Y. Chemin, R. Danchin,
\emph{Fourier Analysis and Nonlinear Partial Differential Equations},
Grundlehren der Mathematischen Wissenschaften, Vol. 343, Springer, 2011.

\bibitem{Bearbook}
J. Bear,
\emph{Dynamics of Fluids in Porous Media},
Dover, 1988. (Reprint of the 1972 Elsevier edition.)

\bibitem{Brennenbook}
C. E. Brennen,
\emph{Fundamentals of Multiphase Flow},
Cambridge University Press, 2005.

\bibitem{BDGG2010}
D. Bresch, B. Desjardins, J.-M. Ghidaglia, E. Grenier,
Global weak solutions to a generic two-fluid model,
\emph{Arch. Ration. Mech. Anal.} \textbf{196}(2) (2010), 599--629.

\bibitem{BHL2012}
D. Bresch, X. Huang, J. Li,
Global weak solutions to one-dimensional nonconservative viscous compressible two-phase system,
\emph{Comm. Math. Phys.} \textbf{309}(3) (2012), 737--755.

\bibitem{BDGS2008}
D. Bresch, B. Desjardins, M. Gisclon, R. Sart,
Instability results related to the compressible Korteweg system,
\emph{Ann. Univ. Ferrara} \textbf{54} (2008), 11--36.

\bibitem{BGL2019}
D. Bresch, M. Gisclon, I. Lacroix-Violet,
On Navier--Stokes--Korteweg and Euler--Korteweg systems: applications to quantum fluid models,
\emph{Arch. Ration. Mech. Anal.} \textbf{233} (2019), 975--1025.

\bibitem{cannone1}
M. Cannone,
A generalization of a theorem by Kato on Navier--Stokes equations,
\emph{Rev. Mat. Iberoam.} \textbf{13}(3) (1997), 515--541.

\bibitem{chemin1}
J.-Y. Chemin, N. Lerner,
Flot de champs de vecteurs non lipschitziens et équations de Navier--Stokes,
\emph{J. Differential Equations} \textbf{121} (1995), 314--328.

\bibitem{chemin2}
J.-Y. Chemin,
Théorèmes d'unicité pour le système de Navier--Stokes tridimensionnel,
\emph{J. Anal. Math.} \textbf{77} (1999), 27--50.

\bibitem{CS}
T. Crin-Barat, Q. He, L. Shou,
The hyperbolic--parabolic chemotaxis system for vasculogenesis: global dynamics and relaxation limit toward a Keller--Segel model,
\emph{SIAM J. Math. Anal.} \textbf{55}(5) (2023), 4445--4492.

\bibitem{CB3}
T. Crin-Barat, L. Shou, J. Tan,
Quantitative derivation of a two-phase porous media system from the one-velocity Baer--Nunziato and Kapila systems,
\emph{Nonlinearity} \textbf{37} (2024), 075002.

\bibitem{CWYZ2016}
H. Cui, W. Wang, L. Yao, C. Zhu,
Decay rates for a nonconservative compressible generic two-fluid model,
\emph{SIAM J. Math. Anal.} \textbf{48}(1) (2016), 470--512.

\bibitem{CDX2021}
F. Charve, R. Danchin, J. Xu,
Gevrey and decay for the compressible Navier--Stokes system with capillarity,
\emph{Indiana Univ. Math. J.} \textbf{70} (2021), 1903--1944.

\bibitem{danchin5}
R. Danchin, J. Xu,
Optimal time-decay estimates for the compressible Navier--Stokes equations in the critical $L^{p}$ framework,
\emph{Arch. Ration. Mech. Anal.} \textbf{224}(1) (2017), 53--90.

\bibitem{danchin17Adv}
R. Danchin,
Compressible Navier--Stokes system: large solutions and incompressible limit,
\emph{Adv. Math.} \textbf{320} (2017), 904--925.

\bibitem{danchinhand}
R. Danchin,
Fourier analysis methods for the compressible Navier--Stokes equations,
in \emph{Handbook of Mathematical Analysis in Mechanics of Viscous Fluids},
Y. Giga and A. Novotny (eds.), Springer, 2018.

\bibitem{EvjeWangWang2016}
S. Evje, W. Wang, H. Wen,
Global well-posedness and decay rates of strong solutions to a nonconservative compressible two-fluid model,
\emph{Arch. Ration. Mech. Anal.} \textbf{221}(3) (2016), 1285--1316.

\bibitem{fujita1}
H. Fujita, T. Kato,
On the Navier--Stokes initial value problem I,
\emph{Arch. Ration. Mech. Anal.} \textbf{16} (1964), 269--315.

\bibitem{Guo-2012}
Y. Guo, Y. Wang,
Decay of dissipative equations and negative Sobolev spaces,
\emph{Commun. Partial Differ. Equ.} \textbf{37} (2012), 2165--2208.

\bibitem{GNT206}
S. Gustafson, K. Nakanishi, T. Tsai,
Scattering for the Gross--Pitaevskii equation,
\emph{Math. Res. Lett.} \textbf{13}(2--3) (2006), 273--285.

\bibitem{GNT209}
S. Gustafson, K. Nakanishi, T. Tsai,
Scattering theory for the Gross--Pitaevskii equation in three dimensions,
\emph{Commun. Contemp. Math.} \textbf{11} (2009), 657--707.

\bibitem{Ishiibook}
M. Ishii,
\emph{Thermo-Fluid Dynamic Theory of Two-Phase Flow},
Eyrolles, 1975.

\bibitem{LT1998}
J. Lowengrub, L. Truskinovsky,
Quasi-incompressible Cahn--Hilliard fluids and topological transitions,
\emph{Proc. R. Soc. Lond. A} \textbf{454} (1998), 2617--2654.

\bibitem{LWY2017}
J. Lai, H. Wen, L. Yao,
Vanishing capillarity limit of the nonconservative compressible two-fluid model,
\emph{Discrete Contin. Dyn. Syst. Ser. B} \textbf{22}(4) (2017), 1361--1392.

\bibitem{LWWZ2023}
Y. Li, H. Wang, G. Wu, Y. Zhang,
Global existence and optimal decay rates for a generic nonconservative compressible two-fluid model,
\emph{J. Math. Fluid Mech.} \textbf{25}(4) (2023), Paper No.~77, 35 pp.

\bibitem{LiShou2023}
H. Li, L. Shou,
Global existence and optimal time-decay rates of the compressible Navier--Stokes--Euler system,
\emph{SIAM J. Math. Anal.} \textbf{55}(3) (2023), 1810--1846.

\bibitem{OT}
M. Oliver, E. Titi,
Remarks on the rate of decay of higher order derivatives for solutions to the Navier--Stokes equations in $\mathbb{R}^n$,
\emph{J. Funct. Anal.} \textbf{172} (2000), 1--18.

\bibitem{Prospererttibook}
A. Prosperetti, G. Tryggvason,
\emph{Computational Methods for Multiphase Flow},
Cambridge University Press, 2007.

\bibitem{runst1}
T. Runst, W. Sickel,
\emph{Sobolev Spaces of Fractional Order, Nemytskij Operators, and Nonlinear Partial Differential Equations},
Walter de Gruyter, 1996.

\bibitem{CV-Song-2024}
Z. Song,
Global dynamics of large solutions for the compressible Navier--Stokes--Korteweg equations,
\emph{Calc. Var. Partial Differential Equations} \textbf{63} (2024), 112.

\bibitem{SWYZ}
L.-Y. Shou, J. Wu, L. Yao, Y. Zhang,
The nonconservative compressible two-fluid system with common pressure: global existence and sharp time asymptotics,
arXiv:2502.05419.

\bibitem{TWYZ2025}
Z. Tan, G. Wu, L. Yao, Y. Zhang,
Vanishing capillarity limit of a generic compressible two-fluid model with common pressure,
\emph{J. Differential Equations} \textbf{420} (2025), 223--262.

\bibitem{WYZ2018}
H. Wen, L. Yao, C. Zhu,
Review on mathematical analysis of some two-phase flow models,
\emph{Acta Math. Sci. Ser. B (Engl. Ed.)} \textbf{38}(5) (2018), 1617--1636.

\bibitem{WuYaoZhangin}
G. Wu, L. Yao, Y. Zhang,
On instability of a generic compressible two-fluid model in $\mathbb{R}^3$,
\emph{Nonlinearity} \textbf{36} (2023), 4740--4757.

\bibitem{WYZMathAnn}
G. Wu, L. Yao, Y. Zhang,
Global well-posedness and large-time behavior of classical solutions to a generic compressible two-fluid model,
\emph{Math. Ann.} \textbf{389} (2024), 3379--3415.

\bibitem{XinXu2021}
Z. Xin, J. Xu,
Optimal decay for the compressible Navier--Stokes equations without additional smallness assumptions,
\emph{J. Differential Equations} \textbf{274} (2021), 543--575.






    \end{thebibliography}
\end{document}